 \def\versionno{ fgjs2   --   version 7.0   --   by dj   --  13.07.24  }
\documentclass[english,12pt]{article}

\usepackage{float}
\usepackage{mathtools}
\usepackage{amsmath}
\usepackage{hyperref} 
\usepackage{amsthm}
\usepackage{amssymb}
\usepackage{enumerate}
\usepackage{multicol}
\usepackage{stmaryrd}

\usepackage{hyperref,fancyhdr}
\usepackage[english]{babel} 
\usepackage{mathrsfs,mathtools}
\usepackage{amsfonts}
\usepackage{amssymb}
\usepackage{enumerate}
\usepackage{multicol}
\usepackage{arydshln}
\usepackage{stmaryrd}
\usepackage{graphicx,float}
\usepackage{wrapfig}
\usepackage[export]{adjustbox}

\usepackage{nameref}
\usepackage{bbm}
\usepackage{color,colordvi}
\usepackage[mathscr]{eucal}
\usepackage{caption}
\usepackage{subcaption}
\usepackage{chngcntr}

\usepackage{tikz}
\usepackage{tikz-cd}
\everymath{\displaystyle}
\usepackage{lmodern}
\usepackage{fancyhdr}
\usetikzlibrary{decorations,arrows}
\usetikzlibrary{decorations.markings}
\usetikzlibrary{decorations.pathmorphing}
\usetikzlibrary{matrix,arrows}
\usetikzlibrary{shadings}

\numberwithin{equation}{section}

\makeatletter \@ifundefined{date}{}{\date{}} \makeatother

\theoremstyle{plain}
\newtheorem{thm}{Theorem}[section]
\newtheorem{cor}[thm]{Corollary}
\newtheorem{lem}[thm]{Lemma}
\newtheorem{prop}[thm]{Proposition}

\theoremstyle{definition}

\newtheorem{defi}[thm]{Definition}
\newtheorem{rem}[thm]{Remark}
\newtheorem{exa}[thm]{Example}

\theoremstyle{definition}

\newtheorem*{rep@lemma}{\rep@title}
\newcommand{\newreplemma}[2]{%
\newenvironment{rep#1}[1]{%
 \def\rep@title{#2 \ref{##1}}%
 \begin{rep@lemma}}%
 {\end{rep@lemma}}}
\newreplemma{lem}{Lemma}

\makeatletter
\newcommand*{\relrelbarsep}{.386ex}
\newcommand*{\relrelbar}{%
  \mathrel{%
      \mathpalette\@relrelbar\relrelbarsep}}
\newcommand*{\@relrelbar}[2]{%
  \raise#2\hbox to 0pt{$\m@th#1\relbar$\hss}%
    \lower#2\hbox{$\m@th#1\relbar$}}
\providecommand*{\rightrightarrowsfill@}{%
      \arrowfill@\relrelbar\relrelbar\rightrightarrows}
\providecommand*{\leftleftarrowsfill@}{%
       \arrowfill@\leftleftarrows\relrelbar\relrelbar}
\providecommand*{\xrightrightarrows}[2][]{%
          \ext@arrow 0359\rightrightarrowsfill@{#1}{#2}}
\providecommand*{\xleftleftarrows}[2][]{%
    \ext@arrow 3095\leftleftarrowsfill@{#1}{#2}}
\makeatother

\usepackage{tikz}
\usepackage{tikz-cd}
\everymath{\displaystyle}
\usepackage{lmodern}
\usepackage{fancyhdr}
\usetikzlibrary{decorations,arrows}
\usetikzlibrary{decorations.markings}
\usetikzlibrary{decorations.pathmorphing}
\usetikzlibrary{matrix,arrows}

             \newcommand\Cite[2] {\cite[#1]{#2}}

\catcode`\@=11
\newif\if@fewtab\@fewtabtrue
{\count255=\time\divide\count255 by 60
\xdef\hourmin{\number\count255}
\multiply\count255 by-60\advance\count255 by\time
\xdef\hourmin{\hourmin:\ifnum\count255<10 0\fi\the\count255}}
\def\ps@draft{\let\@mkboth\@gobbletwo
    \def\@oddfoot{\hbox to 7 cm{\tiny \versionno
       \hfil}\hskip -7cm\hfil\rm\thepage \hfil {\tiny\draftdate}}
    \def\@oddhead{}
    \def\@evenhead{}\let\@evenfoot\@oddfoot}
\def\draftdate{\number\month/\number\day/\number\year\ \ \ \hourmin }

\catcode`\@=12

\def\colorL  {red!51!white}
\def\colorM  {red!08!white}
\def\colorMA {red!11!white}
\def\colorMJ {red!21!white}
\def\colorMlab  {red!62!black} % region labels
\def\colorMlabb {red!33!gray}  % backregion labels
\def\colorMd {red!33!white}
\def\colorMl {white}
\def\colorN  {red!03!white}
\def\radblob {0.9pt} % in phantom edge move 
\def\radcir  {0.7}
\def\radvtx  {0.8pt} % standard radius vertex
\def\stdx    {3.5}   % hor width
\def\stdxx   {0.5}   % hor shift
\def\stdy    {1.9}   % ver width
\def\stdyy   {0.05}  % ver variance
\def\horddots {\filldraw[black]
            (bl1) circle (\radblob) (bl2) circle (\radblob) (bl3) circle (\radblob)
            (br1) circle (\radblob) (br2) circle (\radblob) (br3) circle (\radblob) ;
  }
\def\verddots {\filldraw[black]
            (bl1) circle (\radblob) (bl2) circle (\radblob) (bl3) circle (\radblob)
            (br1) circle (\radblob) (br2) circle (\radblob) (br3) circle (\radblob) ;
  }
\def\mercedesgraph{%
  \coordinate (p0) at (180:0.11*\locR) ;
  \coordinate (p2) at (15:0.48*\locR) ;
  \coordinate (p3) at (130:0.55*\locR) ;
  \coordinate (p4) at (240:0.55*\locR) ;
  \draw[line width=0.8*\widthObj,\colorObj]
        (p0) -- (\locaA:\locR) (p0) -- (\locaB:\locR) (p0) -- (\locaC:\locR)
        (0,0) circle (\locR) ;
  \filldraw[\colorObj]
        (\locaA:\locR) circle (\locv) (\locaB:\locR) circle (\locv)
        (\locaC:\locR) circle (\locv) (p0) circle (\locv) ;
  }
\def\tikzpicTwocell{%
  \coordinate (p01) at (0.3*\stdx,-\stdyy) ;
  \coordinate (p02) at (0.7*\stdx,0.5*\stdyy) ;
  \coordinate (p03) at (\stdx,0) ;
  \coordinate (p10) at (0.5*\stdxx,0.5*\stdy) ;
  \coordinate (p13) at (\stdx+0.5*\stdxx,0.5*\stdy) ;
  \coordinate (p20) at (\stdxx,\stdy) ;
  \coordinate (p21) at (0.4*\stdx+\stdxx,\stdy+\stdyy) ;
  \coordinate (p22) at (0.75*\stdx+\stdxx,\stdy-\stdyy) ;
  \coordinate (p23) at (\stdx+\stdxx,\stdy) ;
  \filldraw[draw=\colorL,thick,fill=\colorM,rounded corners] 
            (0,0) [out=90,in=270] to (p10) [out=90,in=270] to (p20) 
	    [out=0,in=180] to (p21) [out=0,in=180] to (p22) [out=0,in=180] to (p23)
	    [out=270,in=90] to (p13) [out=270,in=90] to (p03)
	    [out=180,in=0] to (p02) [out=180,in=0] to (p01) [out=180,in=0] to (0,0);
  }
\def\tikzpicTwocellEE{%
  \tikzpicTwocell
  \coordinate (p11)  at (0.44*\stdx,0.5*\stdy) ;
  \coordinate (p12)  at (0.71*\stdx,0.55*\stdy) ;
  \coordinate (p20a) at (\stdxx+0.09,\stdy) ;
  \coordinate (p20b) at (0.85*\stdxx+0.04,0.85*\stdy) ;
  \coordinate (p23a) at (\stdx+\stdxx-0.10,\stdy) ;
  \coordinate (p23b) at (\stdx+0.87*\stdxx+0.06,0.87*\stdy) ;
  \coordinate (bl1)  at (0.205*\stdx,0.39*\stdy) ;
  \coordinate (bl2)  at (0.20*\stdx,0.49*\stdy) ;
  \coordinate (bl3)  at (0.21*\stdx,0.59*\stdy) ;
  \coordinate (br1)  at (0.91*\stdx,0.37*\stdy) ;
  \coordinate (br2)  at (0.92*\stdx,0.47*\stdy) ;
  \coordinate (br3)  at (0.915*\stdx,0.57*\stdy) ;
  \draw[black,thick]
            (p11) [out=215,in=25] to (0,0)
            (p11) [out=125,in=-30] to (p20a)
            (p11) [out=155,in=-25] to (p20b)
            (p12) [out=-55,in=125] to (\stdx-0.07,0)
            (p12) [out=55,in=200] to (p23a)
            (p12) [out=40,in=215] to (p23b) ;
  \filldraw[black] (p11) circle (\radvtx) (p12) circle (\radvtx) ;
  }
\def\tikzpicTwocellEEfour{%
  \tikzpicTwocell
  \coordinate (p11)  at (0.44*\stdx,0.5*\stdy) ;
  \coordinate (p12)  at (0.71*\stdx,0.55*\stdy) ;
  \coordinate (p20a) at (\stdxx+0.09,\stdy) ;
  \coordinate (p23b) at (\stdx+0.87*\stdxx+0.06,0.87*\stdy) ;
  \coordinate (bl1)  at (0.205*\stdx,0.39*\stdy) ;
  \coordinate (bl2)  at (0.20*\stdx,0.49*\stdy) ;
  \coordinate (bl3)  at (0.21*\stdx,0.59*\stdy) ;
  \coordinate (br1)  at (0.91*\stdx,0.37*\stdy) ;
  \coordinate (br2)  at (0.92*\stdx,0.47*\stdy) ;
  \coordinate (br3)  at (0.915*\stdx,0.57*\stdy) ;
  \draw[black,thick]
            (p11) [out=215,in=25] to (0,0)
            (p11) [out=125,in=-30] to (p20a)
            (p12) [out=-55,in=125] to (\stdx-0.07,0)
            (p12) [out=40,in=215] to (p23b) ;
  \filldraw[black] (p11) circle (\radvtx) (p12) circle (\radvtx) ;
  }
\def\tikzpicTwocellFF{%
  \tikzpicTwocell
  \coordinate (p1m)  at (0.47*\stdx+0.5*\stdxx,0.52*\stdy) ;
  \coordinate (p20a) at (\stdxx+0.09,\stdy) ;
  \coordinate (p20b) at (0.85*\stdxx+0.04,0.85*\stdy) ;
  \coordinate (p23a) at (\stdx+\stdxx-0.10,\stdy) ;
  \coordinate (p23b) at (\stdx+0.87*\stdxx+0.06,0.87*\stdy) ;
  \coordinate (bl1)  at (0.205*\stdx,0.39*\stdy) ;
  \coordinate (bl2)  at (0.20*\stdx,0.49*\stdy) ;
  \coordinate (bl3)  at (0.21*\stdx,0.59*\stdy) ;
  \coordinate (br1)  at (0.91*\stdx,0.37*\stdy) ;
  \coordinate (br2)  at (0.92*\stdx,0.47*\stdy) ;
  \coordinate (br3)  at (0.915*\stdx,0.57*\stdy) ;
  \draw[black,thick]
            (p1m) [out=215,in=25] to (0,0)
            (p1m) [out=125,in=-30] to (p20a)
            (p1m) [out=155,in=-25] to (p20b)
            (p1m) [out=-55,in=125] to (\stdx-0.07,0)
            (p1m) [out=55,in=200] to (p23a)
            (p1m) [out=40,in=215] to (p23b) ;
  \filldraw[black] (p1m) circle (\radvtx) ;
  }
\def\tikzpicTwocellFFfour{%
  \tikzpicTwocell
  \coordinate (p1m)  at (0.47*\stdx+0.5*\stdxx,0.52*\stdy) ;
  \coordinate (p20a) at (\stdxx+0.09,\stdy) ;
  \coordinate (p23b) at (\stdx+0.87*\stdxx+0.06,0.87*\stdy) ;
  \coordinate (bl1)  at (0.205*\stdx,0.39*\stdy) ;
  \coordinate (bl2)  at (0.20*\stdx,0.49*\stdy) ;
  \coordinate (bl3)  at (0.21*\stdx,0.59*\stdy) ;
  \coordinate (br1)  at (0.91*\stdx,0.37*\stdy) ;
  \coordinate (br2)  at (0.92*\stdx,0.47*\stdy) ;
  \coordinate (br3)  at (0.915*\stdx,0.57*\stdy) ;
  \draw[black,thick]
            (p1m) [out=215,in=25] to (0,0)
            (p1m) [out=125,in=-30] to (p20a)
            (p1m) [out=-55,in=125] to (\stdx-0.07,0)
            (p1m) [out=40,in=215] to (p23b) ;
  \filldraw[black] (p1m) circle (\radvtx) ;
  }
\def\tikzpicTwocellGG{%
  \tikzpicTwocell
  \coordinate (p1m)  at (0.47*\stdx+0.5*\stdxx,0.52*\stdy) ;
  \coordinate (p20a) at (\stdxx+0.09,\stdy) ;
  \coordinate (p20b) at (0.82*\stdxx+0.04,0.82*\stdy) ;
  \coordinate (p23a) at (\stdx+\stdxx-0.10,\stdy) ;
  \coordinate (p23b) at (\stdx+0.19*\stdxx-0.02,0.19*\stdy) ;
  \coordinate (bl1)  at (0.53*\stdx,0.92*\stdy) ;
  \coordinate (bl2)  at (0.59*\stdx,0.93*\stdy) ;
  \coordinate (bl3)  at (0.65*\stdx,0.92*\stdy) ;
  \coordinate (br1)  at (0.46*\stdx,0.19*\stdy) ;
  \coordinate (br2)  at (0.52*\stdx,0.18*\stdy) ;
  \coordinate (br3)  at (0.58*\stdx,0.19*\stdy) ;
  \draw[black,thick]
            (p1m) [out=215,in=25] to (0,0)
            (p1m) [out=125,in=-30] to (p20a)
            (p1m) [out=155,in=-25] to (p20b)
            (p1m) [out=-60,in=145] to (\stdx-0.07,0)
            (p1m) [out=55,in=200] to (p23a)
            (p1m) [out=10,in=170] to (p23b) ;
  \filldraw[black] (p1m) circle (\radvtx) ;
  }
\def\tikzpicTwocellGGfour{%
  \tikzpicTwocell
  \coordinate (p1m)  at (0.47*\stdx+0.5*\stdxx,0.52*\stdy) ;
  \coordinate (p20a) at (\stdxx+0.09,\stdy) ;
  \coordinate (p23a) at (\stdx+\stdxx-0.10,\stdy) ;
  \coordinate (p23b) at (\stdx+0.19*\stdxx-0.02,0.19*\stdy) ;
  \coordinate (bl1)  at (0.53*\stdx,0.92*\stdy) ;
  \coordinate (bl2)  at (0.59*\stdx,0.93*\stdy) ;
  \coordinate (bl3)  at (0.65*\stdx,0.92*\stdy) ;
  \coordinate (br1)  at (0.46*\stdx,0.19*\stdy) ;
  \coordinate (br2)  at (0.52*\stdx,0.18*\stdy) ;
  \coordinate (br3)  at (0.58*\stdx,0.19*\stdy) ;
  \draw[black,thick]
            (p1m) [out=215,in=25] to (0,0)
            (p1m) [out=125,in=-30] to (p20a)
            (p1m) [out=55,in=200] to (p23a)
            (p1m) [out=-10,in=190] to (p23b) ;
  \filldraw[black] (p1m) circle (\radvtx) ;
  }
\def\tikzpicTwocellGGfourlabeled{%
  \tikzpicTwocellGGfour
  \node at (0.06*\stdx+\stdxx,0.72*\stdy) {$\scriptstyle H_{\!1}$} ;
  \node at (0.54*\stdx+0.8*\stdxx,0.9*\stdy) {$\scriptstyle H_{\!2}$} ;
  \node at (0.85*\stdx+\stdxx,0.72*\stdy) {$\scriptstyle H_{\!3}$} ;
  \node at (0.53*\stdx,0.14*\stdy) {$\scriptstyle H_{\!4}$} ;
  }
\def\tikzpicpercol{%
  \shade[left color=\colorMd,right color=\colorN,draw=\colorL,thick]
            (p1m) [out=170,in=10] to ++(-0.14*\stdx,-0.09*\stdy)
            [out=195,in=270] to ++(-0.09*\stdx,0.17*\stdy)
            [out=90,in=180] to ++(0.16*\stdx,0.26*\stdy)
            [out=0,in=180] to ++(0.38*\stdx,-0.32*\stdy) ;
  \draw[black,thick,dash pattern=on 1.8pt off 1.6pt,dash phase=0.8pt]
            (p1m) [out=125,in=-30] to (p20a)
            (p1m) [out=155,in=-25] to (p20b)
            (p1m) [out=55,in=200] to (p23a) ;
  \draw[black,thick,dash pattern=on 1.5pt off 1.2pt,dash phase=0.6pt]
            (p1m) [out=10,in=170] to (p23b) ;
  \filldraw[\colorMJ]
            (p1m) [out=170,in=10] to ++(-0.14*\stdx,-0.09*\stdy)
            [out=195,in=270] to (p1n) [out=90,in=150] to ++(0.075*\stdx,0.05*\stdy)
            [out=-30,in=170] to (p1m) ;
  \draw[black,thick]
            (p1m) [out=170,in=10] to ++(-0.14*\stdx,-0.09*\stdy)
            [out=195,in=270] to (p1n) ;
  \draw[black,thick,densely dotted]
            (p1n) [out=90,in=150] to ++(0.075*\stdx,0.05*\stdy)
            [out=-30,in=170] to (p1m) ;
            }
\def\tikzpicpercolfour{%
  \shade[left color=\colorMd,right color=\colorN,draw=\colorL,thick]
            (p1m) [out=170,in=10] to ++(-0.14*\stdx,-0.09*\stdy)
            [out=195,in=270] to ++(-0.09*\stdx,0.17*\stdy)
            [out=90,in=180] to ++(0.16*\stdx,0.26*\stdy)
            [out=0,in=180] to ++(0.38*\stdx,-0.32*\stdy) ;
  \draw[black,thick,dash pattern=on 1.8pt off 1.6pt,dash phase=0.8pt]
            (p1m) [out=125,in=-30] to (p20a)
            (p1m) [out=55,in=200] to (p23a) ;
  \filldraw[\colorMJ]
            (p1m) [out=170,in=10] to ++(-0.14*\stdx,-0.09*\stdy)
            [out=195,in=270] to (p1n) [out=90,in=150] to ++(0.078*\stdx,0.09*\stdy)
            [out=-30,in=170] to (p1m) ;
  \draw[black,thick]
            (p1m) [out=170,in=10] to ++(-0.14*\stdx,-0.09*\stdy)
            [out=195,in=270] to (p1n) ;
  \draw[black,thick,densely dotted]
            (p1n) [out=90,in=150] to ++(0.078*\stdx,0.09*\stdy)
            [out=-30,in=170] to (p1m) ;
            }
 
\def\Act           {{\triangleright}}   % convention
\def\act           {\,{\Act}\,}         % convention
\def\Actr          {{\triangleleft}}    % convention
\def\actr          {\,{\Actr}\,}        % convention
\def\actF          {\,{\Yright}\,}
\def\ActrF         {{\Yleft}} 
\def\actrF         {\,\ActrF\,} 
\def\be            {\begin{equation}}
\def\bearl         {\begin{array}{l}}
\def\bearll        {\begin{array}{ll}}

\def\boti          {\,{\boxtimes}\,}

\def\cala          {{\mathcal A}}
\def\Cala          {{\!\mathcal A}}
\def\calb          {{\mathcal B}}

\def\calc          {{\mathcal C}}
\def\Calc          {{\!\mathcal C}}
\def\cald          {{\mathcal D}}
\def\Cald          {{\!\mathcal D}}

\def\calm          {{\mathcal M}}
\def\Calm          {{\!\mathcal M}}
\def\caln          {{\mathcal N}}
\def\Caln          {{\!\mathcal N}}
\def\call          {{\mathcal L}}

\def\calx          {{\mathcal X}}
\def\caly          {{\mathcal Y}}
\def\calz          {{\mathcal Z}}
\def\cir           {\,{\circ}\,}
\def\Colon         {:\quad}
\def\complex       {{\mathbbm C}}

\def\dim           {{\rm dim}}

\def\dualcat       {{\cala^*_{\!\calm}}}                 % Dual tensor category
\def\DD            {\mathbb{D}}

\def\DDC           {\mathbb{D}_{\calc}}
\def\dd            {^{\vee\vee}}
\def\ldd           {{}^{\vee\vee\!}}

\def\ee            {\end{equation}}
\def\eear          {\end{array}}

\def\Enumerate     {\def\leftmargini{1.34em}~\\[-1.42em]\begin{enumerate}}
\def\Enumeratei    {%\def\leftmargini{1.84em}
	~\\[-1.42em]\begin{enumerate}[{\rm (i)}]\addtolength\itemsep{-5pt}}
\def\Enumerateii   {~\\[-1.42em]\def\leftmargini{1.84em}  
                   \begin{enumerate}[{\rm (i)}]\addtolength\itemsep{-3pt}}

\def\eq            {\,{=}\,}  % instead of = in in-line formulas, avoids stretchable space
\def\findim        {fini\-te-di\-men\-si\-o\-nal}
\def\Funle         {{\mathrm{Lex}}}                  % left exact functors
\def\Funre         {{\mathrm{Rex}}}                  % right exact functors
\def\Fun           {{\mathrm{Fun}}}            % all functors
\def\FunM          {{\Fun_\cala(\calm,\cala)}}

\def\gra           {\ensuremath{\varGamma}}
\def\Hom           {\mathrm{Hom}}

\newcommand\hsp[1] {\mbox{\hspace{#1 em}}}
\def\id            {{\mathrm{id}}}

\def\Irr            {{\mathrm{Irr}}}
\def\iHom          {\underline{\Hom}}
\def\icoHom          {\underline{\mathrm{coHom}}}
\def\iHomM         {\underline{\Hom}_\calm}
\def\iN            {\,{\in}\,}

\def\Itemize       {%\def\leftmargini{1.05em}
	           ~\\[-1.65em] \begin{itemize}\addtolength\itemsep{-6pt}}
\def\ko            {{\ensuremath{\Bbbk}}}    % field
\def\la            {^{\rm la}}             % superscript left adjoint
\def\lL            {\mathrm l}  % not "L"
\def\lla           {^{\rm lla}}           % superscript double left adjoint

\def\Mor           {\mathbf{\mathbbm{M}}}
\def\Mod           {{\text{Mod}}}

\def\Nat           {\mathrm{Nat}}
\def\Nak           {\mathbb{N} }
\def\naki          {\mathrm n} 

\def\NakL          {\Nak^{\mathrm l}}

\def\NakR          {\Nak^{\mathrm r}}
\newcommand\nicemapsto[1]{\mbox{\LARGE $\stackrel{\mbox{\small $#1$}}\longmapsto$}}
\newcommand\nxl[1] {\\[#1mm]}
\newcommand\Nxl[1] {\\[-1.3em]\\[0.#1em]}

\def\one           {{\bf1}}

\def\opp           {^{\rm opp}}              % opposite
\def\Ot            {{\otimes}}
\def\ot            {\,{\otimes}\,}

\def\oti           {\,{\otimes}\,}

\def\otik          {\,{\otimes_\ko}\,}
\def\Proj          {{\rm Proj}\,}
\newcommand\rarr[1]{\xrightarrow{~#1~}}
\newcommand\Rarr[1]{\,{\xrightarrow{\,#1\,}}\,}
\def \Ra           {\xRightarrow{~\,~}}

\def\Reg           {\mathrm{Reg}}
\def\ra            {^{\rm ra}}              % superscript right adjoint
\def\rR            {\mathrm r}  % not "L"
\def\rra           {^{\rm rra}}            % superscript double right adjoint
\newcommand{\Se}{\mathbb{S}}              % right Serre functor
\newcommand{\lSe}{\overline{\mathbb{S}}}              % left Serre functor
\def\setmin        {\,{\setminus}\,} 
\def\Times         {\,{\times}\,}
\def\To            {\,{\to}\,}
\def\Varphi        {\!\varphi}

\def\xcong         {\,{\Xcong}~}
\def\tr            {\mathrm{tr}}
\def\qtr           {\mathrm{qtr}}
\def\un            {{\mathbf{1}_\Cala}}

\def\vect          {\mathcal V\hspace*{-0.7pt}e\hspace*{-0.3pt}c\hspace*{-0.2pt}t}
\def\Vee           {{}^{\vee\!}}
\def\VEe           {{}^{\vee\!\!}}
\newcommand\void[1] {}
\def\xcong         {\,{\xrightarrow{~\cong\,}}\,}
\def\xnatiso       {{\;\xRightarrow{\,\cong\,}\,}}

\def\Z             {\mathbb{Z}}

\newcommand\bmixt[2]  {#1 \,{\ominus}\, #2}      % second mixed product 1
\newcommand\bmixtd[2] {#1 \,{\boxminus}\, #2}    % second mixed product 2  
\newcounter{cmt}
\definecolor{DarkViolet} {rgb}{0.580392,0.000000,0.827450}

\newcommand\scopeArrow[2] {\begin{scope}[decoration={markings,mark=at position #1
                   with \arrow{#2}}]}
\def\arrowObj   {stealth'}

\def\colorCat   {blue!70!black}
\def\colorMor   {black}
\def\colorObj   {black}
\def\colorPiv   {white}
\def\hsA        {0.11}  %  typical hor shift for arrows
\def\symbPiv    {circle}
\def\widthMor   {1.2pt}
\def\widthObj   {1.5pt} %  normal line thickness
\def\widthObjSm {0.9pt}  
\def\widthPiv   {0.09}  %  line thickness for small pics

\newcommand\drawPiv[2]  {\fill[line width=\widthObjSm,draw=\colorObj,fill=\colorPiv] (#1,#2) \symbPiv (\widthPiv)}

\def\locscale   {1}

\def\locC    {0.4}   %  radius circle
\def\locH    {2.7}   %  height
\def\locHs   {0.7}   %  starting height short line
\def\locHM   {0.58}  %  height morphism
\def\locW    {0.5}   %  distance between lines
\def\locWM   {0.15}  %  extra width morphism

\begin{document}

%   \version\versionno  \draft

\thispagestyle{empty}
\begin{flushright}
   {\sf ZMP-HH/24-13}\\
   {\sf Hamburger$\;$Beitr\"age$\;$zur$\;$Mathematik$\;$Nr.$\;$970}\\[2mm]
   July  2024
\end{flushright}

\vskip 2.7em

\begin{center}
	{\bf \Large A manifestly Morita-invariant construction
	\\[7pt] of Turaev-Viro invariants}

\vskip 2.6em

{\large 
J\"urgen Fuchs\,$^{\,a},~$
C\'esar Galindo\,$^{\,b},~$
David Jaklitsch\,$^{\,c},~$
Christoph Schweigert\,$^{\,d}$
}

\vskip 15mm

 \it$^a$
 Teoretisk fysik, \ Karlstads Universitet\\
 Universitetsgatan 21, \ S\,--\,651\,88\, Karlstad, Sweden
 \\[9pt]
 \it$^b$
 Departamento de Matem\'aticas, Universidad de los Andes\\
 Carrera 1 \# 18A - 12, Edificio H, Bogot\'a, Colombia
 \\[9pt]
 \it$^c$
 Department of Mathematics, University of Oslo\\
 Moltke Moes vei 35, Niels Henrik Abels hus, 0851 Oslo, Norway
 \\[9pt]
 \it$^d$
 Fachbereich Mathematik, \ Universit\"at Hamburg\\
 Bereich Algebra und Zahlentheorie\\
 Bundesstra\ss e 55, \ D\,--\,20\,146\, Hamburg, Germany

\end{center}

\vskip 3.7em

\noindent{\sc Abstract}\\[3pt]
We present a state sum construction that assigns a scalar to a skeleton in a
closed oriented three-dimensional manifold. The input datum
is  the pivotal bicategory $\mathbf{Mod}^{\rm sph}(\cala)$ of
spherical module categories over a spherical fusion category $\cala$.
 \\
The interplay of algebraic structures in this pivotal bicategory with moves
of skeleta ensures that our state sum is independent of the skeleton
on the manifold. We show that the bicategorical invariant recovers
the value of the standard Turaev-Viro invariant associated to $\cala$,
thereby proving the independence of the Turaev-Viro invariant under pivotal Morita
equivalence without recurring to the Reshetikhin-Turaev construction.
 \\
A key ingredient for the construction is the evaluation of graphs on the
sphere with labels in $\mathbf{Mod}^{\rm sph}(\cala)$ that we develop
in this article.  A central tool are Nakayama-twisted traces on pivotal
bimodule categories which we study beyond semisimplicity.  
                         
\newpage
\tableofcontents{}

\newpage
%%%%%%%%%%%%%%%%%%%%%%%%%%%%%%%%%%%%%%%%%%%%%%%%%%%%%%%%%%%%%%%%%%%%%%%%

\section{Introduction}

State sum constructions are an important source of (extended) topological
field theories and modular functors. A classical example is the Turaev-Viro
construction, which is based on a spherical fusion category $\cala$.
It is well understood that this produces an oriented topological field
theory that is once-extended, and it is expected to yield a fully 
extended oriented topological field theory.
The once-extended theory obtained from $\cala$ via the Turaev-Viro construction 
coincides with the one obtained by the Reshetikhin-Turaev surgery construction 
applied to the Drinfeld center $\calz(\cala)$ \Cite{Thm.\,17.1}{TV17}. 

On the other 
hand, the Drinfeld center, as a pivotal braided finite tensor category, is a pivotal 
Morita invariant \Cite{Thm.\,5.16}{FGJS}. Pivotal Morita equivalent spherical fusion
categories therefore lead to equivalent once-extended topological field theories.
This can also be understood by realizing Morita contexts in quantum field theory by 
duality defects labeled by invertible pivotal bimodule categories, as has been first 
demonstrated for two-dimensional field theories in \cite{ffrs5}. 

The standard Turaev-Viro construction is based on a chosen representative of a
pivotal Morita class of pivotal fusion categories. It is instructive to also have a
\emph{manifestly} Morita invariant construction. Achieving such a description is
the primary goal of this paper in which, for the sake of simplicity, we concentrate
on the invariants assigned to closed oriented three-manifolds, rather than
constructing a full-fledged topological field theory.

Our goal can also be motivated algebraically. To obtain a better alignment between
categorical dimensions and geometric dimensions, consider instead of a finite tensor category $\cala$ the one-object bicategory $\mathrm B\cala$ whose 
single endocategory is $\cala$. Just like the one-object category $\mathrm BA$
furnished by a finite-dimensional algebra $A$ calls for being completed to
the category of finite-dimensional $A$-modules, the bicategory 
$\mathrm B\cala$ calls for a completion as well, which has to be 
a suitable bicategory of $\cala$-module categories.
 
Such completions have also emerged in state sum constructions in four dimensions
\cite{DR}, under the assumption of semisimplicity in the form of idempotent completions.
As we are interested in oriented manifolds, we need additional structure
on the algebraic input data. In the classical Turaev-Viro construction this is
a pivotal structure on $\cala$ that has the property of being spherical. In the
present paper we use instead recent results from \cite{FGJS}, where in particular the
notions of spherical module category and pivotal Morita context have been developed
which allow one to work in a bicategorical setting. Specifically, we consider a
pivotal Morita invariant: the pivotal bicategory $\mathbf{Mod}^{\rm sph}(\cala)$ of
spherical $\cala$-module categories, which we study in Section \ref{traces_modsphA}.
We take this bicategory as candidate for 
a spherical $2$-idempotent completion of $\cala$.

Once we decide to work in this setup, it is crucial to have a sound theory
of traces on endomorphisms of objects in bimodule categories at our disposal. This
need explains the organization of the paper: After a review of relevant background
and earlier work in Section \ref{sec:background}, we turn in Section \ref{sec:traces}
to a systematic study of traces of $2$-morphisms in the bicategory of spherical
$\cala$-module categories. To this end we first examine traces on pivotal bimodule
categories without imposing semisimplicity. This combines insights from \cite{FGJS}
with the theory of Nakayama-twisted traces (sometimes called modified traces) as
developed in \cite{shShi,SchW}. The latter theory does not require a monoidal
structure, which makes it suitable for application to finite (bi)module categories 
$\call \eq {}_\calc\call_\cald$. The prize
to pay is that traces are not defined for endomorphisms $\Hom_\call(p,p)$, but for
morphisms twisted by the Nakayama functor $\NakR_\call$ and that one must restrict
oneself to projective objects $p\iN\call$: traces are linear functionals on
$\Hom_\call(p,\NakR_\call(p))$, see Definition \ref{def:tpX}. In Proposition
\ref{bimod_partial_trace_prop}, we show that these traces possess partial trace
properties.
If the categories $\calc$ and $\cald$ are unimodular and pivotal and $\call$ is 
pivotal as well, then the twisted endomorphisms $\Hom_\call(p,\NakR_\call(p))$, for
$p\iN\call$ projective, can naturally be identified in two different ways with 
endomorphisms $\Hom_\call(p,p)$. In Theorem \ref{thm:spherical_traces} we show that
these two different possibilities agree in the case that the Radford natural 
isomorphism that relates the two Serre functors on $\call$ can be expressed in terms 
of the pivotal structures on $\call$, i.e.\ $\call$ is a spherical bimodule category
in the sense of Definition \ref{spherical_bimod}. If $\call$ is semisimple, then
traces are defined for arbitrary endomorphisms on any object in $\call$. This leads
to Theorem \ref{thm:CY_bimodule}, which reveals two structures of a Calabi-Yau 
category on the semisimple category $\call$, which again 
coincide if $\call$ is a spherical bimodule category. 

The algebraic structures uncovered in Section \ref{sec:traces} provide important input
for the Morita invariant state sum construction based on the bicategory of spherical
$\cala$-modules that we present in Section \ref{sec:state_sum}. Our construction
passes two important checks. On the one hand, in Theorem \ref{skeleton_inv} we show
that the new state sum for three-manifolds that is supplied by the Morita invariant 
construction does not depend on the choice of skeleton used to define the sum.

This is a non-trivial insight: it demonstrates that the geometric primary moves on
manifold skeletons interact in synergy with the algebraic structures of the spherical
module categories. Theorem \ref{thm:TV=St}, our second main result, shows, on the 
other hand, that our construction reproduces the three-manifold invariants obtained
via the original Turaev-Viro construction based on the representative spherical
fusion category $\cala$. This constitutes an intrinsic proof that the Turaev-Viro
invariants indeed only depend on the pivotal Morita class of a spherical fusion
category, without any reference to the Reshetikhin-Turaev construction.

%%%%%%%%%%%%%%%%%%%%%%%%%%%%%%%%%%%%%%%%%%%%%%%%%%%%%%%%%%%%%%%%%%%%%%%%

\section{Preliminaries}\label{sec:background}

In this section we revisit a few pertinent concepts and structures. Throughout the
paper we fix an algebraically closed field \ko. All categories we consider are 
supposed to be finite linear abelian over \ko, i.e.\ categories that are equivalent as
an abelian category to the category of \findim\ modules over a \findim\ \ko-algebra.
We take without loss of generality all monoidal categories to be strict so as
to simplify the exposition.

%%%%%%%%%%%%%%%%%%%%%%%%%%%%%%%%%%%%%%%%%%%%%%%%%%%%%%%%%%%%%%%%%%%%%%%%

\subsection{Tensor categories and module categories}

We first recall, closely following \cite{EGno}, a few standard notions from the 
theory of tensor categories. A \emph{finite tensor category} is a finite rigid 
monoidal category $\cala$ whose tensor product functor $\otimes$ is bilinear and
whose monoidal unit $\un$ is a simple object. A finite tensor category is called a
\emph{fusion category} if its underlying category is semisimple. The
\emph{monoidal opposite} $\overline{\cala\,}$ of a monoidal category $\cala$ is the
monoidal category with the same underlying category as $\cala$, but with reversed
tensor product, i.e.\ $a \,{\otimes_ {\overline{\cala\,}}}\,b \eq b \ot a$, and with
the accordingly adjusted associators. 

Given a tensor category $\cala$, our conventions regarding dualities are as follows. 
A right dual $a^\vee$ of an object $a\iN\cala$ comes equipped with evaluation and 
coevaluation morphisms
\be
  {\rm ev}_a\Colon a^\vee \Ot\, a\to \un \qquad{\rm and}\qquad
  {\rm coev}_a\Colon \un\to a\ot a^\vee ,
  \ee
while a left dual $\Vee a$ of $a\iN \cala$ comes with evaluation and coevaluation
morphisms
  \be
  \widetilde{{\rm ev}_a}\Colon a \,\Ot \Vee a\to \un \qquad{\rm and}\qquad
  \widetilde{{\rm coev}_a}\Colon \un\to \Vee a\ot a \,.
  \ee
The graphical depiction of these morphisms is described in Appendix \ref{app_graph}.
Duals are unique up to unique isomorphism compatible with evaluation and 
coevaluation morphisms. We canonically identify the left dual of the right dual of an
object, as well as the right dual of the left dual, with the object itself. As a 
consequence, the component of a pivotal structure for a left dual has a right dual 
as its codomain.

\medskip

A \emph{$($left$)$ module category} over a finite tensor category $\cala$, or 
\emph{$($left$)$ $\cala$-module category}, is a category $\calm$ together with an 
exact \emph{module action} functor 
  \be
  \Act \Colon \cala\Times\calm\rarr~\calm 
  \ee
and a mixed associator obeying a pentagon axiom. In order to indicate the
tensor category over which $\calm$ is a module, we also denote it by
${}_\cala\calm$. Invoking an analogue for module categories of Mac Lane's 
strictification theorem \Cite{Rem.\,7.2.4}{EGno}, we assume strictness also
for module categories. A \emph{right} $\calb$-module $\caln$ is defined as a left 
$\overline{\calb}$-module, with module action functor
  \be
  \actr \Colon \caln\Times\calb\to\caln ,
  \ee
and we also denote it by $\caln_\calb$. Similarly, for finite tensor categories 
$\cala$ and $\calb$ an $(\cala,\calb)$-\emph{bimodule category} is a 
(left) module category over the Deligne product $\cala\boti\overline{\calb}$.

An $\cala$-module category $\calm$ is called \emph{exact} iff $p\act m$ is 
projective in $\calm$ for any projective object $p\iN\cala$ and any object 
$m\iN\calm$. A module category which is not equivalent to a direct sum of two
non-trivial module categories is said to be \emph{indecomposable}.

A \emph{module functor} between $\cala$-module categories $\calm$ and $\caln$ is a
functor $H\colon \calm\Rarr{}\caln$ together with a \emph{module constraint}, i.e.\
a collection of natural isomorphisms $H(a\act m)\xcong a\act H(m)$ for $a\iN\cala$
and $m\iN\calm$ obeying appropriate pentagon axioms. In the case that $\calm$ is
exact, every $H\iN \Fun_\cala(\calm,\caln)$ is an exact module functor. A
\emph{module natural transformation} between module functors is a natural
transformation between the underlying functors that commutes with the respective
module structures. We denote by $\Fun_\cala(\calm,\caln)$ the category that has
module functors between two exact $\cala$-modules $\calm$ and $\caln$ as objects and module
natural transformations as morphisms. We denote by $\Nat_{\rm mod}(H_1,H_2)$ the vector
space of module natural transformations between given module functors $H_1$ and $H_2$.

The \emph{dual tensor category} $\dualcat$ of a
finite tensor category $\cala$ with respect
to an exact $\cala$-module $\calm$ is the category of module endofunctors,
$\dualcat \coloneqq \Fun_\cala(\calm,\calm)$, with tensor product given by
composition of functors, and dual objects given by left and right adjoints. If
$\calm$ is indecomposable, then the identity functor $\id_\calm$ is simple, making
$\dualcat$ a finite tensor category. The evaluation of a functor on an object turns 
$\calm$ into an $\dualcat$-module category.

The category of module functors has the structure of a bimodule category. More
specifically, given exact module categories ${}_\cala\calm$ and ${}_\cala\caln$,
$\Fun_\cala(\calm,\caln)$ becomes an $(\cala_\caln^*,\dualcat)$-bimodule category via
composition of module functors:
  \be
  \label{left_module_functors}
  F_1\actF H:= F_1\cir H \qquad{\rm and}\qquad H\actrF F_2:=H\cir F_2 
  \ee
for $H \iN \Fun_\cala(\calm,\caln)$, $F_1\iN\cala_\caln^*$ and $F_2\iN\dualcat$.

%%%%%%%%%%%%%%%%%%%%%%%%%%%%%%%%%%%%%%%%%%%%%%%%%%%%%%%%%%%%%%%%%%%%%%%%

\subsection{Internal Hom and relative Serre functors}

Given a module category ${}_\cala\calm$, for every object $m\iN\calm$ the action
functor $-{\act}m\colon \cala\Rarr~\calm$ is exact and hence has a right adjoint
$\iHomM^\cala(m,-)\colon \calm\Rarr~\cala$, i.e.\ there are natural isomorphisms
$\Hom_\calm(a\act m,n) \Rarr\cong \Hom_\cala(a,\iHomM^\cala(m,n))$
for $a\iN\cala$ and $m,n\iN\calm$. This extends to a left exact functor
$\iHom_\calm^\cala(-,-)\colon \calm\opp\Times\calm \Rarr~ \cala$,
which is called the \emph{internal Hom} functor of ${}_\cala\calm$. 
Additionally, there are canonical natural isomorphisms
  \be
  \label{ihom_lmod} 
  \iHomM^\cala(m,a\act n)\cong a\ot \iHomM^\cala(m,n)
  \ee
that turn $-{\act} m\dashv \iHom_\calm(m,-)$ into an adjunction of $\cala$-module 
functors, as well as
  \be
  \label{ihom_rmod}
  \iHomM^\cala(a\act m,n) \cong \iHomM^\cala(m,n)\ot a^\vee .
  \ee
Dually, for $m\iN\calm$ the action functor $-{\act} m\colon \cala\To\calm$ has a 
left adjoint, called the \emph{internal coHom}, satisfying
  \be
  \Hom_\calm(n,a\act m) \rarr\cong \Hom_\cala(\icoHom_\calm^\cala(m,n),a)
  \ee
for $a\iN\cala$ and $m,n\iN\calm$. This extends to a right exact functor
$\icoHom_\calm^\cala(-,-)\colon \calm\opp\Times\calm \To\cala$.
The internal Hom and coHom are related by the rigid duality in $\cala$:
  \be
  \label{hom_cohom}
  \icoHom_\calm^\cala(m,n) \cong {}^\vee\iHom_\calm^\cala(n,m)
  \qquad\text{for}~~ m,n\iN\calm \,.
  \ee 
Hence analogously to \eqref{ihom_lmod} and \eqref{ihom_rmod}, there are 
coherent natural isomorphisms
  \be
  \begin{aligned}
  & \icoHom_\calm^\cala(m,a\act n) \cong a\ot \icoHom_\calm^\cala(m,n) \qquad\text{and}
  \Nxl1 
  & \icoHom_\calm^\cala(a\act m,n) \cong \icoHom_\calm^\cala(m,n)\ot \Vee a\,.
  \end{aligned}
  \ee 

A \emph{Serre functor} of a finite \ko-linear category $\calx$ is an endofunctor
$\Se$ of $\calx$ together with natural isomorphisms between the vector spaces
$\Hom_\calx(x,y)$ and $\Hom_\calx (y,\Se(x))^*_{}$. If $\calx$ admits a Serre
functor, then $\calx$ is necessarily semisimple. In contrast, if $\calx \eq \calm$
is a finite module category over a finite tensor 
category $\cala$, then there is an internalized version
of the Serre functor which exists without the requirement of semisimplicity:

\begin{defi}{\Cite{Def.\,4.22}{fuScSc}}
Let $\calm$ be a left $\cala$-module category. A \emph{$($right$)$ relative 
Serre functor} on $\calm$ is an endofunctor $\Se_\calm^\cala \colon \calm\To\calm$
together with a family
  \be
  \iHomM^\cala(m,n)^\vee \rarr\cong \iHomM^\cala(n,\Se_\calm^\cala(m))
  \ee
of natural isomorphisms for $m,n\iN\calm$. Similarly, a \emph{$($left$)$ relative
Serre functor} $\lSe_\calm^\cala$ comes with a family
  \be
  {}^\vee\iHomM^\cala(m,n) \rarr\cong \iHomM^\cala(\lSe_\calm^\cala(n),m)
  \ee
of natural isomorphisms.
\end{defi}

A module category admits relative Serre functors if and only if it is exact  \Cite{Prop.\,4.24}{fuScSc}. In that case the left and right 
relative Serre functors are quasi-inverses of each other and can be uniquely 
expressed (see \Cite{Lemmas\,3.3-3.5}{Sh}) as
  \be
  \Se_\calm^\cala(m)\cong \iHomM^\cala(m,-)\ra(\un) \qquad\text{and}\qquad
  \lSe_\calm^\cala(m)\cong \icoHom_\calm^\cala(m,-)\la(\un)\,.
  \label{eq:relSerre4M}
  \ee

\begin{prop}
For $\calm$ an exact $(\cala,\calb)$-bimodule category, there are coherent
natural isomorphisms
  \be
  \Se_\calm^\cala(a\act m)\cong a\dd\act \Se_\calm^\cala(m) \qquad\text{and}\qquad
  \Se_\calm^\cala(m\actr b)\cong \Se_\calm^\cala(m)\actr b\dd
  \label{Serre_twisted}
  \ee
which turn the relative Serre functor $\Se_\calm^\cala$ into a twisted bimodule 
equivalence. 
\end{prop}

The relation between the relative Serre functor and the Nakayama functor (see Section
\ref{sec:Nakayama} below) and the fact \Cite{Thm.\,3.18}{fuScSc} that the Nakayama
functor comes with a canonical structure of a twisted module functor leads to

\begin{prop}[\Cite{Thm.\,3.10}{Sh} and \Cite{Prop.\,2.11}{FGJS}]%
  \label{Serre_and_module_functors}
Given exact $\cala$-module categories $\calm$ and $\caln$, for every module functor 
$F \colon \calm\To\caln$ there is a natural isomorphism
  \be
  \label{Serre_module_functor}
  \Lambda_F\Colon \Se_\caln^\cala\cir F \xRightarrow{~\cong~}
  F\rra\cir \Se_\calm^\cala
  \ee
of twisted module functors, where $F\rra$ is the double right adjoint of $F$. 
$\Lambda_F$ is compatible with composition of module functors, i.e.\ the diagram
  \be
  \label{Serre_compatibility_composition}
  \begin{tikzcd}[row sep=2.3em,column sep=2.9em]
  \Se_\call^\cala\cir H\cir F \ar[r,"\Lambda_{H\circ F}\,",Rightarrow]
  \ar[d,swap,"\Lambda_H\circ\id\,",Rightarrow]
  & (H\cir F)\rra\cir \Se_\calm^\cala \ar[d,"\,\cong",Rightarrow]
  \\
  H\rra\cir\Se_\caln^\cala\,\cir F \ar[r,swap,"\id\circ\Lambda_F~",Rightarrow]
  & H\rra\cir F\rra\cir \Se_\calm^\cala
  \end{tikzcd}
  \ee
commutes for $H\iN\Fun_\cala(\caln,\call)$ and $F\iN\Fun_\cala(\calm,\caln)$. 
Analogously, in the case of bimodule categories and a bimodule functor $F$, 
the natural isomorphism $\eqref{Serre_module_functor}$ is an isomorphism of 
twisted bimodule functors.
\end{prop}

%%%%%%%%%%%%%%%%%%%%%%%%%%%%%%%%%%%%%%%%%%%%%%%%%%%%%%%%%%%%%%%%%%%%%%%%

\subsection{Dualities for module categories}\label{sec:module_duals}

There is a notion of dualities for a $1$-morphism in a bicategory $\mathscr{F}$: A
\emph{right dual $($or right ad\-joint$\,)$} to a $1$-morphism $a\iN\mathscr{F}(x,y)$
consists of a $1$-morphism $a^\vee\iN\mathscr{F}(y,x)$ and $2$-morphisms
$\textbf{1}_y\,{\Rightarrow}\, a\cir a^\vee$ and
$a^\vee\cir a\,{\Rightarrow}\, \textbf{1}_x$ that satisfy the appropriate snake
relations; left duals are similarly defined. A bicategory in which every $1$-morphism
has both a right and a left dual is said to be a \emph{bicategory with dualities}.

To every exact $\cala$-module category $\calm$ there is associated a two-object
bicategory $\mathbb{M}$ \Cite{Sec.\ 3}{FGJS}. The two objects of $\Mor$ are denoted
by ``$+$'' and ``$-$''; their Hom-categories are
  \be
  \Mor(\mathbf{+},\mathbf{+}) = \cala \,, \quad
  \Mor(\mathbf{-},\mathbf{-}) = \dualcat \,,\quad
  \Mor(\mathbf{-},\mathbf{+}) = \calm \,,\quad
  \Mor(\mathbf{+},\mathbf{-})=\Fun_\cala(\calm,\cala) \,.
  \ee
In particular, objects of $\calm$, $\FunM$, $\cala$ and $\dualcat$ are $1$-morphisms
in $\mathbb{M}$. The bicategory $\mathbb{M}$ is a bicategory with dualities
\Cite{Thm.\,4.2}{FGJS}. The tensor categories $\cala$ and $\dualcat$ are rigid.
Duals for objects in $\calm$ and $\FunM$ can be realized as follows:
\Enumeratei
  \item 
The \emph{right dual} of an object $m\iN \calm$ is the $\cala$-module functor 
  \be 
  m^\vee := \iHom_\calm^\cala(m,-) ~\in \FunM \,,
  \ee 
with evaluation and coevaluation morphisms given by the counit and the unit of the
ad\-junc\-tion $(-\Act m) \,{\dashv}\, \iHom^\cala_\calm(m,-)$, respectively.
  \item 
The \emph{left dual} of an object $m\iN \calm$ is the $\cala$-module functor 
  \be
  \Vee m := \icoHom_\calm^\cala(m,-) ~\in \FunM \,,
  \ee 
with evaluation and coevaluation morphisms given by the counit and the unit of the 
ad\-junc\-tion $\icoHom^\cala_\calm(m,-) \,{\dashv}\, (-\Act m)$, respectively.
    \item 
The \emph{right dual} of a module functor $H \iN \FunM$ is the object
  \be
  H^\vee := H\ra(\un) ~\in\calm \,,
  \ee 
where the evaluation and coevaluation morphisms come from the unit and counit of the 
adjunction $H \,{\dashv}\, H\ra$.
  \item 
The \emph{left dual} of a module functor $H\iN\FunM$ is the object
  \be
  \VEe H := H\la(\un) ~\in\calm \,,
  \ee
with evaluation and coevaluation morphisms that come from the unit and counit 
of the adjunction $H\la \,{\dashv}\, H$.
\end{enumerate}

Double-duals for module categories are the relative Serre functors:

\begin{prop}[\Cite{Prop.\,4.11,\,\&\,Lem.\,5.3}{FGJS}]\label{double_duals}
For $\calm$ an exact module category over a finite tensor category $\cala$,
let $m\iN\calm$ and $H\iN\FunM$. We have natural isomorphisms
  \be
  \begin{aligned}
  {\rm (i)} \quad & m\dd\cong \Se_\calm^\cala(m) \,, \qquad\qquad 
  & {\rm (iii)} \quad & H\dd\cong H\rra \cong \Se_\FunM^{\overline{\dualcat}}(H) \,,
  \Nxl2
  {\rm (ii)} \quad & \ldd m\cong \lSe_\calm^\cala(m) \,, \qquad
  & {\rm (iv)} \quad & \,{}\dd\! H\cong H\lla \cong \Se_\FunM^{\overline{\cala}}(H) \,.
  \end{aligned}
  \ee
More generally, for exact $\cala$-module categories $\calm$ and $\caln$ there are
natural isomorphisms
  \be
  \Se_{\Fun_\Cala(\Calm,\Caln)}^{\cala_\caln^*}(H)\cong H\lla \quad\text{and}\quad
  \Se_{\Fun_\Cala(\Calm,\Caln)}^{\overline{\dualcat}}(H) \cong H\rra 
  \quad\text{for}\quad H\iN\Fun_\Cala(\calm,\caln)\,.
  \ee
\end{prop}

%%%%%%%%%%%%%%%%%%%%%%%%%%%%%%%%%%%%%%%%%%%%%%%%%%%%%%%%%%%%%%%%%%%%%%

\subsection{Pivotal Morita equivalence}\label{sec:PME}

Let $\cala$ be a pivotal finite tensor category and $\calm$ an exact left 
$\cala$-module category. The pivotal structure $\rm{p}\colon\id_\cala\xnatiso(-)\dd$
of $\cala$ turns the relative Serre functor of $\calm$ into an $\cala$-module functor via
  \be
  a\act \Se_\calm^\cala(m)\xrightarrow{\;\rm{p}_a\act\id\;}a\dd\act
  \Se_\calm^\cala(m)\xrightarrow{\;\eqref{Serre_twisted}\;} \Se_\calm^\cala(a\:\Act m) \,.
  \ee

\begin{defi}[\Cite{Def.\,3.11}{Sh}]\label{def:pivmodule}
A \emph{pivotal structure} on an exact left $\cala$-module category $\calm$ is a 
module natural isomorphism $\widetilde{\rm{p}} \colon\id_\calm{\xnatiso}\Se_\calm^\cala$.
\end{defi}

A module category together with a pivotal structure is called a \emph{pivotal module
category}. Pivotal right module categories and pivotal bimodule categories are
defined similarly \Cite{Def.\,5.1}{FGJS}.

\begin{prop}[\rm\Cite{Thm.\,3.13}{Sh}] \label{pivotal_dual_category}
Let $\cala$ be a pivotal tensor category and $\calm$ a pivotal $\cala$-mo\-du\-le
category with pivotal structure $\widetilde{\rm{p}}$. The dual tensor category
$\dualcat$ has a pivotal structure given by the composite
  \be
  {\rm q}_F\Colon F \xRightarrow{\,\id\,\circ\,\widetilde{\rm{p}}~} F \cir \Se_\calm^\cala
  \xRightarrow{\;\eqref{Serre_module_functor}~\,} \Se_\calm^\cala\cir F\lla
  \xRightarrow{~\widetilde{\rm{p}}^{-1}\circ\,\id~\,} F\lla
  \label{pivotal_dualcat}
  \ee
for a module endofunctor $F\iN\dualcat$.
\end{prop}

This leads to a variant of categorical Morita equivalence which takes pivotal 
structures into account:

\begin{defi}[\Cite{Def.\ 5.10}{FGJS}]\label{PME}
Two pivotal tensor categories $\cala$ and $\calb$ are said to be \emph{pivotal Morita 
equivalent} iff there exists a pivotal $\cala$-module category $\calm$ together with 
a pivotal tensor equivalence $\calb\,{\simeq}\,\overline{\dualcat}$.
\end{defi}

In analogy to its non-pivotal version, this notion of Morita equivalence is 
characterized by a suitable bicategory of module categories. However, we require
additional structure: Let $\mathscr{F}$ be a bicategory with dualities. Dualities for
$1$-morphisms extend to a pseudofunc\-tor
  \be
  \begin{aligned}[c]
  (-)^\vee\Colon \mathscr{F} & \,\longrightarrow\, \mathscr{F}^{\,\text{op},\text{op}},
  \\
  \hspace*{2.5em} x & \,\longmapsto\, x \,,
  \\
  \hspace*{2.5em} (a\colon x\,{\to}\, y) & \,\longmapsto\, (a^\vee\colon y\,{\to}\, x) \,.
  \end{aligned}
  \ee
A \emph{pivotal structure} on a bicategory $\mathscr{F}$ with dualities is a 
pseudo-natural equivalence 
  \be\label{pivotal_st_bicategory}
  \textbf{P} \Colon \id_\mathscr{F} \xRightarrow{\;\simeq~\,} (-)^{\!\vee\vee}
  \ee
obeying $\textbf{P}_{\!x} \eq \id_x$ for every object $x\iN\mathscr{F}$.
A bicategory together with a pivotal structure is called a \emph{pivotal bicategory}.

Let $\cala$ be a pivotal tensor category. Denote by 
$\textbf{Mod}^{\text{piv}\!}(\cala)$ the bicategory that has pivotal $\cala$-module 
categories as objects, $\cala$-module functors as $1$-morphisms and module natural
transformations as $2$-morphisms. Since pivotal modules are exact
\Cite{Prop.\,4.24}{fuScSc}, every module functor 
$H\colon {}_{\cala\,}\calm \To {}_{\cala\,}\caln$ comes with adjoints
  \be
  H\Vee := H\la :~ {}_{\cala\,}\caln \rarr{} {}_{\cala\,}\calm \qquad\text{and}\qquad
  \Vee H := H\ra :~ {}_{\cala\,}\caln \rarr{} {}_{\cala\,}\calm \,,
  \ee
which are themselves module functors.
These turn $\textbf{Mod}^{\text{piv}\!}(\cala)$ into a bicategory with dualities 
for $1$-morphisms. Moreover, $\textbf{Mod}^{\text{piv}\!}(\cala)$ is endowed with 
a pivotal structure \eqref{pivotal_st_bicategory}. Indeed, given any $1$-morphism
$H\colon {}_{\cala\,}\calm \Rarr~ {}_{\cala\,}\caln$ in 
$\textbf{Mod}^{\text{piv}\!}(\cala)$, define 
  \be
  \label{pivotal_structure_bicategory_of_pivotal_modules}
  \textbf{P}_{\!H}^{}\Colon H
  \xRightarrow{\;\id\cir\widetilde{\rm{p}}_\calm~} H \cir \Se_{\calm}^\cala
  \xRightarrow{\;\eqref{Serre_module_functor}~} \Se_{\caln}^\cala \cir H\lla
  \xRightarrow{\;(\widetilde{\rm{p}}_\caln)^{-1}\cir\id~\,} H\lla ,
  \ee
where $\widetilde{\rm{p}}_\calm$ is the pivotal structure of $\calm$. 
The $2$-morphisms $\textbf{P}_{\!H}^{}$ are invertible and natural in $H$. Moreover,
commutativity of the diagram \eqref{Serre_compatibility_composition} implies that
they are compatible with the composition of module functors. Therefore $\textbf{P}$
constitutes a pivotal structure on the bicategory 
$\textbf{Mod}^{\text{piv}\!}(\cala)$. In particular we have

\begin{prop}[\Cite{Prop.\,5.4}{FGJS}]
Let $\calm$ and $\caln$ be pivotal $\cala$-module categories. The natural isomorphism
\eqref{pivotal_structure_bicategory_of_pivotal_modules} endows the invertible
$(\cala_\caln^*,\dualcat)$-bimodule category $\Fun_\cala(\calm,\caln)$ with a
pivotal structure.
\end{prop}

We thus arrive at the following characterization of pivotal Morita equivalence:

\begin{thm}[\Cite{Thm.\,5.12}{FGJS}]\label{bicategories_pivotal_modules}
Two pivotal tensor categories $\cala$ and $\calb$ are pivotal 
Morita equivalent if and only if $\mathbf{Mod}^{\rm{piv}\!}(\cala)$ and 
$\mathbf{Mod}^{\rm{piv}\!}(\calb)$ are equivalent as pivotal bicategories.
\end{thm}

%%%%%%%%%%%%%%%%%%%%%%%%%%%%%%%%%%%%%%%%%%%%%%%%%%%%%%%%%%%%%%%%%%%%%%%%%%%%%%

\subsection{Nakayama functors and module Radford isomorphisms}\label{sec:Nakayama}

The right exact \emph{Nakayama functor} $\NakR_\calx$ of a finite \ko-linear category
$\calx$ is the image of the identity functor under the Eilenberg-Watts correspondence 
\Cite{Def.\,3.14}{fuScSc}, i.e.\ the coend
  \be
  \NakR_\calx := \int^{x\in\calx}\!\!\Hom_\calx (-,x)^* \ot x \,.
  \label{Nakayama}
  \ee
It comes equipped with a family of natural isomorphisms
  \be
  \NakR_\calx\circ F\xRightarrow{\;\cong~}F\rra\circ\NakR_\calx
  \label{Nakayama_twisted_functor}
  \ee
for every $F\iN\Funre(\calx,\calx)$ that has a right adjoint that is right exact 
as well. The left exact analogue of the Nakayama functor is the end
$\NakL_\calx \,{:=}\, \int_{x\in\calx}\!\Hom_\calx (x,-) \ot x $,
which is a left adjoint to $\NakR_\calx$. In a similar manner as $\NakR_\calx$ the
functor $\NakL_\calx$ is endowed with a family of natural isomorphisms
  \be
  \NakL_\calx\circ F\xRightarrow{\;\cong~} F\lla\circ\NakL_\calx
  \label{lNakayama_twisted_functor}
  \ee
for every $F\iN\Funle(\calx,\calx)$ having a left adjoint that is left exact. 

The Nakayama functors of the $\ko$-linear category underlying a finite tensor
category $\calc$ \Cite{Lemma\ 4.10}{fuScSc} can be described, with the help of
\eqref{Nakayama_twisted_functor} and \eqref{lNakayama_twisted_functor}, as
  \be
  \NakR_\calc\cong \DD_\calc^{-1}\otimes(-)\dd\quad\text{ and }\quad 
  \NakL_\calc\cong \DD_\calc\otimes\ldd(-)
  \label{tensor_Nakayama}
  \ee
where we have defined \emph{distinguished invertible objects}
$\DDC \,{:=}\, \NakL_\calc(\textbf{1})$ and $\DDC^{-1} \,{:=}\, \NakR_\calc(\textbf{1})$,
which are inverses of each other \Cite{Lemma\ 4.11}{fuScSc}. For every finite tensor
category $\calc$ the invertible object $\DDC$ comes with a monoidal natural 
isomorphism
  \be
  r_\calc^{}\Colon \DDC^{}\otimes-\otimes\DDC^{-1}
  \xRightarrow{~\cong~} (-)^{\vee\vee\vee\vee}
  \label{Radford}
  \ee
called the \emph{Radford isomorphism}, which encodes the classical Radford 
$S^4$-theorem in categorical terms. Composing the isomorphisms 
\eqref{tensor_Nakayama} with the Radford isomorphism we obtain
  \be
  \NakR_\calc\cong \ldd(-)\otimes\DD_\calc^{-1} \qquad\text{and}\qquad
  \NakL_\calc\cong (-)\dd\otimes\DD_\calc\,.
  \label{tensor_Nakayama2}
  \ee
A finite tensor category $\calc$ is said to be \emph{unimodular} iff its 
distinguished invertible object $\DDC$ is isomorphic to the monoidal unit 
$\textbf{1}$. Examples of unimodular categories are fusion categories and 
factorizable finite tensor categories \Cite{Prop.\,8.10.10}{EGno}. For a unimodular
finite tensor category the Radford isomorphism
monoidally trivializes the fourth power of the right dual functor: Consider any
trivialization $\textbf{u}_\calc \colon \textbf{1}\Rarr\cong\DD_\Calc^{-1}$
of the distinguished invertible object. The composite
  \be\label{Radford_can}
  \overline{r}_\Calc^{}\Colon\id_\calc\xRightarrow{\textbf{u}^*\otimes\id\otimes\textbf{u}}
  \DD_\Calc^{}\ot{-}\ot\DD_\Calc^{-1} \xRightarrow{\,\eqref{Radford}\,}(-)^{\vee\vee\vee\vee} 
  \ee
constitutes a canonical monoidal isomorphism
that does not depend of the choice of $\textbf{u}_\calc$, since $\textbf{1}$ is 
simple so that $\Hom_\calc(\textbf{1},\DD_\Calc)$ is one-dimensional.

More generally, given a $(\calc,\cald)$-bimodule category $\call$, the isomorphisms
\eqref{Nakayama_twisted_functor} and \eqref{lNakayama_twisted_functor} endow the
Nakayama functors with a twisted module functor structure \Cite{Thm.\,4.4-4.5}{fuScSc}:
  \be
  \NakL_\call(c\act y \actr d)\cong  c \dd\act\Nak^l_\call(y)\actr \ldd d
  \qquad\text{and}\qquad
  \NakR_\call(c\act y \actr d)\cong \ldd c\act\Nak^r_\call(y)\actr d\dd .
  \label{module_Nakayama_twisted_functor}
  \ee
The relative Serre functors of ${}_\calc\call$ are related to the Nakayama functors by
the isomorphisms
  \be
  \DDC\act\NakR_\call \cong \Se_\call^\calc \qquad\text{and}\qquad
  \DDC^{-1}\act\NakL_\call \cong \lSe_\call^\calc
  \label{Nakayama_Serre}
  \ee
of twisted bimodule functors \Cite{Thm.\,4.26}{fuScSc}. Similar isomorphisms of
twisted bimodule functors exist for ${}\call_\cald$. These lead to an extension of
Radford's theorem to bimodule categories:

\begin{thm}[\Cite{Thm.\,4.14}{FGJS}]\label{thm:bimod_Radford}
Let ${}_\calc\call_\cald$ be an exact bimodule category. There exists a natural isomorphism
  \be
  \mathcal{R}_{\!\call}\Colon \DD_{\!\calc}^{-1} \act \Se^\calc_\call(-)
  \xnatiso \Se^{\overline{\cald}}_\call(-)\actr\DD_ {\!\cald}^{-1}
  \label{eq:bimod_Radford}
  \ee
of twisted bimodule functors.
\end{thm}

As a particular instance, there is a Radford isomorphism for an exact module category
with the relative Serre functor playing the role of the double right dual functor:

\begin{cor}[\Cite{Cor.\ 4.16}{FGJS}]\label{thm:Radford}
% \\
Let $\cala$ be a finite tensor category and $\calm$ an exact $\cala$-mo\-du\-le. There
is a natural isomorphism
  \be
  r_\Calm^{}\Colon \DD_\Cala^{}\act
  {-}\actr\DD_{\!\dualcat}^{-1} \xRightarrow{~\cong~\,} \Se_\calm^\cala\cir\Se_\calm^\cala
  \label{Radford_mod}
  \ee
of twisted bimodule functors.
\end{cor}

%%%%%%%%%%%%%%%%%%%%%%%%%%%%%%%%%%%%%%%%%%%%%%%%%%%%%%%%%%%%%%%%%%%%%%%%

\subsection{Spherical module categories}\label{sec:sph_module}

The property of sphericality of a pivotal tensor category is defined \cite{DSPS}
by means of the Radford isomorphism \eqref{Radford_can}, under the assumption of
unimodularity. In the semisimple case this notion is equivalent to trace-sphericality
\Cite{Prop.\,3.5.4}{DSPS}, i.e.\ to the property that right and left traces of 
endomorphisms coincide. 

\begin{defi}[\Cite{Def.\,3.5.2}{DSPS}] \label{spherical_tc}
A unimodular pivotal tensor category $\calc$ is called 
\emph{spherical} iff the diagram 
  \be
  \begin{tikzcd}[row sep=2.5em,column sep=1.8em]
  \id_\calc \ar[rr,"\overline{r}_\calc^{}",Rightarrow] \ar[rd,"\rm{p}",swap,Rightarrow] &~
  & (-)^{\vee\vee\vee\vee}
  \\
  ~& (-)\dd \ar[ru,"\!\rm{p}\dd",swap,Rightarrow] &~
  \end{tikzcd}
  \ee
commutes, where $\rm{p}$ is the pivotal structure of $\calc$.
\end{defi}

As it turns out, this notion of sphericality is a pivotal Morita invariant:

\begin{prop}[\Cite{Cor.\,5.19}{FGJS}]
Let $\cala$ and $\calb$ be two pivotal Morita equivalent pivotal tensor categories. 
$\cala$ is $($unimodular$)$ spherical if and only if $\calb$ is $($unimodular$)$ spherical.
\end{prop}

In a similar manner, we define sphericality for pivotal bimodule categories:

\begin{defi}[\Cite{Def.\ 5.20}{FGJS}](Spherical bimodule category)\label{spherical_bimod}~\\
Let $\calc$ and $\cald$ be $($unimodular$)$ spherical tensor categories and 
$\textbf{u}_\calc:\textbf{1}\xcong\DD_\Calc^{-1}$ and 
$\textbf{u}_\cald:\textbf{1}\xcong\DD_\Cald^{-1}$ trivializations of the 
corresponding distinguished invertible objects.
 %\\
A pivotal bimodule category ${}_\Calc\call_\cald$ is called 
$(\textbf{u}_\calc,\textbf{u}_\cald)$-\emph{spherical}, or just
\emph{spherical}, iff the diagram
  \be
  \begin{tikzcd}[row sep=2.5em,column sep=1.8em]
  \Se^\calc_\call\ar[rr,"\overline{\mathcal{R}}_{\!\call}",Rightarrow]&~
  & \Se^{\overline{\cald}}_\call
  \\
  ~ & \id_\call
  \ar[ul,"\tilde{\rm{p}}",Rightarrow] \ar[ur,"\tilde{\rm{q}}",Rightarrow,swap] &~
  \end{tikzcd}
  \label{eq:def:sphericalbimodule}
  \ee
commutes, where $\overline{\mathcal{R}}_{\!\call}$ is the composite
  \be
  \Se^\calc_\call
  \xRightarrow{~\textbf{u}_\calc\act\id~} \DD_{\!\calc}^{-1} \act \Se^\calc_\call
  \xRightarrow{~\;\mathcal{R}_{\!\call}\;~}
  \Se^{\overline{\cald}}_\call(-)\actr\DD_ {\!\cald}^{-1} 
  \xRightarrow{~\id\actr\textbf{u}_\cald^{-1}~} \Se^{\overline{\cald}}_\call
  \ee
and $\tilde{\rm{p}}$ and $\tilde{\rm{q}}$ are the pivotal structures of
${}_\calc\call$ and $\call_\cald$, respectively.
\end{defi}

\begin{rem}\label{semi_dist_choices}
The notion of sphericality for bimodule categories given in Definition 
\ref{spherical_bimod} is relative to the choice of isomorphisms 
$\textbf{u}_\calc \colon \textbf{1}\Rarr\cong\DD_\Calc^{-1}$ and
$\textbf{u}_\cald \colon \textbf{1}\Rarr\cong\DD_\Cald^{-1}$. As we will see in
Definition \ref{norm_convention}, in the semisimple case there is a special choice 
of such trivializations.
\end{rem}

\begin{prop}[\Cite{Prop.\,5.24}{FGJS}]\label{Fun_spherical}
Let $\calm$ and $\caln$ be spherical $\cala$-module categories. Then the
$(\cala_\caln^*,\dualcat)$-bimodule category $\Fun_\cala(\calm,\caln)$ is spherical.
\end{prop}

Spherical (left) module categories are defined in terms of their associated bimodule 
category:

\begin{defi}[\Cite{Def.\ 5.22}{FGJS}](Spherical module category)\label{spherical_mod}~\\
Let $\calm$ be a pivotal module category over a spherical fusion category $\cala$,
with pivotal structure $\widetilde{\rm{p}}$.
The pivotal module $({}_\cala\calm,\widetilde{\rm{p}})$ is called \emph{spherical} iff the pivotal
$(\cala,\overline{\dualcat})$-bimodule category $\calm$ is spherical, i.e.\ iff
the diagram 
  \be\label{spherical_module}
  \begin{tikzcd}[row sep=2.5em,column sep=1.8em]
 \Se^\cala_\calm  \ar[rr,"\overline{\mathcal{R}}_{\calm}",Rightarrow]  &~
  & \overline{\Se}^{\cala}_\Calm\ar[dl,"\tilde{\rm{p}}",Rightarrow] \\
  ~& \id_\calm  \ar[ul,"\tilde{\rm{p}}",Rightarrow]&~
  \end{tikzcd}
  \ee
commutes.
\end{defi}

Given a spherical fusion category $\cala$, we denote by 
$\textbf{Mod}^{\text{sph}\!}(\cala)$ the full pivotal sub-bicategory of 
$\textbf{Mod}^{\text{piv}\!}(\cala)$ that has spherical $\cala$-module categories as
objects. This pivotal sub-bicategory is sufficient for witnessing pivotal Morita
equivalence of spherical fusion categories:

\begin{prop}\label{bicategories_spherical_modules}
Two spherical fusion categories $\cala$ and $\calb$ are pivotal 
Morita equivalent if and only if $\mathbf{Mod}^{\rm{sph}\!}(\cala)$ and 
$\mathbf{Mod}^{\rm{sph}\!}(\calb)$ are equivalent as pivotal bicategories.
\end{prop}

\begin{proof}
This is seen in complete analogy to the proof of Theorem 5.12 of \cite{FGJS}. Given
an indecomposable spherical $\cala$-module category $\calm$, it follows from
Proposition \ref{Fun_spherical} that $\Fun_\cala(\calm,\caln)$ is spherical for every
$\caln\iN\textbf{Mod}^{\text{sph}\!}(\cala)$, and thus we have an assignment
  \be
  \begin{aligned}
  \Psi\Colon \textbf{Mod}^{\text{sph}\!}(\cala) & 
  \longrightarrow\textbf{Mod}^{\text{sph}\!} (\overline{\dualcat}) \,,
  \nxl1
  \caln & \longmapsto \Fun_\cala(\calm,\caln) 
  \end{aligned}
  \ee
that is a $2$-equivalence \Cite{Thm.\,7.12.16}{EGno} and preserves the pivotal
structure \Cite{Thm.\,5.12}{FGJS}. Conversely, given a pivotal $2$-equivalence
  \be
  \Phi\Colon \textbf{Mod}^{\text{sph}\!}(\calb) 
  \longrightarrow\textbf{Mod}^{\text{sph}\!} (\cala) \,,
  \ee
the image of the regular spherical $\calb$-module category ${}_\calb\calb$ under
$\Phi$ gives the desired pivotal $\cala$-mo\-du\-le category 
$\calm \,{;=}\, \Phi(\calb)$. Moreover, the pivotality of $\Phi$ ensures that the
tensor equivalence $\Phi_{\calb,\calb}\colon \overline{\calb} \,{\simeq}\,
\Fun_\calb(\calb,\calb) \,{\simeq}\, \Fun_\cala(\calm,\calm) \eq \dualcat$ is 
pivotal, and hence $\cala$ and $\calb$ are pivotal Morita equivalent.
\end{proof}

%%%%%%%%%%%%%%%%%%%%%%%%%%%%%%%%%%%%%%%%%%%%%%%%%%%%%%%%%%%%%%%%%%%%%%%%

\section{Categories with traces and dimensions}\label{sec:traces}

Suitable traces for the endomorphisms of a category allow one to assign scalars to
closed graphs on spheres. This feature is a crucial ingredient in the computation of 
topological invariants, both in the Turaev-Viro state sum construction for the
case of semisimple categories and for the invariants studied in \cite{gePa5,gepV} 
for non-semisimple categories. In the bicategorical state sum construction in 
Section \ref{sec:state_sum} this ingredient will be instrumental as well.

In the present section we study traces on bimodule categories. For the sake of
generality, and to exhibit the relevant structures more clearly, contrary to what
we will do in Section \ref{sec:state_sum}, we do not impose the condition that the
bimodule categories under consideration are semisimple. We also analyze 
particular properties of the traces in the case that the bimodule categories admit
pivotal structures. Additionally, we explore the properties of the traces constructed
for $2$-endomorphisms in the pivotal bicategory $\mathbf{Mod}^{\rm sph}(\cala)$
of spherical module categories over a spherical fusion category $\cala$. 
This will allow us to evaluate, in Section \ref{sec:state_sum}, 
$\mathbf{Mod}^{\rm sph}(\cala)$-labeled graphs on spheres to produce a manifestly
Morita-balanced topological invariant.

We will make use of the graphical string calculus in bicategories and in pivotal
bicategories, following the conventions that we summarize in Appendix \ref{app_graph}.
The graphical calculus will in particular be used in the context of a tensor category,
seen as a one-object bicategory.

%%%%%%%%%%%%%%%%%%%%%%%%%%%%%%%%%%%%%%%%%%%%%%%%%%%%%%%%%%%%%%%%%%%%%%%%

\subsection{The canonical Nakayama-twisted trace of a linear category}

Non-semisimple \ko-linear categories may not admit non-degenerate traces defined 
for all endomorphisms \Cite{Ex.\,3.12}{He}.
In contrast, once we restrict to a particular subclass of morphisms,
there is a canonical trace structure on any finite \ko-linear category. 

Let $\calx$ be a finite \ko-linear category. Denote by ${\rm Proj}\,\calx$ the full
subcategory of projective objects in $\calx$, and by $\Irr\,\calx$ a set of
representatives for the equivalence classes of simple objects in $\calx$.
Recall that every finite $\ko$-linear category comes with a distinguished
right exact endofunctor, the Nakayama functor \eqref{Nakayama}. An important feature
of Nakayama functors is the existence of canonical natural isomorphisms 
  \be
  \naki_{p;x} \Colon \Hom_\calx(p,x)\xcong \Hom_\calx (x,\NakR_\calx(p))^*
  \quad \text{ for }~ p\iN{\rm Proj}\,\calx \;\text{ and }~ x\iN\calx \,,
  \label{Nak_proj_CY}
  \ee
which can be derived from the coend description of $\NakR_\calx$ 
\Cite{Cor.\,2.3}{SchW}.

\begin{defi}[\Cite{Def.\ 2.4}{SchW} and \Cite{Def.\ 4.4}{shShi}]\label{def:tpX}
Let $\calx$ be a finite $\ko$-linear category. 
For every projective object $p\iN {\rm Proj}\,\calx$ the linear map
  \be
  \mathbf{t}^\calx_p := \naki_{p;p}(\id_p) \Colon \Hom_\calx\left(p,\NakR_\calx(p)\right)
  \rarr~ \ko \,,
  \label{twist_trace}
  \ee 
is called the \emph{Nakayama-twisted trace}, or just the \emph{twisted trace}, for 
$p\iN {\rm Proj}\,\calx$.
\end{defi}

The twisted trace of a linear category obeys twisted variants of the properties of
ordinary traces:

\begin{lem}{\rm \Cite{Lemma\,2.5}{SchW}}\label{twisted_trace_properties}
Let $\calx$ be a finite $\ko$-linear category. The twisted trace \eqref{twist_trace} 
has the following properties:
\Enumeratei
\item 
Cyclicity: For $p,q\iN {\rm Proj}\,\calx$ one has
  \be
  \mathbf{t}^\calx_q(f\cir g) = \mathbf{t}^\calx_p\left(\,\NakR_\calx (g)\cir f\,\right)
  \ee
for every $f\iN\Hom_\calx(p,\Nak^r_\calx(q))$ and every $g\iN\Hom_\calx(q,p)$.

\item 
Non-degeneracy: The induced pairing
  \be
  \langle-,-\rangle_\calx^{} \Colon \Hom_\calx(p,x) \otik \Hom_\calx(x,\NakR_\calx(p))
  \rarr~ \ko \,,\quad~ f\oti g\xmapsto{~~} \mathbf{t}^\calx_p(g\cir f)
  \ee
is non-degenerate for every $p\iN {\rm Proj}\,\calx$ and every $x\iN \calx$.
\end{enumerate}
\end{lem}

%%%%%%%%%%%%%%%%%%%%%%%%%%%%%%%%%%%%%%%%%%%%%%%%%%%%%%%%%%%%%%%%%%%%%%%%

\subsection{Nakayama-twisted traces on tensor categories}\label{sec:traces_tensor}

Even in the absence of pivotality, a rigid monoidal
structure on a category permits the definition of \emph{partial traces}. These
lead to the so-called \emph{quantum traces}, which are defined for a certain 
subclass of morphisms for which appropriate instances of double-duals appear in 
the codomain. In this section we explain the interaction of these structures
with the Nakayama-twisted trace.

\begin{defi}
Let $\cala$ be a finite tensor category.
\Enumeratei
\item 
For objects $a,b,c\iN\cala$, the \emph{left partial trace} with respect to $a$ is the map 
  \be
  \label{left_partial}
  \begin{aligned}
  \tr^a_{\rm l}\Colon  \Hom_\cala(a\oti b, \ldd a\oti c)&\rarr~ \,\Hom_\cala(b,c)
   \\
  \hspace*{1.9em}
  \raisebox{-3.2em}{
  \scalebox{0.75}{
  \begin{tikzpicture}[scale=\locscale]
  \draw[line width=\widthObj,\colorObj]
       (0,0) node[below=1pt] {$a$} -- +(0,\locH) node[above=1pt,xshift=-5pt] {$\ldd a$}
       (\locW,0) node[below=-1pt] {$b$} -- +(0,\locH) node[above=1pt,xshift=1pt] {$c$} ;
  \filldraw[line width=\widthMor,fill=white,draw=\colorMor]
       (-\locWM,0.5*\locH-0.5*\locHM) rectangle +(\locW+2*\locWM,\locHM) ;
  \node at (0.5*\locW,0.5*\locH) {$f$} ;
  \end{tikzpicture}
  }}
  ~~ &\longmapsto 
\raisebox{-3.2em}{
\scalebox{0.75}{
  \begin{tikzpicture}[scale=\locscale]
  \draw[line width=\widthObj,\colorObj]
       (\locW,0) node[below=1pt] {$b$} -- +(0,\locH) node[above=-1pt] {$c$}
       (0,\locHs) -- +(0,\locH-2*\locHs)
       (-2*\locC,\locHs) -- +(0,\locH-2*\locHs) node[left=-1pt,midway] {${}^{\vee\!}a$}
       (0,\locH-\locHs) arc (0:180:\locC)
       (0,\locHs) arc (0:-180:\locC) ;
  \scopeArrow{0.5}{\arrowObj}
  \draw[line width=\widthObj,\colorObj,postaction={decorate}]
       (\hsA-\locC,\locH-\locHs+\locC) -- + (0.01,0) ;
  \draw[line width=\widthObj,\colorObj,postaction={decorate}]
       (\hsA-\locC,\locHs-\locC) -- +(0.01,0) ;
  \end{scope}
  \filldraw[line width=\widthMor,fill=white,draw=\colorMor]
       (-\locWM,0.5*\locH-0.5*\locHM) rectangle +(\locW+2*\locWM,\locHM) ;
  \node at (0.5*\locW,0.5*\locH) {$f$} ;
  \end{tikzpicture}
  }
  }
  \end{aligned}
  \ee
\item 
Similarly, the \emph{right partial trace} with respect to $a$ is the map
  \be
  \label{right_partial}
  \begin{aligned}
  {\rm tr}^a_{\rm r}\Colon  \Hom_\cala(b\oti a, c\oti a\dd)&\rarr~ \Hom_\cala(b,c)
   \\
  \raisebox{-3.2em}{
  \scalebox{0.75}{
  \begin{tikzpicture}[scale=\locscale]
  \draw[line width=\widthObj,\colorObj]
       (0,0) node[below=-1pt] {$b$} -- +(0,\locH) node[above=1pt,xshift=-1pt] {$c$}
       (\locW,0) node[below=1pt] {$a$} -- +(0,\locH)
       node[above=1pt,xshift=5pt] {$a^{\vee\vee}$} ;
  \filldraw[line width=\widthMor,fill=white,draw=\colorMor]
       (-\locWM,0.5*\locH-0.5*\locHM) rectangle +(\locW+2*\locWM,\locHM) ;
  \node at (0.5*\locW,0.5*\locH) {$f$} ;
  \end{tikzpicture}
  }}
  ~ &\longmapsto ~~
    \raisebox{-3.2em}{
    \scalebox{0.75}{
  \begin{tikzpicture}[scale=\locscale]
  \draw[line width=\widthObj,\colorObj]
       (0,0) node[below=1pt] {$b$} -- +(0,\locH) node[above=-1pt] {$c$}
       (\locW,\locHs) -- +(0,\locH-2*\locHs)
       (\locW+2*\locC,\locHs) -- +(0,\locH-2*\locHs) node[right=-1pt,midway] {$a^\vee$}
       (\locW,\locH-\locHs) arc (180:0:\locC)
       (\locW,\locHs) arc (180:360:\locC) ;
  \scopeArrow{0.5}{\arrowObj}
  \draw[line width=\widthObj,\colorObj,postaction={decorate}]
       (\locW-\hsA+\locC,\locH-\locHs+\locC) -- + (-0.01,0) ;
  \draw[line width=\widthObj,\colorObj,postaction={decorate}]
       (\locW-\hsA+\locC,\locHs-\locC) -- +(-0.01,0) ;
  \end{scope}
  \filldraw[line width=\widthMor,fill=white,draw=\colorMor]
       (-\locWM,0.5*\locH-0.5*\locHM) rectangle +(\locW+2*\locWM,\locHM) ;
  \node at (0.5*\locW,0.5*\locH) {$f$} ;
  \end{tikzpicture}
  }}
\end{aligned}
\ee

\end{enumerate}
\end{defi}

In a finite tensor category $\cala$ the monoidal unit $\one$ is by definition a
simple object and hence, since \ko\ is algebraically closed,
$\Hom_\cala(\mathbf{1},\mathbf{1})$ is isomorphic to \ko. Thus it is natural to
identify the endomorphisms of the monoidal unit as elements of the base field. There
is a canonical way to make this identification, namely via the linear map (see 
\Cite{Sect.\ 4.2.3}{TV17} and \Cite{Rem.\ 4.7.2}{EGno})
  \be
  \label{can_ident_unit}
  \begin{aligned}
  {\tr^\cala_{\mathbf{1}}}\Colon \Hom_\cala\!\left(\mathbf{1},\mathbf{1}\right)
  & \rarr~ \ko \,,
  \\
  \id_{\mathbf{1}}&\xmapsto{~~~} 1 \,,
  \end{aligned}
  \ee
which is the unique isomorphism of unital $\ko$-algebras 
from $ \Hom_\cala(\mathbf{1},\mathbf{1})$ to
\ko. Left and right $($quantum$)$ traces are 
then defined as the scalars associated to partial traces under this identification:

\begin{defi}{\rm \Cite{Def.\ 4.7.1}{EGno}}\label{quantum_traces}
Let $\cala$ be a finite tensor category.
\Enumeratei
\item The \emph{left $($quantum$)$ trace} of an object $a\iN\cala$ is the map
  \be
  \qtr^\cala_{\lL;\,a}
  \Colon \Hom_\cala(a,\ldd a) \xrightarrow{~\;\tr_{\rm l}^a\;~}
  \Hom_\cala(\mathbf{1},\mathbf{1})\xrightarrow{~\;\tr^\cala_{\mathbf{1}}\;~} \ko \,.
  \ee
\item The \emph{right $($quantum$)$ trace} of $a$ is the map
  \be
  \qtr^\cala_{\rR;\,a}
  \Colon \Hom_\cala(a,a\dd)\xrightarrow{~\;\tr_{\rm r}^a\;~}
  \Hom_\cala(\mathbf{1},\mathbf{1})\xrightarrow{~\;\tr^\cala_{\mathbf{1}}\;~} \ko \,.
  \ee
\end{enumerate}
\end{defi}

In case that $\cala$ admits a pivotal structure, it can be precomposed with the 
quantum traces so as to extend their definition to vector spaces of endomorphisms.

\begin{exa}
Quantum traces and Nakayama-twisted traces are, in general, different. As an
illustration consider the category $\cala \eq \mathrm{Rep}(H)$ with
$H \eq \ko_2[C_2]$ the group algebra of the cyclic group $C_2 \eq \langle g\rangle$ 
of order two over a field $\ko_2$ of characteristic $\rm{char}(\ko_2) \eq 2$. 
The indecomposable projective object $P \eq {}_{H\!}H\iN{\rm Proj}\,\cala$ satisfies 
$P\Vee \eq \Vee P \eq P^*$, and the canonical \ko-linear isomorphism $P\xcong P^{**}$
is an intertwiner. This implies that
  \be
  \qtr^\cala_P := \qtr^\cala_{\lL;\,P} = \qtr^\cala_{\rR;\,P} \Colon
  \Hom_H(P,P)\cong\Hom_H(P,P^{**}) \rarr~ \ko\,.
  \ee
Moreover, a direct computation shows that 
$\qtr^\cala_P(\id_P) \eq 0 \eq \qtr^\cala_P(g.-)$, i.e.\ the quantum trace is
degenerate.
 \\
On the other hand, from \Cite{Lemma\,3.15}{fuScSc} it follows that $\NakR_\cala(P)
\,{\cong}\, H^* \,{\otimes_H}\, P \eq P^*\,{\otimes_H}\, P $. Considering that 
$\Hom_H(P,P^*\,{\otimes_H}\, P ) \, {\cong}\, P^*\,{\otimes_H}\, P$, the
Nakayama-twisted trace is nothing else than the left evaluation on
$P^*\,{\otimes_H}\, P$. Under the identification $P\xcong P^*\,{\otimes_H}\, P $ that
is afforded by the intertwiner that assigns $g\,{\mapsto}\, e^*\otimes e$ and
$e \,{\mapsto}\, g^*\otimes e$ one finds that the twisted trace
  \be
  \mathbf{t}^\cala_P\Colon
  \Hom_H(P,P)\cong\Hom_H\left(P,P^*\,{\otimes_H}\, P\right) \rarr~ \ko 
  \ee
becomes the linear map defined by $\mathbf{t}^\cala_P(\id_P)=0$ and
$\mathbf{t}^\cala_P(g.-)=1$. This
differs from the quantum trace and is, in particular, non-degenerate.
\end{exa}

As illustrated by this example, quantum traces may be degenerate if the underlying 
category is non-semisimple (see also \Cite{Rem.\ 4.8.5}{EGno}). It is therefore 
appropriate to consider instead the $\NakR_\Cala$-twisted trace associated
to the $\ko$-linear category that underlies a finite tensor category $\cala$. We 
can extend the partial traces to morphisms with codomain in the image of the
Nakayama functor \Cite{Def.\,3.3}{SchW}, by means of the twisted $\cala$-bimodule 
functor structure on $\NakR_\cala$. Explicitly, for $a,b\in\cala$ 
the \emph{left partial trace} is given by the composite
  \be
  \tr^a_{\rm l}\Colon \Hom_\cala(a\oti b,\NakR_\cala(a\oti b))
  \xrightarrow{\eqref{module_Nakayama_twisted_functor}\,}
  \Hom_\cala(a\oti b,\! \ldd a\oti\NakR_\cala(b))
  \xrightarrow{\eqref{left_partial}\,} \Hom_\cala(b,\NakR_\cala(b)) \,,
  \ee
while the \emph{right partial trace} is the composite
  \be
  \tr^a_{\rm r}\Colon \Hom_\cala(b\oti a,\NakR_\cala(b\oti a))
  \xrightarrow{\eqref{module_Nakayama_twisted_functor}\,}
  \Hom_\cala(b\oti a,\NakR_\cala(b)\oti a\dd)
  \xrightarrow{\eqref{right_partial}\,} \Hom_\cala(b,\NakR_\cala(b))
  \ee
for $a,b\iN\cala$. If $b$ is projective, then we can post-compose the 
Nakayama-twisted trace \eqref{twist_trace} after such a partial trace to obtain 
scalar values:

\begin{prop}{\rm \Cite{Prop.\,3.5}{SchW}}\label{tensor_partial_trace_prop}\\%works on multi
Let $\cala$ be a finite tensor category and $p\iN {\rm Proj}\,\cala$. The twisted 
trace $\mathbf{t}^\cala_p\colon \Hom_\cala(p,\NakR_\cala(p)) \Rarr~ \ko$
\eqref{twist_trace} satisfies the right and left partial trace properties: We have
  \be
  \mathbf{t}^\cala_{p\otimes a}(f) = \mathbf{t}^\cala_p\,{\rm tr}^a_{\rm r}(f)
  \qquad \text{and} \qquad 
  \mathbf{t}^\cala_{a\otimes p}(g) = \mathbf{t}^\cala_p\,{\rm tr}^a_{\rm l}(g) 
  \ee
for all $a\iN\cala$, $f\iN \Hom_\cala(p\oti a,\,\NakR_\cala(p\oti a))$ 
and $g\iN \Hom_\cala(a\oti p,\,\NakR_\cala(a\oti p))$.
\end{prop}

\begin{proof}
The right partial trace property is shown in \Cite{Prop.\,3.5}{SchW}. The proof
strongly relies on the fact that the Nakayama functor $\NakR_\cala$ comes with a
canonical right twisted module functor structure 
$\NakR_\cala({-}\oti a)\,{\cong}\,\NakR_\cala(-)\oti a\dd$. Mutatis mutandis
the same argument shows the left partial trace property, where now the canonical left
twisted structure $\NakR_\cala(a\ot{-})\,{\cong}\,\ldd a\oti\NakR_\cala(-)$ coming
from \eqref{module_Nakayama_twisted_functor} is in play.
\end{proof}

As pointed out in \Cite{Thm.\,3.6}{SchW} and \Cite{Thm.\,6.8}{shShi}, one can further
study the implications of a choice of pivotal structure on the finite tensor category
$\cala$. We will deal with this issue in Section \ref{sec:traces_module_pivotal}, where we 
consider a more general situation involving bimodule categories.

%%%%%%%%%%%%%%%%%%%%%%%%%%%%%%%%%%%%%%%%%%%%%%%%%%%%%%%%%%%%%%%%%%%%%%%%

\subsection{Nakayama-twisted traces on bimodule categories}\label{sec:traces_module}

In \cite{schaum} the notion of a \emph{module trace} is defined in the setting of
fusion categories. Twisted traces for right module categories beyond the semisimple
setting are studied in \cite{shShi}. In the present section we explore the
Nakayama-twisted trace on bimodule categories and their partial trace properties. 

By definition, for ${}_\calc\call_\cald$ a bimodule category over finite tensor 
categories $\calc$ and $\cald$, the underlying categories are all $\ko$-linear, so 
that in particular each of them comes naturally equipped with Nakayama-twisted
traces. The structure of a bimodule category on $\call$ together with the twisted
bimodule functor structure on $\NakR_\call$ then allow us to define partial traces 
in full analogy to the case of a tensor category: 

\begin{defi}\label{partial_traces_bimod}
Let ${}_\calc\call_\cald$ be a bimodule category. 
\Enumeratei
  \item 
For objects $x\iN\call$ and $c\iN\calc$, the \emph{left partial trace} with respect
to $c$ is defined by the composite
  \be
  \hspace*{-2.4em} \begin{array}{rcl}
  {\rm tr}^c_{\rm l}\Colon \Hom_\call\!\left(c\act x,\NakR_\call(c\act x)\right)
  \xrightarrow{\,\eqref{module_Nakayama_twisted_functor}\,} 
  \Hom_\call(c\act x, \ldd c\act\NakR_\call(x))
  &\hspace*{-0.5em} \longrightarrow \hspace*{-0.5em}& \Hom_\call(x,\NakR_\call(x)) \,,
  \\~\\[-0.7em]
  \raisebox{-3.2em}{ \scalebox{0.75}{
  \begin{tikzpicture}[scale=\locscale]
  \draw[line width=\widthObj,\colorObj]
       (0,0) node[below=1pt] {$c$} -- +(0,\locH)
             node[above=1pt,xshift=-10pt] {${}^{\vee\vee\!}c$}
       (\locW,0) node[below=1pt] {$x$} -- +(0,\locH)
             node[above=-1pt,xshift=10pt] {$\NakR_\call(x)$} ;
  \filldraw[line width=\widthMor,fill=white,draw=\colorMor]
       (-\locWM,0.5*\locH-0.5*\locHM) rectangle +(\locW+2*\locWM,\locHM) ;
  \node[teal] at (0.5*\locW,0.5*\locH) {$f$} ;
  \node[\colorCat] at (-1.5*\locW,0.5*\locH) {$\calc$} ;
  \node[\colorCat] at (0.5*\locW,0.8*\locH) {$\calc$} ;
  \node[\colorCat] at (0.5*\locW,0.2*\locH) {$\calc$} ;
  \node[\colorCat] at (2.5*\locW,0.5*\locH) {$\call$} ;
  \end{tikzpicture}
  }}
  &\hspace*{-0.5em} \longmapsto \hspace*{-0.5em}&
  \raisebox{-3.2em}{ \scalebox{0.75}{  
  \begin{tikzpicture}[scale=\locscale]
  \draw[line width=\widthObj,\colorObj]
       (\locW,0) node[below=1pt] {$x$} -- +(0,\locH) node[above=-1pt] {$\NakR_\call(x)$}
       (0,\locHs) -- +(0,\locH-2*\locHs)
       (-2*\locC,\locHs) -- +(0,\locH-2*\locHs)
             node[left=-1pt,midway,yshift=9pt] {${}^{\vee\!}c$}
       (0,\locH-\locHs) arc (0:180:\locC)
       (0,\locHs) arc (0:-180:\locC) ;
  \scopeArrow{0.5}{\arrowObj}
  \draw[line width=\widthObj,\colorObj,postaction={decorate}]
       (\hsA-\locC,\locH-\locHs+\locC) -- + (0.01,0) ;
  \draw[line width=\widthObj,\colorObj,postaction={decorate}]
       (\hsA-\locC,\locHs-\locC) -- +(0.01,0) ;
  \end{scope}
  \filldraw[line width=\widthMor,fill=white,draw=\colorMor]
       (-\locWM,0.5*\locH-0.5*\locHM) rectangle +(\locW+2*\locWM,\locHM) ;
  \node[teal] at (0.5*\locW,0.5*\locH) {$f$} ;
  \node[\colorCat] at (-1*\locW,0.5*\locH) {$\calc$} ;
  \node[\colorCat] at (-3.5*\locW,0.5*\locH) {$\calc$} ;
  \node[\colorCat] at (2.5*\locW,0.5*\locH) {$\call$} ;
  \end{tikzpicture}
  }}
  \end{array}
  \ee

  \item 
The \emph{right partial trace} with respect to $d\iN\cald$ is given by the composition
  \be
  \hspace*{-2.3em} \begin{array}{rcl}
  {\rm tr}^d_{\rm r}\Colon \Hom_\call(x\actr d,\NakR_\call(x\actr d))
  \xrightarrow{\,\eqref{module_Nakayama_twisted_functor}\,}
  \Hom_\call(x\actr d,\NakR_\call(x)\actr d\dd)
  &\hspace*{-0.5em} \longrightarrow \hspace*{-0.5em}& \Hom_\call(x,\NakR_\call(x)) \,, 
  \\~\\[-0.7em]
  \raisebox{-3.2em}{ \scalebox{0.75}{
  \begin{tikzpicture}[scale=\locscale]
  \draw[line width=\widthObj,\colorObj]
       (0,0) node[below=1pt,xshift=-3pt] {$x$} -- +(0,\locH)
             node[above=-1pt,xshift=-10pt] {$\NakR_\call(x)$}
       (\locW,0) node[below=-1pt,xshift=3pt] {$d$} -- +(0,\locH)
       node[above=1pt,xshift=10pt] {$d^{\vee\vee}$} ;
  \filldraw[line width=\widthMor,fill=white,draw=\colorMor]
       (-\locWM,0.5*\locH-0.5*\locHM) rectangle +(\locW+2*\locWM,\locHM) ;
  \node[teal] at (0.5*\locW,0.5*\locH) {$f$} ;
  \node[\colorCat] at (-1.5*\locW,0.5*\locH) {$\calc$} ;
  \node[\colorCat] at (0.5*\locW,0.8*\locH) {$\call$} ;
  \node[\colorCat] at (0.5*\locW,0.2*\locH) {$\call$} ;
  \node[\colorCat] at (2.5*\locW,0.5*\locH) {$\call$} ;  
  \end{tikzpicture}
  }}
  &\hspace*{-0.5em} \longmapsto \hspace*{-0.5em}&
    \raisebox{-3.2em}{ \scalebox{0.75}{
  \begin{tikzpicture}[scale=\locscale]
  \draw[line width=\widthObj,\colorObj]
       (0,0) node[below=1pt] {$x$} -- +(0,\locH) node[above=-1pt] {$\NakR_\call(x)$}
       (\locW,\locHs) -- +(0,\locH-2*\locHs)
       (\locW+2*\locC,\locHs) -- +(0,\locH-2*\locHs)
            node[right=-1pt,midway,yshift=9pt] {$d^\vee$}
       (\locW,\locH-\locHs) arc (180:0:\locC)
       (\locW,\locHs) arc (180:360:\locC) ;
  \scopeArrow{0.5}{\arrowObj}
  \draw[line width=\widthObj,\colorObj,postaction={decorate}]
       (\locW-\hsA+\locC,\locH-\locHs+\locC) -- + (-0.01,0) ;
  \draw[line width=\widthObj,\colorObj,postaction={decorate}]
       (\locW-\hsA+\locC,\locHs-\locC) -- +(-0.01,0) ;
  \end{scope}
  \filldraw[line width=\widthMor,fill=white,draw=\colorMor]
       (-\locWM,0.5*\locH-0.5*\locHM) rectangle +(\locW+2*\locWM,\locHM) ;
  \node[teal] at (0.5*\locW,0.5*\locH) {$f$} ;
  \node[\colorCat] at (-1.5*\locW,0.5*\locH) {$\calc$} ;
  \node[\colorCat] at (4.2*\locW,0.5*\locH) {$\call$} ;
  \node[\colorCat] at (1.9*\locW,0.5*\locH) {$\call$} ;
  \end{tikzpicture}
  }}
  \end{array}
  \ee
\end{enumerate}
\end{defi}

Note that these bimodule partial traces are defined with the help of evaluation and
coevaluation morphisms for objects in $\calc$ and in $\cald$ and their rigid duals.
In Section \ref{sec:module_duals} we briefly recalled the treatment of dualities for
module categories from \Cite{Sect.\,4}{FGJS}. This ensuing notion of duals allows one
to extend partial traces with respect to objects in invertible bimodule categories.
These are set up by means of the following isomorphisms which involve the
double-duals of objects in the bimodule category, and are analogous to the
twisted module structure of the Nakayama functor:

\begin{lem}
Let ${}_\calc\call_\cald$ be an invertible bimodule category over finite tensor 
categories. There are natural isomorphisms
  \be
  \NakR_\call(c\act y)\cong \NakR_\calc(c)\act y\dd \qquad\text{and}\qquad
  \NakR_\call(y \actr d)\cong \ldd y\actr\NakR_\cald(d) 
  \label{twisted_Nakayamas}
  \ee
for all $c\iN\calc$, $d\iN\cald$ and $y\iN\call$.
\end{lem}

\begin{proof}
For the first isomorphism in \eqref{twisted_Nakayamas} consider the composition
  \be
  \NakR_\call(c\act y) \rarr\cong \ldd c\act\NakR_\call(y)  \rarr\cong
  \ldd c\oti \DD_\calc^{-1}\act\Se^\calc_\call(y)  \rarr\cong \NakR_\calc(c)\act y\dd ,
  \ee
where the first isomorphism comes from the twisted module functor structure 
\eqref{module_Nakayama_twisted_functor} of the Nakayama functor, the second
isomorphism is \eqref{Nakayama_Serre} and the last one uses \eqref{tensor_Nakayama2} 
together with the fact that relative Serre functors are isomorphic to double duals.
The second isomorphism in \eqref{twisted_Nakayamas} is shown analogously.
\end{proof}

Next we define partial traces with respect to objects in invertible bimodule 
categories. These will be crucial for proving Propositions \ref{spherical_traces2}
and \ref{multi_sph_mod} below, which ultimately will allow us to formulate the
bicategorical state sum construction of Section \ref{sec:state_sum}.

\begin{defi}\label{mod_partial_traces}
Let ${}_\calc\call_\cald$ be an invertible bimodule category over finite tensor
categories.
 \Enumeratei
 \item
For objects $y\iN\call$ and $d\iN\cald$, the \emph{left partial trace} 
with respect to $y$ is defined to be the composite
  \be
  \hspace*{-2.3em} \begin{array}{rcl}
  {\rm tr}^y_{\rm l} \Colon \Hom_\call(y\actr d,\NakR_\call(y\actr d))
  \xrightarrow{\eqref{twisted_Nakayamas}}
  \Hom_\call(y\actr d,\ldd y\actr \NakR_\cald(d))
  &\hspace*{-0.5em} \longrightarrow \hspace*{-0.5em}& \Hom_\cald(d,\NakR_\cald(d)) \,,
  \\~\\[-0.7em]
  \raisebox{-3.2em}{ \scalebox{0.75}{
  \begin{tikzpicture}[scale=\locscale]
  \draw[line width=\widthObj,\colorObj]
       (0,0) node[below=1pt] {$y$} -- +(0,\locH)
             node[above=1pt,xshift=-10pt] {${}^{\vee\vee\!}y$}
       (\locW,0) node[below=-1pt] {$d$} -- +(0,\locH)
             node[above=-1pt,xshift=10pt] {$\NakR_\cald(d)$} ;
  \filldraw[line width=\widthMor,fill=white,draw=\colorMor]
       (-\locWM,0.5*\locH-0.5*\locHM) rectangle +(\locW+2*\locWM,\locHM) ;
  \node[teal] at (0.5*\locW,0.5*\locH) {$f$} ;
  \node[\colorCat] at (-1.5*\locW,0.5*\locH) {$\calc$} ;
  \node[\colorCat] at (0.5*\locW,0.8*\locH) {$\call$} ;
  \node[\colorCat] at (0.5*\locW,0.2*\locH) {$\call$} ;
  \node[\colorCat] at (2.5*\locW,0.5*\locH) {$\call$} ; 
  \end{tikzpicture}
  }}
  &\hspace*{-0.5em} \longmapsto \hspace*{-0.5em}&
    \raisebox{-3.2em}{
    \scalebox{0.75}{
  \begin{tikzpicture}[scale=\locscale]
  \draw[line width=\widthObj,\colorObj]
       (\locW,0) node[below=1pt] {$d$} -- +(0,\locH) node[above=-1pt] {$\NakR_\cald(d)$}
       (0,\locHs) -- +(0,\locH-2*\locHs)
       (-2*\locC,\locHs) -- +(0,\locH-2*\locHs)
              node[left=-1pt,midway,yshift=9pt] {${}^{\vee\!}y$}
       (0,\locH-\locHs) arc (0:180:\locC)
       (0,\locHs) arc (0:-180:\locC) ;
  \scopeArrow{0.5}{\arrowObj}
  \draw[line width=\widthObj,\colorObj,postaction={decorate}]
       (\hsA-\locC,\locH-\locHs+\locC) -- + (0.01,0) ;
  \draw[line width=\widthObj,\colorObj,postaction={decorate}]
       (\hsA-\locC,\locHs-\locC) -- +(0.01,0) ;
  \end{scope}
  \filldraw[line width=\widthMor,fill=white,draw=\colorMor]
       (-\locWM,0.5*\locH-0.5*\locHM) rectangle +(\locW+2*\locWM,\locHM) ;
  \node[teal] at (0.5*\locW,0.5*\locH) {$f$} ;
  \node[\colorCat] at (-1*\locW,0.5*\locH) {$\calc$} ;
  \node[\colorCat] at (-3.5*\locW,0.5*\locH) {$\call$} ;
  \node[\colorCat] at (2.5*\locW,0.5*\locH) {$\call$} ;
  \end{tikzpicture}
  }}
  \end{array}
  \ee
 \item 
Similarly, for $y\iN\call$ and $c\iN\calc$ the \emph{right partial trace} with
respect to $y$ is the composite
  \be 
  \hspace*{-2.3em} \begin{array}{rcl}
  {\rm tr}^y_{\rm r}\Colon \Hom_\call(c\act y,\NakR_\call(c\act y))
  \xrightarrow{\eqref{twisted_Nakayamas}}
  \Hom_\call(c\act y, \NakR_\calc(c)\act y\dd )
  &\hspace*{-0.5em} \longrightarrow \hspace*{-0.5em}& \Hom_\calc(c,\NakR_\calc(c)) \,,
  \\~\\[-0.7em]
  \raisebox{-3.2em}{ \scalebox{0.75}{
  \begin{tikzpicture}[scale=\locscale]
  \draw[line width=\widthObj,\colorObj]
       (0,0) node[below=1pt,xshift=-3pt] {$c$} -- +(0,\locH)
             node[above=-1pt,xshift=-10pt] {$\NakR_\calc(c)$}
       (\locW,0) node[below=-1pt,xshift=3pt] {$y$} -- +(0,\locH)
       node[above=1pt,xshift=10pt] {$y^{\vee\vee}$} ;
  \filldraw[line width=\widthMor,fill=white,draw=\colorMor]
       (-\locWM,0.5*\locH-0.5*\locHM) rectangle +(\locW+2*\locWM,\locHM) ;
  \node[teal] at (0.5*\locW,0.5*\locH) {$f$} ;
  \node[\colorCat] at (-1*\locW,0.5*\locH) {$\calc$} ;
  \node[\colorCat] at (0.5*\locW,0.8*\locH) {$\calc$} ;
  \node[\colorCat] at (0.5*\locW,0.2*\locH) {$\calc$} ;
  \node[\colorCat] at (2*\locW,0.5*\locH) {$\call$} ;  
  \end{tikzpicture}
  }}
  &\hspace*{-0.5em} \longmapsto \hspace*{-0.5em}&
  \raisebox{-3.2em}{ \scalebox{0.75}{
  \begin{tikzpicture}[scale=\locscale]
  \draw[line width=\widthObj,\colorObj]
       (0,0) node[below=1pt] {$c$} -- +(0,\locH) node[above=-1pt] {$\NakR_\calc(c)$}
       (\locW,\locHs) -- +(0,\locH-2*\locHs)
       (\locW+2*\locC,\locHs) -- +(0,\locH-2*\locHs)
            node[right=-1pt,midway,yshift=9pt] {$y^\vee$}
       (\locW,\locH-\locHs) arc (180:0:\locC)
       (\locW,\locHs) arc (180:360:\locC) ;
  \scopeArrow{0.5}{\arrowObj}
  \draw[line width=\widthObj,\colorObj,postaction={decorate}]
       (\locW-\hsA+\locC,\locH-\locHs+\locC) -- + (-0.01,0) ;
  \draw[line width=\widthObj,\colorObj,postaction={decorate}]
       (\locW-\hsA+\locC,\locHs-\locC) -- +(-0.01,0) ;
  \end{scope}
  \filldraw[line width=\widthMor,fill=white,draw=\colorMor]
       (-\locWM,0.5*\locH-0.5*\locHM) rectangle +(\locW+2*\locWM,\locHM) ;
  \node[teal] at (0.5*\locW,0.5*\locH) {$f$} ;
  \node[\colorCat] at (-1.5*\locW,0.5*\locH) {$\calc$} ;
  \node[\colorCat] at (4.2*\locW,0.5*\locH) {$\calc$} ;
  \node[\colorCat] at (1.9*\locW,0.5*\locH) {$\call$} ;
  \end{tikzpicture}
  }}
  \end{array}
  \ee
\end{enumerate}
\end{defi}

\begin{prop}\label{bimod_partial_trace_prop}
Let ${}_\calc\call_\cald$ be a bimodule category over finite tensor categories. 
  \Enumeratei
  \item 
Let $p\iN {\rm Proj}\,\call$ be a projective object, and $c\iN\calc$ and 
$d\iN\cald$ arbitrary objects. The twisted trace \eqref{twist_trace}
satisfies right and left partial trace properties: we have
  \be
  \mathbf{t}^\call_{p\Actr d}(f) = \mathbf{t}^\call_p\,{\rm tr}^d_{\rm r}(f)
  \quad~\text{for}\quad f\iN \Hom_\call(p\actr d,\,\NakR_\call(p\actr d))
  \label{partial_trace_proj}
  \ee
and
  \be 
  \mathbf{t}^\call_{c\Act p}(g) = \mathbf{t}^\call_p\,{\rm tr}^c_{\rm l}(g)
  \quad~\text{for}\quad g\iN \Hom_\call(c\act p,\,\NakR_\call(c\act p)) \,.
  \label{partial_trace_proj1}
  \ee
 % \\
  \item 
If $\call$ is invertible, then we further have
  \be
  \label{lpartial_trace_exact}
  \mathbf{t}^\call_{y\Actr q}(f) = \mathbf{t}^\cald_q\,{\rm tr}^y_{\rm l}(f)
  \quad~\text{for}\quad f\iN \Hom_\call(y\actr q,\,\NakR_\call(y\actr q)) 
  \ee
and
  \be
  \label{rpartial_trace_exact}
  \mathbf{t}^\call_{p\Act y}(g) = \mathbf{t}^\calc_p\,{\rm tr}^y_{\rm r}(g)
  \quad~\text{for}\quad g\iN \Hom_\call(p\act y,\,\NakR_\call(p\act y))
  \ee
for all objects $y\iN\call$ and projective objects $p\iN {\rm Proj}\,\calc$ and
$q\iN {\rm Proj}\,\cald$.
\end{enumerate}
\end{prop}

\begin{proof}
The statements in part (i) follow analogously as in the proof of Proposition
\ref{tensor_partial_trace_prop}. The equalities
\eqref{partial_trace_proj} and \eqref{partial_trace_proj1} are obtained by using the same argument on the twisted structure of $\NakR_\call$ given by
\eqref{module_Nakayama_twisted_functor} (see also \Cite{Lemma\,5.5}{shShi}). 
Part (ii) requires invertibility for two reasons: First,
this implies exactness of $\call$, whereby $p\act y$ and $y\actr q$ are projective,
thus ensuring that the traces $\mathbf{t}^\call_{y\Actr q}(f)$ and
$\mathbf{t}^\call_{p\Act y}(g)$ in \eqref{lpartial_trace_exact} and
\eqref{rpartial_trace_exact} are defined. Second, invertibility allows one to define
partial traces for objects in $\call$ as described in Definition
\ref{mod_partial_traces}. Again the statement can be proved by means of the argument
in Proposition \ref{tensor_partial_trace_prop} mutatis mutandis, where the
isomorphisms \eqref{twisted_Nakayamas} play the role of the twisted structure
of the Nakayama functor.
\end{proof}

%%%%%%%%%%%%%%%%%%%%%%%%%%%%%%%%%%%%%%%%%%%%%%%%%%%%%%%%%%%%%%%%%%%%%%%%

\subsection{The role of pivotal structures} \label{sec:traces_module_pivotal}

In order to obtain a trace on endomorphisms of a bimodule category 
${}_\calc\call_\cald$ we need to trivialize the Nakayama functor $\NakR_\call$.
According to \eqref{Nakayama_Serre} this functor is isomorphic to the two
relative Serre functors of $\call$ up to the action of the respective distinguished
invertible object. The choice of a pivotal structure (if it exists) permits to 
trivialize the relative Serre functors. If we additionally require unimodularity, we
can further choose isomorphisms $\textbf{u}_\calc:\textbf{1}\xcong\DD_\Calc^{-1}$ and
$\textbf{u}_\cald:\textbf{1}\xcong\DD_\Cald^{-1}$ and thus obtain two trivializations
of the Nakayama functor. This procedure leads to two traces on endomorphisms of
projective objects of $\call$. We say that $\call$ is 
$(\textbf{u}_\calc,\textbf{u}_\cald)$-trace-spherical if these two traces 
coincide:

\begin{defi}\label{left_right_traces}%works on tensor
Let $\calc$ and $\cald$ be unimodular pivotal finite tensor categories and
$\textbf{u}_\calc:\textbf{1}\xcong\DD_\Calc^{-1}$ and
$\textbf{u}_\cald:\textbf{1}\xcong\DD_\Cald^{-1}$ be trivializations of the
corresponding distinguished invertible objects.  Let ${}_\calc\call_\cald$
be a pivotal bimodule category with pivotal structures 
$\widetilde{\rm{p}} \colon \id_\call\,{\xRightarrow{\,\cong\,}}\,\Se_\call^{\calc}$ and 
$\widetilde{\rm{q}}: \id_\call\,{\xRightarrow{\,\cong\,}}\,\Se_\call^{\overline{\cald}}$.
  \Enumeratei
  \item 
The \emph{left trace} \rm on $\call$ is the linear map defined as the composite
  \be
  \label{un_left_trace_bimod}
  \tr^\call_{\lL;\,p} \Colon \Hom_\call(p,p)
  \xrightarrow{\,\textbf{u}_{\scriptscriptstyle\calc}\act\tilde{\rm{p}}\circ-\,}
  \Hom_\call(p,\DD^{-1}_\calc\Act\Se_\call^\calc(p))
  \xrightarrow{\eqref{Nakayama_Serre}}\Hom_\call(p,\NakR_\call(p))
  \xrightarrow{\eqref{twist_trace}} \ko
  \ee
for any projective object $p\iN {\rm Proj}\,\call$. 
\item Similarly, the \emph{right trace} for $p\iN {\rm Proj}\,\call$ is given by
  \be
  \label{un_right_trace_bimod}
  \hspace*{-0.4em}
  \tr^\call_{\rR;\,p}\Colon \Hom_\call(p,p)
  \xrightarrow{\tilde{\rm{q}}\actr\textbf{u}_{\scriptscriptstyle{\cald}}\circ-\,}
  \Hom_\call(p,\Se_\call^{\overline{\cald}}(p)\Actr\DD^{-1}_\cald)
  \xrightarrow{\eqref{Nakayama_Serre}}\Hom_\call(p,\NakR_\call(p))
  \xrightarrow{\eqref{twist_trace}}\ko\,.
  \ee

\item The pivotal bimodule category ${}_\Calc\call_\cald$ is called 
$(\textbf{u}_\calc,\textbf{u}_\cald)$-\emph{trace-spherical}, or just
\emph{trace-spherical}, iff 
\be
  \label{trace_spherical}
  \tr^\call_{\lL;\,p}(f) = \tr^\call_{\rR;\,p}(f) =: {\rm tr}^\call_{\;p}(f)
  \ee
for every projective $p\in\Proj\call$ and every endomorphism $f\iN\Hom_\call(p,p)$.
\end{enumerate}
\end{defi}

Recall from Definition \ref{spherical_bimod} that the notion of bimodule sphericality
is defined in terms of the bimodule Radford isomorphism \eqref{eq:bimod_Radford}. It
turns out that trace-sphericality is implied by sphericality:

\begin{thm}\label{thm:spherical_traces}
Let ${}_\calc\call_\cald$ be a $(\textbf{u}_\calc,\textbf{u}_\cald)$-spherical 
bimodule category over $($unimodular$)$ spherical tensor categories $\calc$ and
$\cald$. Then ${}_\calc\call_\cald$ is 
$(\textbf{u}_\calc,\textbf{u}_\cald)$-trace-spherical, for the same choice of
the trivializations $\textbf{u}_\calc$ and $\textbf{u}_\cald$.
\end{thm}

\begin{proof}
To establish the equality \eqref{trace_spherical} we must show that the diagram 
  \be
  \begin{tikzcd}[row sep=2.3em,column sep=3.9em]
  & \Hom_\call(p,\DD^{-1}_\calc\Act\Se_\call^\calc(p))
  \ar[xshift=-5pt,yshift=-3pt]{rd}{\eqref{Nakayama_Serre}}
  \ar[dd,"\mathcal{R}_\call\cir-"] &
  \\
  \Hom_\call(p,p) 
  % \ar[u,"\textbf{u}_\calc\Act\id\cir-"]
  \ar[ru,"(\,\textbf{u}_\calc\Act\tilde{\rm{p}}\,)\cir-"]
  \ar[swap,rd,"(\,\tilde{\rm{q}}\Actr \textbf{u}_\cald\,)\cir-"]
  % \ar[d,"\id\Actr\textbf{u}_\cald\cir-",swap]
  & & \hspace*{-2.1em} \Hom_\call(p,\NakR_\call(p))
  \,\xrightarrow{~\eqref{twist_trace}~}\, \ko
  \\
  & \Hom_\call(p,\Se_\call^{\overline{\cald}}(p)\Actr\DD^{-1}_\cald)
  \ar[xshift=-21pt,yshift=-1pt]{ru}[swap]{\eqref{Nakayama_Serre}} &
  \end{tikzcd}    
  \ee
commutes. This is indeed the case: according to the diagram 
\eqref{eq:def:sphericalbimodule}, commutativity of the left triangle is nothing but 
the sphericality condition on the bimodule category $\call$, while commutativity of 
the right triangle is the definition of the Radford isomorphism $\mathcal{R}_\call$.
\end{proof}

\begin{rem}
The notion of trace-sphericality from Definition \ref{left_right_traces}(iii), 
specialized to the case of the regular bimodule category ${}_\calc\calc_\calc$ of a
finite tensor category $\calc$, differs from \Cite{Def.\,3.5.3}{DSPS} in that the
latter is defined in terms of the quantum traces instead of the twisted trace
\eqref{twist_trace}. As pointed out in \Cite{Exs.\,3.5.5\,\&\,3.5.6}{DSPS},
if $\calc$ is not semisimple, then sphericality (in the sense of Definition
\ref{spherical_tc}) neither implies nor is implied by 
quantum-trace-sphericality of $\calc$. In contrast, if
$\calc$ is a fusion category, then the traces \eqref{un_left_trace_bimod} and
\eqref{un_right_trace_bimod} reproduce the classical quantum traces from Definition
\ref{quantum_traces}, and Theorem \ref{thm:spherical_traces} reduces to the statement
that for fusion categories sphericality implies quantum-trace-sphericality. This
illustrates the advantage of working with the twisted trace.
\end{rem}

%%%%%%%%%%%%%%%%%%%%%%%%%%%%%%%%%%%%%%%%%%%%%%%%%%%%%%%%%%%%%%%%%%%%%%%%

\subsection{Semisimple setting and Calabi-Yau categories}

For a finite $\ko$-linear category $\calx$ one can axiomatize a different trace
structure that is defined for every endomorphism of $\calx$, unlike the 
Nakayama-twisted trace. However, this trace only exists if $\calx$ is semisimple.

\begin{defi}{\rm \Cite{Def.\ 2.9}{He}}\label{CY}
\Enumerate
    \item 
A \emph{Calabi-Yau category} $(\calx,{\rm tr}^\calx)$ consists of a finite 
$\ko$-linear category $\calx$ and a family of \emph{traces}, i.e. $\ko$-linear maps
  \be
   {\rm tr}^\calx_x\Colon \Hom_\calx(x,x)\rarr~{}\ko
  \ee
for each $x\iN\calx$ that obey the following conditions:
 \Enumeratei \addtolength\itemsep{5pt}
    \item 
Symmetry: For any pair of objects $x,y\iN\calx$ one has
  \be
  {\rm tr}^\calx_y(f\cir g) = {\rm tr}^\calx_x(g\cir f)
  \label{symm}
  \ee
for every $f\iN \Hom_\calx(x,y)$ and $g\iN \Hom_\calx(y,x)$.
    \item 
Non-degeneracy: The induced pairing
  \be 
  \begin{aligned}
  \langle-,-\rangle_\calx^{} \Colon \Hom_\calx(x,y)\otimes_\ko\Hom_\calx(y,x)
  & \rarr~ \ko \,,
  \\
  f\oti g &\xmapsto{~~~} {\rm tr}^\calx_x(g\cir f)
  \end{aligned}
  \ee
is non-degenerate, i.e.\ the $\ko$-linear map
  \be
  \Hom_\calx(x,y)\,\rarr~ \,\Hom_\calx(y,x)^* , \quad
  f\xmapsto{~~~} {\rm tr}^\calx_x({-}\cir f)
  \ee
is an isomorphism for every pair $x,y\iN\calx$.
\end{enumerate}
\item 
A \emph{Calabi-Yau functor} between Calabi-Yau categories $(\calx,{\rm tr}^\calx)$
and $(\caly,{\rm tr}^\caly)$ is a functor $F\iN\Funre(\calx,\caly)$ satisfying
  \be 
  {\rm tr}^\calx_x(f)={\rm tr}^\caly_{F(x)}(F(f))
  \ee
for all $x\iN\calx$ and $f\iN\Hom_\calx(x,x)$.
\end{enumerate}
\end{defi}

\begin{rem}
The datum of a Calabi-Yau structure on a finite $\ko$-linear category $\calx$ is 
equivalent to the data of a collection of natural isomorphisms
  \be
  \label{CY_Homs}
  \Hom_\calx(x,y) \rarr\cong \Hom_\calx(y,x)^*
  \ee
for all pairs of objects $x,y\iN\calx$ \Cite{Prop.\ 4.1}{schaum}. The existence of
isomorphism \eqref{CY_Homs} implies that the Hom functor is exact and thus $\calx$
is necessarily semisimple.
\end{rem}

In a Calabi-Yau category there is the notion of a \emph{dimension of an object 
$x\iN\calx$}, given by
  \be
  \label{ob_dim} 
  {\dim}(x) := {\rm tr}^\calx_x(\id_x) \,\in\ko \,.
  \ee
The \emph{categorical dimension of $\calx$} is defined as the scalar
  \be
  \label{cat_dim} 
  {\dim}(\calx) := \sum_{x\in\Irr\calx} {\dim}(x)^2 \,\in\ko \,.
  \ee

\begin{lem}\label{semi_CY}
Let $(\calx,\rm{tr}^\calx)$ be a Calabi-Yau category. 
  \Enumeratei
  \item 
Any endomorphism $f\iN\Hom_\calx(h,h)$ of a simple object $h\iN\Irr\,\calx$
can be expressed as
  \be
  f = \frac{{\rm tr}^\calx_h(f)}{\dim(h)}\; \id_h \,.
  \ee
  \label{eq:f-idh}
  \item 
For every object $x\iN\calx$ one has
  \be
  \label{dim_expansion}
  \dim(x) = \sum_{h\in\Irr\calx} \dim(h)\, \dim_\ko(\Hom_\calx(h,x)) \,.
  \ee
\end{enumerate}
\end{lem}

\begin{proof}
Due to simplicity of $h$ we have $f \eq \lambda \,\id_h$ for some $\lambda\iN\ko$,
so that ${\rm tr}^\calx_h(f) \eq \lambda\, \dim(h)$. By the non-degeneracy of the 
trace, $\dim(h)$ is not zero; this implies \eqref{eq:f-idh}. Further, since $\calx$
is semisimple, we can decompose
  \be
  \Hom_\calx(x,x) \cong \bigoplus_{h\in\Irr\calx} \Hom_\calx(x,h) \otik \Hom_\calx(h,x) \,,
  \ee
and thus for each $h\iN\Irr\,\calx$ there are bases $\{\psi_i^*\} \,{\subset}\,
\Hom_\calx(h,x)$ and $\{\psi_{i}\} \,{\subset}\, \Hom_\calx(x,h)$ such that
  \be
  \label{dual_base}
  \psi_{i}\circ \psi_j^* = \delta_{i,j}\,\id_h \qquad \text{and}\qquad
  \sum_{h\in\Irr\calx}\,\sum_i\,  \psi_i^*\circ \psi_{i} = \id_x \,.
  \ee
It follows that
  \be
  \begin{aligned}
  \dim(x) = {\rm tr}^\calx_x(\id_x)
  \overset{\eqref{dual_base}}{=} \sum_{h\in\Irr\calx} \sum_i
  {\rm tr}^\calx_x(\psi_i^*\cir \psi_{i})
  \overset{\eqref{symm}}{=} & \sum_{h\in\Irr\calx}\sum_i {\rm tr}^\calx_h(\psi_{i}\cir\psi_i^*)
  \\
  = & \sum_{h\in\Irr\calx} \dim(h)\, \dim_\ko \Hom_\calx(h,x) \,,   
  \end{aligned}
  \ee
which gives statement (ii).
\end{proof}

We will now focus our attention back to pivotal bimodule categories in the semisimple
setting. We then deal with Calabi-Yau categories, as will be seen below in Theorem
\ref{thm:CY_bimodule}. Let ${}_\calc\call_\cald$ be a pivotal bimodule category, and 
assume in addition that $\calc$, $\cald$ and $\call$ are semisimple. In particular,
$\calc$ and $\cald$ are unimodular and all their objects are projective. Moreover,
there is a specific trivialization of the distinguished invertible object that is 
compatible with the twisted trace associated to the monoidal unit:

\begin{defi}\label{norm_convention}
Let $\calc$ be a fusion category. The \emph{standard trivialization} of the
distinguished invertible object $\DD_\calc$ of $\calc$ is the isomorphism 
$\textbf{s}_{\calc} \colon \mathbf{1}\xcong\NakR_\calc(\mathbf{1})
\,{\equiv}\, \DD_\calc^{-1}$ for which the composite
  \be
  \Hom_\calc\left(\mathbf{1},\mathbf{1}\right)
  \xrightarrow{~\textbf{s}_{{\scriptscriptstyle\calc}}\cir-~}
  \Hom_\calc\left(\mathbf{1},\NakR_\calc(\mathbf{1})\right)
  \xrightarrow{~\eqref{twist_trace}~} \ko
  \ee
is equal to the $\ko$-linear map $\tr^\calc_{\mathbf{1}}$ from Equation
\eqref{can_ident_unit} that assigns $\id_{\mathbf{1}}\,{\xmapsto{~~}}\,1$.
\end{defi}

\begin{rem}\label{can_sph}
Recall that the notion of sphericality on a pivotal bimodule category 
${}_\calc\call_\cald$ from Definition \ref{spherical_bimod} and of trace-sphericality 
from Definition \ref{left_right_traces} depend on the choices 
of trivialization of the distinguished invertible objects of $\calc$ and $\cald$. In
the semisimple setting we will always consider sphericality relative to the standard
trivializations $\textbf{s}_{\calc}$ and 
$\textbf{s}_{\cald}$ as given by Definition \ref{norm_convention}.
The same convention will apply in the case of (left) spherical module categories 
over a spherical fusion category.
\end{rem}

Now since $\call$ is semisimple, every object in $\call$ is projective. As a 
consequence one can take traces for every endomorphism via the standard
trivialization of the distinguished invertible objects.

\begin{thm}\label{thm:CY_bimodule}
Let ${}_\calc\call_\cald$ be a pivotal bimodule category over pivotal fusion 
categories $\calc$ and $\cald$. 
 \Enumeratei
 \item 
The left and right traces ${{\rm tr^\call_\lL}\,}$ 
and ${{\rm tr^\call_\rR}\,}$ from Definition \ref{left_right_traces} endow the 
semisimple category $\call$ with two Calabi-Yau structures.
 \item 
If $\call$ is a spherical bimodule over spherical fusion categories, then the
two Calabi-Yau structures in $(\rm i)$ coincide.
\end{enumerate}
\end{thm}

\begin{proof}
The first statement is an immediate consequence of Lemma
\ref{twisted_trace_properties}, once one notices that, owing to semisimplicity, the
non-degenerate traces are defined for any endomorphism. Statement $(\rm ii)$ follows
from Theorem  \ref{thm:spherical_traces}.
\end{proof}

This result leads in particular to the notion of left and right dimensions for any 
object $x\iN\call$:
  \be
  \dim_{\lL}(x):=\tr^\call_{\lL;\,x}(\id_x) \qquad \text{and} \qquad
  \dim_{\rR}(x):=\tr^\call_{\rR;\,x}(\id_x) \,,
  \ee
which agree if the pivotal structures of $\calc$, $\cald$ and $\call$ are spherical.
% According to Theorem \ref{thm:spherical_traces} the two Calabi-Yau structures agree

\begin{rem}
Provided that the $\ko$-linear category underlying the bimodule category $\call$ is
semisimple, the statement of Theorem \ref{thm:CY_bimodule} holds even in the case
that $\calc$ and $\cald$ are not semisimple.
\end{rem}

\begin{prop}\label{spherical_traces2}
Let ${}_\calc\call_\cald$ be an invertible spherical bimodule category over spherical
fusion categories $\calc$ and $\cald$. One has
  \be
  \label{un_triangle} 
  \def\locC    {0.31}
  \def\locH    {0.4}
  \tr^{\calc}_{\mathbf{1}}\Bigg( \hspace*{-0.3em}
  \raisebox{-1em}{
  \begin{tikzpicture}[scale=\locscale]
  \draw[line width=\widthObjSm,\colorObj]
       (0,0) -- node[left=-2pt,midway,teal] {$\scriptstyle f$} +(0,\locH)
       (2*\locC,0) --  +(0,\locH)
       (0,0) arc (180:360:\locC)
       (0,\locH) arc (180:0:\locC) ;
  \drawPiv {2*\locC}{0.5*\locH} ;
  \scopeArrow{0.5}{\arrowObj}
  \draw[line width=\widthObjSm,\colorObj,postaction={decorate}]
       (\hsA+\locC,\locH+\locC) -- + (0.01,0) ;
  \draw[line width=\widthObjSm,\colorObj,postaction={decorate}]
       (-\hsA+\locC,-\locC) -- +(-0.01,0) ;
  \end{scope}
    \node[\colorCat] at (1*\locC,0.1*\locH) {\tiny $\call\,$} ;
    \node[\colorCat] at (-1.5*\locC,0.5*\locH+1*\locC) {\tiny $\calc$} ;
  \filldraw[black] (0,0.5*\locH) circle (1.5pt);
  \end{tikzpicture}
  }
  \hspace*{-0.2em} \Bigg)
  \;=\;
  \tr^{\call}_{x}\Bigg( \hspace*{-0.6em}
  \raisebox{-1.8em}{
  \begin{tikzpicture}[scale=\locscale]
  \draw[line width=\widthObjSm,\colorObj]
       (0,-\locC) node[below=-2pt]{$\scriptstyle x$} -- node[left=-1.5pt,midway,teal]
       {$\scriptstyle f$} +(0,\locH+2*\locC) node[above=-3pt]{$\scriptstyle x$};
  \node[\colorCat] at (1.2*\locC,0.8*\locH) {\tiny $\call\,$} ;
  \node[\colorCat] at (-1.7*\locC,0.8*\locH) {\tiny $\calc$} ;
  \filldraw[black] (0,0.5*\locH) circle (1.5pt);
  \end{tikzpicture}
  }
  \hspace*{-0.6em} \Bigg)
  \;=\;
  \tr^{\cald}_{\mathbf{1}}\Bigg( \hspace*{-0.3em}
  \raisebox{-1em}{
  \begin{tikzpicture}[scale=\locscale]
  \draw[line width=\widthObjSm,\colorObj]
       (0,0)  -- +(0,\locH)
       (2*\locC,0) -- node[right=-2pt,midway,teal] {$\scriptstyle f$} +(0,\locH)
       (0,0) arc (180:360:\locC)
       (0,\locH) arc (180:0:\locC) ;
  \drawPiv {0}{0.5*\locH} ;
  \scopeArrow{0.5}{\arrowObj}
  \draw[line width=\widthObjSm,\colorObj,postaction={decorate}]
       (-\hsA+\locC,\locH+\locC) -- + (-0.01,0) ;
  \draw[line width=\widthObjSm,\colorObj,postaction={decorate}]
       (\hsA+\locC,-\locC) -- +(0.01,0) ;
  \end{scope}
  \node[\colorCat] at (1*\locC,0.1*\locH) {\tiny $\calc\,$} ;
  \node[\colorCat] at (3.5*\locC,0.6*\locH+1*\locC) {\tiny $\call$} ;
  \filldraw[black] (2*\locC,0.5*\locH) circle (1.5pt);
  \end{tikzpicture}
  }
  \hspace*{-0.5em} \Bigg) 
  \ee
for every $f\iN\Hom_\call(x,x)$ and every object $x\iN\call$, with 
$\tr_\mathbf{1}^\calc$ and
$\tr_\mathbf{1}^\cald$ the $\ko$-algebra isomorphisms given by \eqref{can_ident_unit}.
\end{prop}

\begin{proof}
The equalities \eqref{un_triangle} follow from \eqref{rpartial_trace_exact} and
\eqref{lpartial_trace_exact} as applied to the monoidal units $\mathbf{1}\iN\calc$
and $\mathbf{1}\iN\cald$, which are projective due to semisimplicity.
\end{proof}

%%%%%%%%%%%%%%%%%%%%%%%%%%%%%%%%%%%%%%%%%%%%%%%%%%%%%%%%%%%%%%%%%%%%%%%%

\subsection{Traces in the bicategory of spherical module categories}
\label{traces_modsphA}

We now focus our attention on a specific type of bicategory with traces for 
$2$-endomorphisms that will be relevant in Section \ref{sec:state_sum}. For $\cala$ 
a spherical fusion category, denote by
$\mathbf{Mod}^{\rm sph}(\cala)$ the bicategory that has spherical $\cala$-module
categories as objects, $\cala$-module functors as $1$-morphisms, and module natural
transformations as $2$-morphisms. At the level of Hom-categories, this
bica\-te\-go\-ry is effectively a collection of invertible spherical bimodule 
categories: Given two indecomposable spherical module categories $\calm$ and $\caln$,
the $(\cala_\caln^*,\dualcat)$-bimodule category $\Fun_\cala(\calm,\caln)$ is
invertible and is, by Proposition \ref{Fun_spherical}, naturally endowed with
a spherical structure. Hence according to Theorem \ref{thm:spherical_traces},
$\Fun_\cala(\calm,\caln)$ is a Calabi-Yau category; in other words, the bicategory
$\textbf{\Mod}^{\rm sph}(\cala)$ is locally Calabi-Yau. 

Moreover, as described in Section \ref{sec:PME}, $\textbf{\Mod}^{\rm sph}(\cala)$
is a pivotal bicategory, with pivotal structure
given by \eqref{pivotal_structure_bicategory_of_pivotal_modules}.
We denote the identity module functor $\id_\calm$ of an $\cala$-module $\calm$ 
by $\mathbf{1}_{\!\calm}$, so as to emphasize its role as monoidal unit of $\dualcat$.

\begin{rem}
Recall that in a spherical fusion category $\cala$ the following holds:
 \Enumeratei
 \item 
Dual objects have equal dimensions: $\dim(a^\vee) \eq \dim(a)$ for every
$a\iN\cala$.
 \item
Traces are multiplicative: 
  \be
  \tr^\cala_{a\otimes b}(f\oti g) = \tr^\cala_a(f) \; \tr^\cala_b(g)
  \ee
for all $a,b \iN\cala$, $f\iN\Hom_\Cala(a,a)$ and $g\iN\Hom_\Cala(b,b)$.
 \item 
In particular, dimensions are multiplicative:
  \be
  \dim(a\oti b) = \dim(a)\;\dim(b)
  \ee
for all $a,b \iN\cala$.
\end{enumerate}
\end{rem}

These familiar properties generalize to $\textbf{\Mod}^{\rm sph}(\cala)$:

\begin{prop}\label{multi_sph_mod}
Let $\calm$, $\caln$ and $\call$ be indecomposable spherical module categories over 
a spherical fusion category $\cala$.
 \Enumeratei
 \item For every $H\iN\Fun_\cala(\calm,\caln)$ one has
  \be
  {\rm dim}(H\la) = {\rm dim}(H) \,.
  \ee
 \item
For $H_1\iN\Fun_\cala(\calm,\caln)$, $H_2\iN\Fun_\cala(\caln,\call)$ and module 
natural transformations $\alpha\colon H_1\,{\xRightarrow{~~}}\, H_1$ and 
$\beta\colon H_2 \,{\xRightarrow{~~}}\, H_2$ one has
  \be
  {\rm tr}^{\Fun_{\!\cala}\!(\calm,\call)}_{H_2\circ H_1}(\beta\cir\alpha)
  = {\rm tr}^{\Fun_{\!\cala}\!(\caln,\call)}_{H_2}(\beta) \cdot
  {\rm tr}^{\Fun_{\!\cala}\!(\calm,\caln)}_{H_1}(\alpha) \,,
  \ee
i.e.\ the horizontal composition functor
  \be
  \circ\Colon\Fun_\cala(\caln,\call) \boxtimes \Fun_\cala(\calm,\caln)
  \rarr~ \Fun_\cala(\calm,\call)
  \ee
is a Calabi-Yau functor.
 \item 
For any two module functors $H_1\iN\Fun_\cala(\calm,\caln)$ and
$H_2\iN\Fun_\cala(\caln,\call)$ one has
  \be
  \dim(H_2\cir H_1) = \dim(H_2) \cdot \dim(H_1) \,.
  \ee
 \end{enumerate}
\end{prop}

\begin{proof}
Assertion (i) follows by the computation
   \def\locC    {0.31}  %  radius circle
   \def\locH    {0.6}   %  height
  \be
  \begin{aligned}
  \dim(H)={\rm tr}^{\Fun_{\!\cala}\!(\calm,\caln)}_{H}\left(\id_H\right)
  ={\rm tr}^{\cala_{\!\caln}^*}_{\mathbf{1}_{\!\caln}}
  \Big( \hspace*{-0.7em}
  \scalebox{0.75}{
  \raisebox{-1.2em}{ \begin{tikzpicture}[scale=\locscale]
  \draw[line width=\widthObjSm,\colorObj]
       (0,0) -- node[left=-3pt,midway] {$\scriptstyle H$} +(0,\locH) 
       (2*\locC,0) -- node[right=-3.3pt,very near start] {$\scriptstyle H\la$}
                      node[right=-3.3pt,very near end] {$\scriptstyle H\ra$} +(0,\locH)
       (0,0) arc (180:360:\locC)
       (0,\locH) arc (180:0:\locC) ;
  \drawPiv{2*\locC}{0.5*\locH} ;
    \node[\colorCat] at (-1.5*\locC,0.8*\locH) {\tiny $\caln$} ;
    \node[\colorCat] at (0.9*\locC,0.3*\locH) {\tiny $\calm$} ;
  \scopeArrow{0.5}{\arrowObj}
  \draw[line width=\widthObjSm,\colorObj,postaction={decorate}]
       (\hsA+\locC,\locH+\locC) -- + (0.01,0) ;
  \draw[line width=\widthObjSm,\colorObj,postaction={decorate}]
       (-\hsA+\locC,-\locC) -- +(-0.01,0) ;
  \end{scope}
  \end{tikzpicture}
  }}
  \hspace*{-0.5em} \Big)
  & = {\rm tr}^{\cala_{\!\caln}^*}_{\mathbf{1}_{\!\caln}}
  \Big( \hspace*{-0.5em}
  \scalebox{0.75}{
  \raisebox{-1.2em}{ \begin{tikzpicture}[scale=\locscale]
  \draw[line width=\widthObjSm,\colorObj]
       (2*\locC,0) -- node[right=-3pt,midway] {$\scriptstyle H\la$} +(0,\locH) 
       (0,0) -- node[left=3.3pt,very near start] {$\scriptstyle H$}
                 node[left=-3.3pt,very near end] {$\scriptstyle H\lla$} +(0,\locH)
       (0,0) arc (180:360:\locC)
       (0,\locH) arc (180:0:\locC) ;
  \drawPiv{0*\locC}{0.5*\locH} ;
    \node[\colorCat] at (4*\locC,0.1*\locH) {\tiny $\caln$} ;
    \node[\colorCat] at (1*\locC,0.4*\locH) {\tiny $\calm$} ;
  \scopeArrow{0.5}{\arrowObj}
  \draw[line width=\widthObjSm,\colorObj,postaction={decorate}]
       (\hsA+\locC,\locH+\locC) -- + (0.01,0) ;
  \draw[line width=\widthObjSm,\colorObj,postaction={decorate}]
       (-\hsA+\locC,-\locC) -- +(-0.01,0) ;
  \end{scope}
  \end{tikzpicture}
  }}
  \hspace*{-0.5em} \Big)
  \\[0.8em]
  &={\rm tr}^{\Fun_{\!\cala}\!(\caln,\calm)}_{H\la}\left(\id_{H\la}\right)=\dim(H\la) \,,
  \end{aligned}
  \ee
where the second and fourth equality make use of \eqref{un_triangle}.
 \\[2pt]
Similarly, to prove (ii) we notice that
   \def\locC    {0.2}  %  radius circle
   \def\locH    {0.4}   %  height
  \be
  \begin{aligned}
  {\rm tr}^{\Fun_{\!\cala}\!(\calm,\call)}_{H_2\circ H_1}(\beta\cir\alpha)
  & = {\rm tr}^{\cala_{\!\call}^*}_{\mathbf{1}_{\!\call}}\Bigg( 
  \hspace*{-0.6em}
  \raisebox{-1.5em}{ \begin{tikzpicture}[scale=\locscale]
  \draw[line width=\widthObjSm,\colorObj]
       (0,0) -- node[left=-3pt,midway,teal] {$\scriptstyle \alpha$} +(0,\locH)
       (2*\locC,0) -- +(0,\locH)
       (0,0) arc (180:360:\locC)
       (0,\locH) arc (180:0:\locC) 
       (-1.6*\locC,0) -- node[left=-3pt,midway,teal] {$\scriptstyle \beta$} +(0,\locH)
       (3.6*\locC,0) -- +(0,\locH)
       (-1.6*\locC,0) arc (180:360:2.6*\locC)
       (-1.6*\locC,\locH) arc (180:0:2.6*\locC)
	  ;
  \drawPiv {2*\locC}{0.5*\locH} ;
  \drawPiv {3.6*\locC}{0.5*\locH} ;
  \scopeArrow{0.5}{\arrowObj}
  \draw[line width=\widthObjSm,\colorObj,postaction={decorate}]
       (\hsA+\locC,\locH+\locC) -- + (0.01,0) ;
  \draw[line width=\widthObjSm,\colorObj,postaction={decorate}]
       (-\hsA+\locC,-\locC) -- +(-0.01,0) ;
  \draw[line width=\widthObjSm,\colorObj,postaction={decorate}]
       (\hsA+\locC,\locH+2.6*\locC) -- + (0.01,0) ;
  \draw[line width=\widthObjSm,\colorObj,postaction={decorate}]
       (-\hsA+\locC,-2.6*\locC) -- +(-0.01,0) ;
  \end{scope}
  \node[\colorCat] at (1*\locC,0.1*\locH) {\tiny $\calm\,$} ;
  \node[\colorCat] at (2.6*\locC,-0.1*\locH-0.5*\locC) {\tiny $\caln$} ;
  \node[\colorCat] at (-3.1*\locC,1*\locH+1.3*\locC) {\tiny $\call$} ;
  \filldraw[black] (-1.6*\locC,0.5*\locH) circle (1.2pt) (0,0.5*\locH) circle (1.2pt);
  \end{tikzpicture}
  }
  \hspace*{-0.1em} \Bigg)
  =
   \def\locC    {0.31}  %  radius circle
   \def\locS    {1.4}   %  hor shift
  {\rm tr}^{\cala_{\!\caln}^*}_{\mathbf{1}_{\!\caln}}\Bigg( \hspace*{-0.3em}
  \raisebox{-1.0em}{
  \begin{tikzpicture}[scale=\locscale]
  \draw[line width=\widthObjSm,\colorObj]
       (0,0) -- +(0,\locH)  
       (2*\locC,0) -- node[right=-2pt,midway,teal] {$\scriptstyle \beta$} +(0,\locH) 
       (0,0) arc (180:360:\locC)
       (0,\locH) arc (180:0:\locC) ;
  \drawPiv {0}{0.5*\locH} ;
  \scopeArrow{0.5}{\arrowObj}
  \draw[line width=\widthObjSm,\colorObj,postaction={decorate}]
       (-\hsA+\locC,\locH+\locC) -- + (-0.01,0) ;
  \draw[line width=\widthObjSm,\colorObj,postaction={decorate}]
       (\hsA+\locC,-\locC) -- +(0.01,0) ;
  \end{scope}
  \draw[line width=\widthObjSm,\colorObj]
       (\locS,0) -- node[left=-2pt,midway,teal] {$\scriptstyle \alpha$} +(0,\locH) 
       (\locS+2*\locC,0) -- +(0,\locH) 
       (\locS,0) arc (180:360:\locC)
       (\locS,\locH) arc (180:0:\locC) ;
  \drawPiv {\locS+2*\locC}{0.5*\locH} ;
  \scopeArrow{0.5}{\arrowObj}
  \draw[line width=\widthObjSm,\colorObj,postaction={decorate}]
       (\locS+\hsA+\locC,\locH+\locC) -- + (0.01,0) ;
  \draw[line width=\widthObjSm,\colorObj,postaction={decorate}]
       (\locS-\hsA+\locC,-\locC) -- +(-0.01,0) ;
  \end{scope}
  \node[\colorCat] at (1*\locC,0.1*\locH) {\tiny $\call\,$} ;
  \node[\colorCat] at (1*\locC+0.5*\locS,1*\locH+1*\locC) {\tiny $\caln$} ;
  \node[\colorCat] at (1*\locS+1*\locC,0.1*\locH) {\tiny $\calm$} ;
\filldraw[black] (2*\locC,0.5*\locH) circle (1.2pt)
    (1*\locS,0.5*\locH) circle (1.2pt);
  \end{tikzpicture}
  }
  \hspace*{-0.2em} \Bigg)
  \\[0.8em]
  & =
    \def\locC    {0.31}
  {\rm tr}^{\cala_{\!\caln}^*}_{\mathbf{1}_{\!\caln}}\Bigg( \hspace*{-0.3em}
  \raisebox{-1em}{ \begin{tikzpicture}[scale=\locscale]
  \draw[line width=\widthObjSm,\colorObj]
       (0,0) -- +(0,\locH)
       (2*\locC,0) -- node[right=-3pt,midway,teal] {$\scriptstyle \beta$} +(0,\locH)
       (0,0) arc (180:360:\locC)
       (0,\locH) arc (180:0:\locC) ;
  \drawPiv {0}{0.5*\locH} ;
  \scopeArrow{0.5}{\arrowObj}
  \draw[line width=\widthObjSm,\colorObj,postaction={decorate}]
       (-\hsA+\locC,\locH+\locC) -- + (-0.01,0) ;
  \draw[line width=\widthObjSm,\colorObj,postaction={decorate}]
       (\hsA+\locC,-\locC) -- +(0.01,0) ;
  \end{scope}
  \node[\colorCat] at (1*\locC,0.1*\locH) {\tiny $\call\,$} ;
  \node[\colorCat] at (3.5*\locC,0.5*\locH+1*\locC) {\tiny $\caln$} ;
  \filldraw[black] (2*\locC,0.5*\locH) circle (1.2pt);
  \end{tikzpicture}
  }
  \hspace*{-0.6em} \Bigg) \cdot\,
  {\rm tr}^{\cala_{\!\caln}^*}_{\mathbf{1}_{\!\caln}}\Bigg( \hspace*{-0.7em}
  \raisebox{-1em}{ \begin{tikzpicture}[scale=\locscale]
  \draw[line width=\widthObjSm,\colorObj]
       (0,0) -- node[left=-2pt,midway,teal] {$\scriptstyle \alpha$} +(0,\locH)
       (2*\locC,0) --  +(0,\locH)
       (0,0) arc (180:360:\locC)
       (0,\locH) arc (180:0:\locC) ;
  \drawPiv {2*\locC}{0.5*\locH} ;
  \scopeArrow{0.5}{\arrowObj}
  \draw[line width=\widthObjSm,\colorObj,postaction={decorate}]
       (\hsA+\locC,\locH+\locC) -- + (0.01,0) ;
  \draw[line width=\widthObjSm,\colorObj,postaction={decorate}]
       (-\hsA+\locC,-\locC) -- +(-0.01,0) ;
  \end{scope}
    \node[\colorCat] at (1*\locC,0.1*\locH) {\tiny $\calm\,$} ;
    \node[\colorCat] at (-1.1*\locC,0.5*\locH+1*\locC) {\tiny $\caln$} ;
\filldraw[black] (0,0.5*\locH) circle (1.2pt);
  \end{tikzpicture}
  }
  \hspace*{-0.2em} \Bigg)
  = \,{\rm tr}^{\Fun_{\!\cala}\!(\caln,\call)}_{H_2}(\beta)
  \cdot {\rm tr}^{\Fun_{\!\cala}\!(\calm,\caln)}_{H_1}(\alpha) \,.
  \end{aligned}
  \ee
Here the second equality is \eqref{un_triangle} applied to the $2$-morphism
$\beta \,{\cdot}\, {\rm tr}_{\rm r}^{H_1}(\alpha)$, while the third step is a
consequence of the multiplicativity of ${\rm tr}_{\mathbf{1}_\caln}^{\cala_\caln^*}$.
The first and last steps follow from \eqref{un_triangle} as well.
 \\[2pt]
Statement (iii) directly follows by applying (ii) to the respective identity
$2$-morphisms.
\end{proof}

Further identities valid for fusion categories generalize to
$\textbf{\Mod}^{\rm sph}(\cala)$ as well:

\begin{prop}\label{properties_inv}
Let $\calm$ and $\caln$ be indecomposable spherical mo\-du\-le categories over 
a spherical fusion category $\cala$, and let $H\iN\Fun_\cala(\calm,\caln)$
be a simple module functor.
 \Enumeratei %hold in pivotal
 \item 
Every module natural endotransformation $\alpha\colon H \,{\xRightarrow{\,~}}\, H$
satisfies
  \be
  \alpha = \frac{{\rm tr}^{\Fun_\Cala\!(\!\calm,\caln)}_H(\alpha)}{\dim(H)}\;\id_H \,.
   % =\frac{{\rm tr_L}_H(\alpha)}{\dim_{\rm L}(H)}\,\id_H
  \ee
 \item
For any two module natural transformations $\beta\colon H\,{\xRightarrow{\,~}}\, H$
and $\alpha\colon H \,{\xRightarrow{\,~}}\, H$ one has
  \be
  \dim(H)\cdot {\rm tr}^{\Fun_\Cala\!(\!\calm,\caln)}_H(\beta \,{\cdot}\,\alpha)
  = {\rm tr}^{\Fun_\Cala\!(\!\calm,\caln)}_H(\beta)
  \cdot {\rm tr}^{\Fun_\Cala\!(\!\calm,\caln)}_H(\alpha) \,.
  \ee
 \end{enumerate}
Further, given indecomposable spherical $\cala$-module categories $\calm_1$, 
$\calm_2, ... \,, \calm_{n+1}$ we have:
 \Enumeratei
 \setcounter{enumi}{2}
 \item 
For composable simple module functors $H_1, H_2, ...\,, H_n$, with 
$H_k\iN\Fun_\Cala(\calm_k,\calm_{k+1})$ for $k\eq1,2,...\,,n{-}1$, the equality
   \def\locA    {0.21}  %  hor distance small blobs
   \def\locB    {0.04}  %  radius small blobs
   \def\locCa    {2pt}  %  radius blob
   \def\locD    {0.45}  %  hor distance
   \def\locHa    {1.5}   %  height
  \be
  \label{BK}
  \scalebox{0.75}{ \raisebox{-5.9em}{
    \begin{tikzpicture}[scale=\locscale]
    \draw[line width=\widthObj,\colorObj]
        (-4*\locW,0) node[below=1pt] {$H_n$} -- +(0,2.7*\locHa)
        (2*\locW,0) node[below=1pt] {$H_2$} -- +(0,2.7*\locHa)
        (4*\locW,0) node[below=1pt] {$H_1$} -- +(0,2.7*\locHa);
    \node[\colorCat] at (-5.5*\locW,0.5*2.7*\locHa) {\small$\calm_{n+1}$} ;
    \node[\colorCat] at (-1*\locW,0.5*2.7*\locHa) {\small$\calm_{n}\;~~\cdots~~\;\calm_3$} ;    
    \node[\colorCat] at (3*\locW,0.5*2.7*\locHa) {\small$\calm_2$} ;
    \node[\colorCat] at (5*\locW,0.5*2.7*\locHa) {\small$\calm_1$} ;
    \end{tikzpicture}}}
  ~=~ \sum_{F}\,\dim(F)\,\sum_i
  \scalebox{0.75}{ \raisebox{-5.9em}{
  \begin{tikzpicture}[scale=\locscale]
  \draw[line width=\widthObj,\colorObj]
       (-2.2*\locD,0) node[below=-1pt] {\footnotesize$H_n$} [out=90,in=200] to (0,\locHa)
       (0.5*\locD,0) node[below=-1pt] {\footnotesize$H_2$} [out=90,in=295] to (0,\locHa)
       (1.8*\locD,0) node[below=-1pt] {\footnotesize$H_1$} [out=90,in=340] to (0,\locHa)
       (0,\locHa) -- node[left=2pt,midway] {\footnotesize$F$} +(0,0.7*\locHa)
       (0,1.7*\locHa) [in=270,out=160] to (-2.2*\locD,2.7*\locHa) node[above=-1pt] {\footnotesize$H_n$}
       (0,1.7*\locHa) [in=270,out=65] to (0.6*\locD,2.7*\locHa) node[above=-1pt] {\footnotesize$H_2\,$}
       (0,1.7*\locHa) [in=270,out=20]  to (1.8*\locD,2.7*\locHa) node[above=-1pt] {\footnotesize$H_1$}
  ;
  \fill[\colorObj]
       (0,\locHa) circle (\locCa) node[right=1pt,yshift=3pt,teal] {\small$\psi_i$}
       (0,1.7*\locHa) circle (\locCa) node[right=1pt,yshift=-3pt,teal] {\small$\psi_i^*$} ;
  \fill[black]
       (-1.13*\locD,0.35*\locHa)
       circle (\locB) ++(\locA,0) circle (\locB) +(\locA,0) circle (\locB)
       (-1.13*\locD,2.35*\locHa)
       circle (\locB) ++(\locA,0) circle (\locB) +(\locA,0) circle (\locB) ;
  \node[\colorCat] at (-3*\locD,0.6*2.7*\locHa) {\small$\calm_{n+1}\;$} ;
  \node[\colorCat] at (3*\locD,0.6*2.7*\locHa) {\small$\calm_1$} ;
  \end{tikzpicture}}
}
  \ee
is valid, where the $F$-summation runs over all simple objects in the Calabi-Yau category
$\Fun_\Cala(\calm_1,\calm_{\!n+1})$ and where $\{\psi^*_i\,\}$ and $\{\psi_i\}$ are a
pair of dual bases of the vector spaces $\Nat_{\rm mod}(H_n\,{\circ}\cdots{\circ}\,
H_1, F)$ and $\Nat_{\rm mod}(F,H_n\,{\circ}\cdots{\circ}\, H_1)$.
 \item 
If, moreover, the module categories satisfy $\calm_{\!n+1}\eq \calm_1 \,{=:}\, \calm$,
then for any module natural transformation 
$\eta\iN\Nat_{\rm mod}(H_n \,{\circ}\cdots{\circ}\, H_1, \id_\calm)$ one has
  \be
  \scalebox{0.75}{ \raisebox{-3.7em}{
  \begin{tikzpicture}[scale=\locscale]
  \draw[line width=\widthObj,\colorObj]
       (-2.2*\locD,0) node[below=-1pt] {\footnotesize$H_n$} [out=90,in=200] to (0,2*\locHa)
       (0.5*\locD,0) node[below=-1pt] {\footnotesize$H_2$} [out=90,in=287] to (0,2*\locHa)
       (1.8*\locD,0) node[below=-1pt] {\footnotesize$H_1$} [out=90,in=340] to (0,2*\locHa)  ;
  \fill[\colorObj] (0,2*\locHa) circle (\locCa) ;
  \fill[black] (-1.13*\locD,0.8*\locHa)
       circle (\locB) ++(\locA,0) circle (\locB) +(\locA,0) circle (\locB) ;
    \node[\colorCat] at (-3.3*\locD,0.8*\locHa) {\small$\calm$} ;
    \node[\colorCat] at (3.3*\locD,0.8*\locHa) {\small$\calm$} ;
    \node[teal,above=2pt,right=2pt] at (0,2*\locHa) {\small$\eta$};
  \end{tikzpicture}
  }}
  \quad =~\sum_i ~
\scalebox{0.75}{
  \raisebox{-4.0em}{
  \begin{tikzpicture}[scale=\locscale]
  \draw[line width=\widthObj,\colorObj]
    (-2.2*\locD,0) node[below=-1pt] {\footnotesize$H_n$}  [out=90,in=200] to (0,0.7*\locHa)
       (0.5*\locD,0) node[below=-1pt] {\footnotesize$H_2\,$} [out=90,in=295] to (0,0.7*\locHa)
       (1.8*\locD,0) node[below=-1pt] {\footnotesize$H_1$} [out=90,in=340] to (0,0.7*\locHa)
       (0,1.2*\locHa) [in=270,out=160] to (-2.2*\locD,1.9*\locHa) node[above=-1pt,left=1pt] {\tiny$H_n$}
       (0,1.2*\locHa) [in=270,out=65] to (0.5*\locD,1.9*\locHa) node[above=-1pt,right=-1pt] {\tiny$H_2$}
       (0,1.2*\locHa) [in=270,out=20]  to (1.8*\locD,1.9*\locHa) node[above=-1pt,right=1pt] {\tiny$H_1$}
       (-2.2*\locD,1.9*\locHa)   [out=90,in=200] to (0,2.6*\locHa)
       (0.5*\locD,1.9*\locHa)   [out=90,in=295] to (0,2.6*\locHa)
       (1.8*\locD,1.9*\locHa)   [out=90,in=340] to (0,2.6*\locHa);
  \fill[\colorObj] (0,2.6*\locHa) circle (\locCa) ; 
  \fill[\colorObj]
       (0,0.7*\locHa) circle (\locCa) node[right=2pt,yshift=5pt,teal] {\small$\psi_i$}
       (0,1.2*\locHa) circle (\locCa) node[right=2pt,yshift=-3pt,teal] {\small$\psi_i^*$} ;
  \fill[black]
       (-1.13*\locD,0.17*\locHa)
       circle (\locB) ++(\locA,0) circle (\locB) +(\locA,0) circle (\locB)
       (-1.13*\locD,1.9*\locHa)
       circle (\locB) ++(\locA,0) circle (\locB) +(\locA,0) circle (\locB) ;
  \node[\colorCat] at (-3*\locD,0.9*\locHa) {\small$\calm$} ;
  \node[\colorCat] at (3*\locD,0.9*\locHa) {\small$\calm$} ;
  \node[teal,above=2pt,right=2pt] at (0,2.6*\locHa) {\small$\eta$};
  \end{tikzpicture}}
  }
  \ee
with a pair  $\{\psi^*_i\,\}$ and $\{\psi_i\}$ of dual bases of the vector spaces 
$\Nat_{\rm mod}(H_n \,{\circ}\cdots{\circ}\, H_1, F)$ and 
$\Nat_{\rm mod}(F,H_n \,{\circ}\cdots{\circ}\, H_1)$.
 \end{enumerate}
\end{prop}

\begin{proof}
 \Enumeratei
 \item 
The statement is Lemma \ref{semi_CY}(i) applied to the category
$\Fun_{\!\cala}(\calm,\caln)$ with associated Calabi-Yau structure \eqref{trace_spherical}.
 \item 
is a direct consequence of (i).
 \item 
The statement is a modest generalization of Lemma 1.1 of \cite{BK} to the
bicategorical situation considered here and follows from semisimplicity. Without 
loss of generality, it is enough to consider the case $n \eq 1$ for every simple 
$K\iN\Fun_{\!\cala}(\calm_1,\calm_{n+1})$. In that case $\Nat_{\rm mod}(K, F)$ is 
non-zero only for $F\,{\cong}\, K$, in which case it is one-dimensional; then 
$\psi \,{\circ}\,\psi^* \eq \lambda\,\id_K$ for some $\lambda\iN\ko^\times$. This 
implies $\lambda\,\dim(K) \eq {\rm tr}(\psi \,{\circ}\, \psi^*)
\eq {\rm tr}(\psi^* \,{\circ}\, \psi)=1$, which produces the factor of $\dim(K)$ in
\eqref{BK}. In view of semisimplicity, extending the previous argument to direct sums
implies the full statement.
 \item 
By Schur's lemma we have $\Nat_{\rm mod}(F,\mathbf{1}_\calm) \eq 0$ unless $F$ is
isomorphic to the identity functor. In addition, by the normalization condition we
have $\dim(\mathbf{1}_\calm) \eq 1$; hence the claim follows from (iii).
 \end{enumerate}
  ~\\[-3.1em]
\end{proof}

%%%%%%%%%%%%%%%%%%%%%%%%%%%%%%%%%%%%%%%%%%%%%%%%%%%%%%%%%%%%%%%%%%%%%%%%

\section{A bicategorical state sum construction}\label{sec:state_sum}

In this section we formulate a bicategorical three-dimensional state sum construction
using the pivotal bicategory $\mathbf{Mod}^{\rm sph}(\cala)$ of spherical
$\cala$-module categories as input datum. The main results are Theorem 
\ref{skeleton_inv} in which we prove that 
the bicategorical state sums obtained by our 
construction do not depend on the choice of skeleton that is used in the construction,
and Theorem \ref{thm:TV=St} in which we show that the values of these state sums
reproduce the ones obtained in the original Turaev-Viro invariant based on the
spherical fusion category $\cala$, thereby intrinsically proving that the Turaev-Viro
construction only depends on the pivotal Morita class of a spherical fusion category.

%%%%%%%%%%%%%%%%%%%%%%%%%%%%%%%%%%%%%%%%%%%%%%%%%%%%%%%%%%%%%%%%%%%%%%%%

\subsection{Skeletons on manifolds}\label{sec:primary_moves}

As a preparation, we summarize the pertinent aspects of an auxiliary datum on a
manifold $M$ that is crucial to us: a \emph{skeleton} of $M$, as
introduced in \cite{TV17}.

\begin{defi}{ \Cite{Sect.\,11.1}{TV17}}\label{polyhedron}
\Enumeratei
  \item 
A \emph{graph} \gra\ is a topological space that can be obtained from a disjoint 
union of intervals $[0,1]$ by identifying end points. The images of the intervals 
are called the \emph{edges} of \gra, and the images of the end points are called the
\emph{vertices} of \gra. Each edge connects two (possibly coinciding) vertices,
and each vertex is incident to at least one edge. The images of the half-intervals
$[0,\textstyle \frac{1}{2}]$ and $[\textstyle \frac{1}{2},1]$ in \gra\ are called 
\emph{half-edges}.

  \item 
A \emph{$2$-polyhedron }$S$ is a compact topological space that can be triangulated
using a finite number of simplices of dimension 2 or lower. 
The interior ${\rm Int}(S)$ of $S$ is the set of points in $S$ that have a 
neighborhood homeomorphic to $\mathbb{R}^2$.

  \item 
A \emph{stratification} on a $2$-polyhedron $S$ is a graph $S^{(1)}$ embedded in 
$S$ such that $S \setmin {\rm Int}(S ) \subset S^{(1)}$. 

  \item 
Given a stratified $2$-polyhedron $S$, the edges and vertices of the graph $S^{(1)}$ 
are called the \emph{edges of} $S$ and \emph{vertices of} $S$, respectively. The
sets of edges and vertices of $S$ are denoted by $v(S)$ and $e(S)$, respectively. The
connected components of $S \setmin S^{(1)}$ are called the \emph{regions} of $S$.
The finite set of regions of $S$ is denoted by $\Reg(S)$.

  \item 
A \emph{branch} of a stratified $2$-polyhedron $S$ at an edge $l\iN e(S)$ is a 
germ of a region $r\iN\Reg(S)$ adjacent to $l$, i.e.\ a homotopy class of embeddings
$\sigma \colon [0,1)\Rarr~ S$ such that $\sigma(0)\iN l$ and 
$\sigma((0,1)) \,{\subset}\, S \setmin l$. The set of branches at an edge 
$l\iN e(S)$ is denoted by $S_l$.

  \item 
An \emph{orientation} of a stratified $2$-polyhedron $S$ is an orientation of the
surface $S \setmin S^{(1)}$, i.e.\ a choice of orientation for each region of $S$.

  \item 
The \emph{boundary} $\partial S$ of a stratified $2$-polyhedron $S$ is the union of those edges of $S$ that are adjacent to a single region of $S$.   
\end{enumerate}
\end{defi}

\begin{defi}{\Cite{Sect.\,11.2.1}{TV17}}
Let $M$ be a closed $3$-manifold. A \emph{skeleton} of $M$ is an oriented 
stratified polyhedron $S$ embedded in $M$ such that $\partial S \eq \emptyset$ and 
$M \setminus S$ is a disjoint union of open $3$-balls. Such open $3$-balls are called the
\emph{$3$-cells} of $(M,S)$.
\end{defi}

The set of $3$-cells determined by a skeleton $S$ of $M$ is denoted by $|M{\setminus}S|$. By abuse of notation, the symbol $|M{\setminus}S|$ also stands for the cardinality of this set, i.e.\ for the number of $3$-cells. An example of an skeleton is given by a triangulation of a closed manifold.

Next we summarize results \Cite{Sect.\,11.3.1}{TV17} about local transformations on
skeletons known as \emph{moves}, which are derived from the theory developed in 
\cite{Mat}. Any two skeletons $S$ and $S'$ of a closed $3$-manifold $M$
can be connected by a finite sequence of \emph{primary moves}.
A primary move happens locally on a subset of $M$. In what follows, this subset will 
be represented by a part of a plane consisting of regions adjacent to common 
connected components of $M {\setminus} S$. The primary moves are the following:

\begin{enumerate}
  \item[$T_0\;$:] 
The \emph{bubble move}. This move consists of first removing an open disk from a
region and afterwards gluing a sphere along its equator to the boundary that results
from the removal of the disk; in addition, a new edge along the equator of the 
sphere as well as a new vertex on that edge are added:
  \be
  \raisebox{-2.1em}
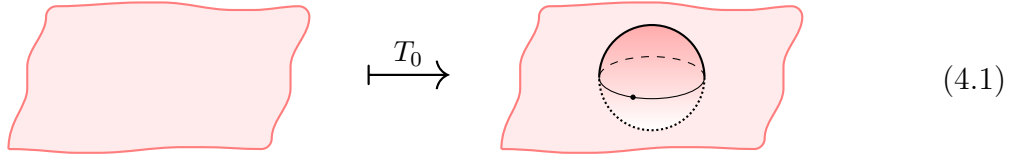
{ \begin{tikzpicture}
  \tikzpicTwocell
  \end{tikzpicture}
  }
  \hspace*{1.4em} \nicemapsto{T_0} \hspace*{1.2em}
  \raisebox{-2.1em}{ \begin{tikzpicture}
  \tikzpicTwocell
  \shade[top color=\colorMd,bottom color=\colorMl]
            (0.5*\stdx+0.5*\stdxx,0.5*\stdy) circle (\radcir);
  \draw[black,thick] (0.5*\stdx+0.5*\stdxx+\radcir,0.5*\stdy) arc (0:180:\radcir);
  \draw[black,thick,densely dotted] (0.5*\stdx+0.5*\stdxx+\radcir,0.5*\stdy) arc (0:-180:\radcir);
  \draw[black] (0.5*\stdx+0.5*\stdxx+\radcir,0.5*\stdy) arc (0:-180:\radcir cm and 0.4*\radcir cm);
  \draw[black,dashed] (0.5*\stdx+0.5*\stdxx+\radcir,0.5*\stdy) arc (0:180:\radcir cm and 0.4*\radcir cm);
  \filldraw[black] (0.5*\stdx,0.5*\stdy-0.372*\radcir) circle (\radvtx) ;
  \end{tikzpicture}
  }
  \ee
%%%%%%%%%%%%%%%%%%%%%%%%%%%%%%%%%%%%%%%
%
Thus we have $|v(S')| \eq |v(S)|\,{+}\,1$ and $|e(S')| \eq |e(S)|\,{+}\,1$. The
two hemispheres of the attached sphere constitute new regions, so that
$|\Reg(S')| \eq |\Reg(S)|\,{+}\,2$, and the $3$-ball inside the sphere is an 
additional $3$-cell, so that $|M{\setminus} S'| \eq |M{\setminus} S|\,{+}\,1$.

  \item[$T_1\;$:] 
The \emph{phantom edge move}. This move adds a new edge by gluing its end points 
to two existing distinct vertices of $S$:
  \be
  \raisebox{-2.1em}
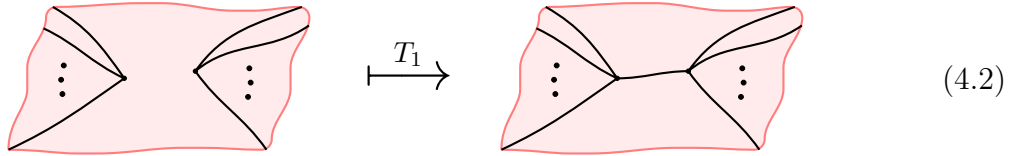
{ \begin{tikzpicture}
  \tikzpicTwocellEE
  \verddots
  \end{tikzpicture}
  }
  \hspace*{1.4em} \nicemapsto{T_1} \hspace*{1.2em}
  \raisebox{-2.1em}{ \begin{tikzpicture}
  \tikzpicTwocellEE 
  \verddots
  \draw[black,thick] (p11) [out=0,in=180] to (p12) ; 
  \end{tikzpicture}
  }
  \ee
Thus we have $|e(S')| \eq |e(S)|\,{+}\,1$. Also, no vertices are added, hence
$|v(S')| \eq |v(S)|$. The number of $3$-cells is preserved as well,
$|M{\setminus} S'| \eq |M{\setminus} S|$. Finally, depending on whether the regions
that are adjacent to the new edge actually are the same region or not, we have either
$|\Reg(S')| \eq |\Reg(S)|$ or $|\Reg(S')| \eq |\Reg(S)|\,{+}\,1$.

    \item[$T_2\;$:] 
The \emph{contraction move}. This collapses an edge in $S$ with distinct end points
to a single vertex:
  \be
  \raisebox{-2.1em}
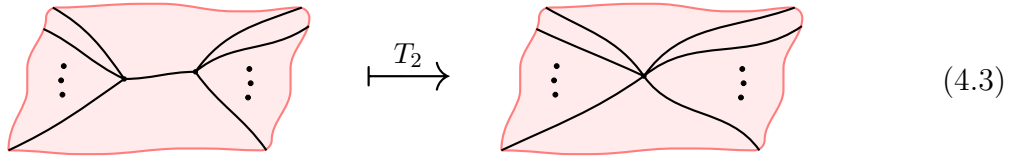
{ \begin{tikzpicture}
  \tikzpicTwocellEE
  \verddots
  \draw[black,thick] (p11) [out=0,in=180] to (p12) ; 
  \end{tikzpicture} 
  }
  \hspace*{1.4em} \nicemapsto{T_2} \hspace*{1.2em}
  \raisebox{-2.1em}{ \begin{tikzpicture}
  \tikzpicTwocellFF 
  \verddots
  \end{tikzpicture}
  }
  \ee
As a consequence, the numbers of vertices and of edges decrease by one, 
$|v(S')| \eq |v(S)|\,{-}\,1$ and $|e(S')| \eq |e(S)|\,{-}\,1$, while the numbers of 
regions and of $3$-cells are preserved, $|\Reg(S')| %\,\
         $\linebreak[0]$
         {=}\, |\Reg(S)|$
and $|M{\setminus} S'| \eq |M{\setminus} S|$.

    \item[$T_3\;$:] 
The \emph{percolation move}. This move pushes a branch across a vertex $v$ in $S$.
In more detail, an open disk whose boundary contains $v$ is removed from a branch, and
then the boundary resulting from the disk removal is glued on another branch at $v$:
  \be
  \raisebox{-2.1em}
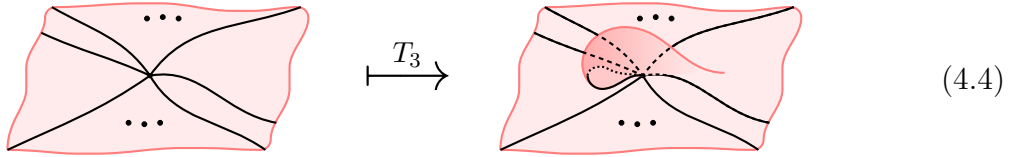
{ \begin{tikzpicture}
  \tikzpicTwocellGG
  \horddots
  \end{tikzpicture}
  }
  \hspace*{1.4em} \nicemapsto{T_3} \hspace*{1.2em}
  \raisebox{-2.1em}{ \begin{tikzpicture}
  \coordinate (p1n)  at (0.26*\stdx+0.5*\stdxx,0.52*\stdy) ;
  \tikzpicTwocellGG
  \horddots
  \tikzpicpercol
  \end{tikzpicture}
  }
  \ee
This move creates a new edge each of whose end points is the vertex $v$; in 
particular, $|e(S')| \eq |e(S)|\,{+}\,1$. 
 \\
The disk bounded by this edge is a new region, hence
$|\Reg(S')| \eq |\Reg(S)|\,{+}\,1$. The numbers of vertices and of $3$-cells do not
change, $|v(S')| \eq |v(S)|$ and $|M{\setminus} S'| \eq |M{\setminus} S|$.
\end{enumerate}

\noindent
For each of the primary moves $T_i$ there is a primary move $T_i^{-1}$ in the 
opposite direction that is inverse to $T_i$.
For instance, the move $T_0^{-1}$ removes a bubble from a skeleton.

\begin{thm}{\rm \Cite{Thm.\,11.1}{TV17}}\label{skeletons}
Any two skeletons of a closed $3$-manifold $M$ can be related by a finite sequence of primary moves.
\end{thm}

%%%%%%%%%%%%%%%%%%%%%%%%%%%%%%%%%%%%%%%%%%%%%%%%%%%%%%%%%%%%%%%%%%%%%%%%

\subsection{Defining the bicategorical state sum}\label{sec:state_sum_inv}

Consider an oriented closed $3$-manifold $M$ with a skeleton $S$. We will
construct a scalar invariant of $M$. We closely follow the procedure described in 
Section 13.1.1 of \cite{TV17}, with the major difference that we work with the pivotal 
bicategory $\mathbf{Mod}^{\rm sph}(\cala)$ for a spherical fusion category $\cala$ 
as an input, which necessitates additional considerations. 

We fix a set $\Irr\,\mathbf{Mod}^{\rm sph}(\cala)$ of representatives for the
equivalence classes of indecomposable spherical $\cala$-module categories (with
equivalence furnished by invertible $1$-morphisms). Similarly,
for every pair $\calm,\caln\in \Irr\,\mathbf{Mod}^{\rm sph}(\cala)$, we fix a set
$\Irr\,\Fun_\cala(\calm,\caln)$ of representatives for the isomorphism classes of
simple module functors. The construction of the bicategorical state sum proceeds
in multiple steps:
\Enumerate
   \item \textbf{Labeling:} 
We define a \emph{labeling} on $(M,S)$ as a pair 
$\varphi \,{:=}\, (\varphi_3,\varphi_2)$ of functions of the following form.
\Enumerateii
   \item 
$3$-labeling: Every $3$-cell in $|M{\setminus}S|$ is labeled by a simple object in
$\Irr\,\mathbf{Mod}^{\rm sph}(\cala)$, i.e.\ by an indecomposable spherical 
$\cala$-module category. This specifies the function
  \be
  \label{label_3}
  \varphi_3\Colon |M{\setminus} S|\rarr~ \Irr\,\mathbf{Mod}^{\rm sph}(\cala)\,.
  \ee

   \item $2$-labeling:
Every region $r\iN \Reg(S)$ is adjacent to two $3$-cells. The orientation on $r$ and
the global orientation of the $3$-manifold $M$ determine uniquely a direction normal 
to $r$, so that it makes sense to say that one $3$-cell $\emph{R}_r$ is located to 
the right of the region $r$ and the other $3$-cell $\emph{L}_r$ is located to its left,
as illustrated in the following picture:
  \def\locH    {0.85}   %  height  
 \be
 \raisebox{-3em}{ \begin{tikzpicture}[scale=2*\locscale]
 \filldraw[draw=\colorL,thick,fill=\colorM]
    (\locH,0.5*\locH) -- (0,0) -- (0,2*\locH) -- ++(\locH,0.5*\locH);
    %orientation manifold
    \draw[->] (-0.5*\locH,2*\locH) to (-0.5*\locH,2.5*\locH);
    \draw[->] (-0.5*\locH,2*\locH) to (-1*\locH,2*\locH);
    \draw[->] (-0.5*\locH,2*\locH) to (-0.75*\locH,2.25*\locH);    
    \node[black] at (-0.75*\locH,2.25*\locH) [anchor=south east] {\scriptsize $3$} ;
    \node[black] at (-1*\locH,2*\locH) [anchor=east] {\scriptsize $2$} ;
    \node[black] at (-0.5*\locH,2.5*\locH) [anchor=south] {\scriptsize $1$} ;
    %orientation region
   \draw[->] (0.25*\locH,1*\locH) to (0.25*\locH,1.25*\locH);
   \draw[->]   (0.25*\locH,1*\locH) to (0.5*\locH,1.125*\locH); 
    \node[black] at (0.25*\locH,1.35*\locH) {\tiny $2$} ;
    \node[black] at (0.6*\locH,1.125*\locH) {\tiny $1$} ;
    %labels
    \node[\colorMlab] at (0.6*\locH,1.8*\locH) {\large $r$} ;
    \node[\colorCat] at (-2*\locC,1*\locH) {\Large $\emph{R}_r$} ;
    \node[\colorCat] at (3*\locC,1*\locH) {\Large $\emph{L}_r$} ;
  \end{tikzpicture}
  }
  \ee
Given a $3$-labeling $\varphi_3$, we label the region $r$ with a simple object $H$
from the set $\Irr\,\Fun_\cala(\varphi_3(\emph{L}_r),\varphi_3(\emph{R}_r))$, i.e.\ 
with a simple $\cala$-module functor between the spherical module categories 
that label the $3$-cells adjacent to $r$. This assignment determines the function
  \be
  \label{label_2}
  \varphi_2\Colon \Reg(S)\rarr~ \bigsqcup_{r\in \Reg(S) \!\! }\Irr\,
  \Fun_\cala\left(\varphi_3(\emph{L}_r),\varphi_3(\emph{R}_r)\right) .
  \ee
\end{enumerate}
We fix for now such a labeling $\varphi \eq (\varphi_3,\varphi_2)$ on $(M,S)$.

   \item 
\textbf{Assignment of vector spaces:} 
For every edge in $S$ we pick an orientation. Each edge comes with two half-edges. We
are going to assign to each half-edge $l \,{\equiv}\, (l,v)$ a vector space $V_l$.
The pair of vector spaces assigned to the half-edges of a given edge remains the 
same if we pick the opposite orientation on said edge. In order to specify the 
assignment we make
use of the set $S_l$ of branches at $l$, see Definition \ref{polyhedron}(v).
\Enumerateii
   \item 
Every branch $r\iN S_l$ at the edge $l$ is contained in a region $\Tilde{r}$ bounding
$l$. We assign to $r$ a label depending on the label of the region $\Tilde{r}$ and on
a sign $\delta_r \iN \{\pm1\}$ that is determined from the relative orientation 
between $r$ and $l$, as indicated in the following pictures:
  \def\locH    {0.85} 
 \be
 \raisebox{-2em}{ \begin{tikzpicture}[scale=1.5*\locscale]
  \filldraw[draw=\colorL,thick,fill=\colorM,line width=\widthObjSm]
       (1.75*\locH,0) -- (0,0) -- (0,1*\locH) -- (1.75*\locH,1*\locH) ;
  \draw[line width=\widthObjSm,\colorCat] (0,0) -- (0,1*\locH) ;
  \filldraw[black] (0,0) circle (1pt);
  \scopeArrow{0.5}{\arrowObj}
  \draw[line width=\widthObjSm,\colorCat,postaction={decorate}]
       (0,0.5*\locH+\hsA) -- + (0,0.01) ;
  \end{scope}    
    %orientation region
   \draw[->] (0.5*\locH,0.25*\locH) to (0.5*\locH,0.5*\locH);
   \draw[->]   (0.5*\locH,0.25*\locH) to (0.75*\locH,0.25*\locH); 
    \node[black] at (0.5*\locH,0.5*\locH)[anchor=south] {\tiny $1$} ;
    \node[black] at (0.75*\locH,0.25*\locH) [anchor=west]{\tiny $2$} ;
    %labels
    \node[\colorMlab] at (1.5*\locH,0.75*\locH) {\small $r$} ;
    \node[black] at (0,0) [anchor=north east] {\footnotesize $v$} ;
    \node[\colorCat] at (0,0.5*\locH) [anchor=east] {\footnotesize $l$} ;
  \end{tikzpicture}
  }
  ~~\delta_r=+1 \hspace*{5em}
  \raisebox{-2em}{ \begin{tikzpicture}[scale=1.5*\locscale]
  \filldraw[draw=\colorL,thick,fill=\colorM,line width=\widthObjSm]
       (1.75*\locH,0) -- (0,0) -- (0,1*\locH) -- (1.75*\locH,1*\locH) ;
  \draw[line width=\widthObjSm,\colorCat] (0,0) -- (0,1*\locH) ;
  \filldraw[black] (0,0) circle (1pt);
  \scopeArrow{0.5}{\arrowObj}
  \draw[line width=\widthObjSm,\colorCat,postaction={decorate}]
       (0,0.5*\locH+\hsA) -- + (0,0.01) ;
  \end{scope}    
    %orientation region
   \draw[->] (0.5*\locH,0.6*\locH) to (0.5*\locH,0.35*\locH);
   \draw[->]   (0.5*\locH,0.6*\locH) to (0.75*\locH,0.6*\locH); 
    \node[black] at (0.5*\locH,0.35*\locH)[anchor=north] {\tiny $1$} ;
    \node[black] at (0.75*\locH,0.6*\locH) [anchor=west]{\tiny $2$} ;
    %labels
    \node[\colorMlab] at (1.5*\locH,0.75*\locH) {\small $r$} ;
    \node[black] at (0,0) [anchor=north east] {\footnotesize $v$} ;
    \node[\colorCat] at (0,0.5*\locH) [anchor=east] {\footnotesize $l$} ;
  \end{tikzpicture}
  }
  ~~\delta_r=-1
  \ee
Denote by $H\iN\Fun_\cala(\varphi_3(\emph{L}_{\Tilde{r}}),\varphi_3(\emph{R}_{\Tilde{r}}))$ 
the label of the region $\Tilde{r}\iN\Reg(S)$; then the branch $r$ is labeled by 
  \be\label{delta_def}
  H^{\delta_r} = 
    \begin{cases}
  H    &\quad\text{if } \delta_r=+1 \,,\\
  H\ra &\quad\text{if } \delta_r=-1 \,.
    \end{cases}
  \ee
(Choosing here right instead of left adjoints is a mere convention because the 
pivotal structure on the bicategory identifies these two options.)
    
    \item 
The global orientation of the $3$-manifold $M$ together with the orientation of $l$
determine an orientation on a circle around the half-edge $l$. This defines a cyclic 
order on the set of branches $S_l \eq \{r_1,...,r_n\}$ that bound $l$.
  \def\locA  {55}  % angle
  \def\locB  {20}  % angle
  \def\locC  {-25} % angle
  \def\locH  {2.5} % height
  \def\locL  {3.5} % length
  \be
  \raisebox{-6.7em}
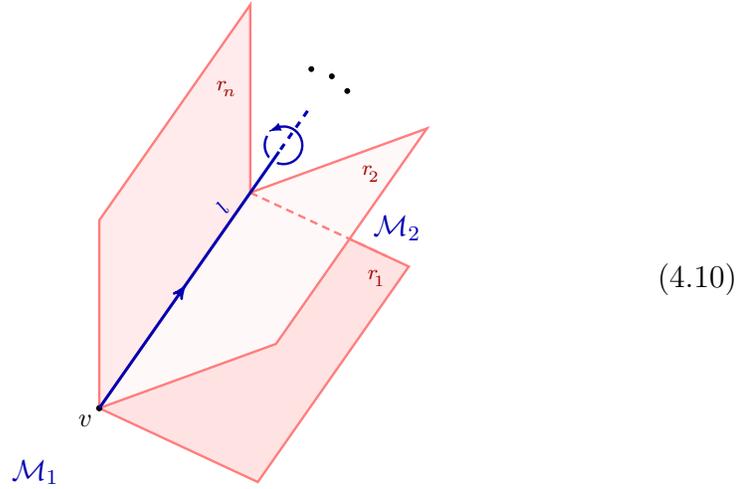
{ \begin{tikzpicture}
   %%%%%orientation manifold
    \scalebox{0.75}{\draw[->] (-0.5*\locH,1*\locH) to (-0.5*\locH,1.5*\locH);
    \draw[->] (-0.5*\locH,1*\locH) to (-1*\locH,1*\locH);
    \draw[->] (-0.5*\locH,1*\locH) to (-0.75*\locH,1.25*\locH);    
    \node[black] at (-0.75*\locH,1.25*\locH) [anchor=south east] {\scriptsize $3$} ;
    \node[black] at (-1*\locH,1*\locH) [anchor=east] {\scriptsize $2$} ;
    \node[black] at (-0.5*\locH,1.5*\locH) [anchor=south] {\scriptsize $1$} ;}
    %%%%%%%%%%%%%%
  \coordinate (bl1) at (58:1.52*\locL) ;
  \coordinate (bl2) at (55:1.54*\locL) ;
  \coordinate (bl3) at (52:1.53*\locL) ;
  \filldraw[draw=\colorL,fill=\colorMA,line width=\widthObjSm]
       (0,0) -- ++ (\locC:0.93*\locH) -- ++ (\locA:\locL) -- ++ (\locC+180:0.93*\locH) -- cycle ; 
  \filldraw[draw=\colorL,fill=\colorM,line width=\widthObjSm]
       (0,0) -- ++ (0,\locH) -- ++ (\locA:\locL) -- ++ (0,-\locH) -- cycle ; 
  \filldraw[draw=\colorL,fill=\colorN,line width=\widthObjSm]
       (0,0) -- ++ (\locB:\locH) -- ++ (\locA:\locL) -- ++ (\locB+180:\locH) -- cycle ; 
  \draw[\colorL,line width=\widthObjSm,densely dashed]
       (\locA:\locL) -- ++ (\locC:0.8*\locH) ;
  \scopeArrow{0.5}{\arrowObj}
  \draw[line width=\widthObjSm,\colorCat,postaction={decorate},rotate=-35]
       (0,0.75*\locH+\hsA) -- + (0,0.01) ;
  \end{scope}
  \filldraw[black] (0,0) circle (1.2*\radvtx)
       (bl1) circle (\radblob) (bl2) circle (\radblob) (bl3) circle (\radblob) ;
  \node[black] at (0.05,0.05) [anchor=north east] {\footnotesize $v$} ;
  \node[\colorMlab] at (0.91*\stdx+\stdxx,0.91*\stdy) {$\scriptstyle r_{\!1}^{}$} ;
  \node[\colorMlab] at (\locL+0.11,1.2*\locH+0.12) {$\scriptstyle r_{\!2}^{}$} ;
  \node[\colorMlab] at (0.5*\locL-0.06,0.5*\locL+\locH) {$\scriptstyle r_{\!n}^{}$} ;
  \node[\colorCat] at (-0.85,-0.85) {\small $\calm_1$} ;
  \node[\colorCat] at (\locL,\locH-0.1) [anchor=west] {\small $\calm_2$} ;   
  \begin{scope}[shift={(\locA:1.22*\locL)}]
       \draw[\colorCat,line width=\widthObjSm]
            (150:0.25) arc (-210:125:0.25) ;
       \draw[\colorCat,line width=0.6*\widthObjSm,->,>=latex]
            (135:0.265) -- +(203:0.001) ;
  \end{scope}
  \draw[white,line width=4*\widthObjSm] (\locA:1.1*\locL) -- (\locA:1.2*\locL) ;
  \draw[\colorCat,line width=1.26*\widthObjSm]
       (0,0) -- (\locA:1.19*\locL) node[near end,above,sloped,yshift=-2pt]{$\scriptstyle l$} ;
  \draw[\colorCat,line width=1.3*\widthObjSm,dash pattern=on 2.6pt off 2.5pt]
       (\locA:1.2*\locL) -- (\locA:1.4*\locL) ;
  \end{tikzpicture}
 %\hspace*{-10em}\includegraphics[width=6cm,center]{Graphs/edge_space1.jpg}
  }
  \ee
We pick a \emph{starting branch} $r_1$ in $S_l$ and compose the labels 
$H_i^{\delta_i}$ of the adjacent branches $r_i$ following the cyclic order on $S_l$.
The resulting composite is an endofunctor of the spherical module category $\calm_1$
that labels the $3$-cell to the left of the starting branch $r_1$. We assign to the 
half-edge $l$ a vector space of 
$2$-morphisms in the pivotal bicategory $\mathbf{Mod}^{\rm sph}(\cala)$:
  % \raisebox{-5em}{\includegraphics[width=5cm]{Graphs/edge_space2.jpg}}\quad\longmapsto \;
  %
   \def\locH    {0.85}   %  height  
   \def\locC    {0.31}
 \be
 \raisebox{-3em}{ \begin{tikzpicture}[scale=2*\locscale]
 \draw[line width=\widthObjSm,\colorObj]
     (0,0.5*\locC) to (0,0.5*\locC+\locH)
     (0.353*\locC,-0.353*\locC) --  +(0.707*\locH,-0.707*\locH);
\draw[line width=\widthObjSm,\colorObj,rotate=-30]
     (0.353*\locC,0.353*\locC) --  +(0.707*\locH,0.707*\locH) ;     
   \scopeArrow{0.3}{\arrowObj}
  \draw[line width=0.7*\widthObjSm,\colorObj,postaction={decorate}]
       (0.5*\locC-0*\hsA,0) -- + (0,0.01) ;
  \draw[line width=0.7*\widthObjSm,\colorObj,postaction={decorate}]
       (-0*\hsA,0.5*\locC+0.5*\locH) -- + (0,-0.01) ;
  \draw[line width=0.7*\widthObjSm,\colorObj,postaction={decorate},rotate=-75]
       (-0*\hsA,0.5*\locC+0.5*\locH) -- + (0,0.01) ;
  \draw[line width=0.7*\widthObjSm,\colorObj,postaction={decorate},rotate=-135]
       (-0*\hsA,0.5*\locC+0.5*\locH) -- + (0,0.01) ;
  \end{scope}
    %orientation manifold
    \draw[->] (-0.5*\locH,1*\locH) to (-0.5*\locH,1.5*\locH);
    \draw[->] (-0.5*\locH,1*\locH) to (-1*\locH,1*\locH);
    \draw[->] (-0.5*\locH,1*\locH) to (-0.75*\locH,1.25*\locH);    
    \node[black] at (-0.75*\locH,1.25*\locH) [anchor=south east] {\scriptsize $3$} ;
    \node[black] at (-1*\locH,1*\locH) [anchor=east] {\scriptsize $2$} ;
    \node[black] at (-0.5*\locH,1.5*\locH) [anchor=south] {\scriptsize $1$} ;

    \draw[black] (0,0) circle (0.5*\locC);
    \filldraw[black] (0,0) circle (0.3pt);    
    %labels
    \node[black] at (-0.49*\locC,0) [anchor=east] {\scriptsize $l$} ;
    \node[black] at (0,0) [anchor=north] {\scriptsize $v$} ;
    \node[\colorObj] at (0,0.48*\locC+\locH) [anchor=south] {\footnotesize $H_n$} ;
    \node[\colorObj] at (\locC+0.707*\locH,0*\locC+0.5*\locH) [anchor=north west]{\footnotesize $H_2$} ;
    \node[\colorObj] at 
    (0.21*\locC+0.707*\locH,-0.3*\locC-0.7*\locH) [anchor=north west]{\footnotesize $H_1$} ;    
    \node[\colorCat] at (-1.8*\locC,-1.1*\locC) {\small $\calm_1$} ;
    \node[\colorCat] at (1.5*\locC,-0.35*\locC) {\small $\calm_2$} ;
    \node[\colorCat] at (0.7*\locC,0.5*\locC+0.85*\locH) {\small $\calm_n$};
  \filldraw[black]
            (40:0.87*\locH) circle (0.5*\radblob)
            (50:0.9*\locH) circle (0.5*\radblob)
            (60:0.87*\locH) circle (0.5*\radblob) ;
  \end{tikzpicture}
  }\quad\longmapsto \;
  \Nat_{\rm mod}\left(
  1_{\!\calm_1},H_n^{\delta_n}\,{\circ}\cdots{\circ}\, H_1^{\delta_1}\right) .
  \ee
The so obtained vector space depends up to canonical isomorphism only on the cyclic
order of $S_l$ and not on the linear order determined by the selected starting branch.
More explicitly, this vector space can be defined as the tip of the limit cone of the diagram defined by the choices of linear orders compatible with the cyclic order of $S_l$ \Cite{Def.\ 2.9}{FuSY}:
    \be
    V_l:= {\rm lim}_{\rho} \;\Nat_{\rm mod}\left(
    1_{\!\calm_{\rho(1)}},H_{\rho(n)}^{\delta_{\rho(n)}} \,{\circ}\cdots{\circ}\, 
    H_{\rho(1)}^{\delta_{\rho(1)}}\right) ,
    \ee
where $\rho$ is a cyclic permutation of the ordered tuple $(1,2,...,n)$.
\end{enumerate}

Every edge in $S$ consists of two half-edges $l$ and $\overline{l}$. By tensoring
over all edges we obtain a vector space
  \be
  V(M,S,\varphi):=\bigotimes_{l\in e(S)}\! V_{l}^{} \otimes_\ko V_{\overline{l}}
  \ee
associated to the labeled skeleton $S$ on $M$.

   \item 
\textbf{Distinguished vector:} For each edge in the skeleton, we can realize the 
vector spaces $V_{l}$ and $V_{\overline{l}}$ associated to its half-edges $l$ and 
$\overline{l}$ as morphism spaces in the dual category $\dualcat$ with respect to 
the label $\calm$ of any $3$-cell adjacent to the starting branch at $l$. Since
$\mathbf{Mod}^{\rm sph}(\cala)$ is locally Calabi-Yau, in particular $\dualcat$ is
Calabi-Yau (see Section \ref{traces_modsphA}), and hence the non-degeneracy of the
traces yields an isomorphism $V_{\overline{l}}\,{\cong}\,V_l^*$. Thus the 
coevaluation morphism determines a non-zero vector $*_l\iN V_{l}\otik V_{\overline{l}}$.
 \\
By tensoring over all edges in $S$ we then obtain a \emph{distinguished vector}:
  \be
  *_\varphi:=\bigotimes_{l\in e(S)} *_l\;\in V(M,S,\varphi)\,.
  \ee

   \item 
\textbf{Evaluation at vertices:} For any given vertex $v\iN v(S)$, consider a small ball
$B_v$ around $v$ in $M$. The intersection $B_v\,{\cap}\, S$ of this ball with the 
skeleton is a closed graph $\mathcal{G}_v$ on the sphere ${\mathtt S}^2_v \eq \partial B_v$.
To see this, denote by $e_v(S)$ the set of half-edges whose end point is the vertex $v$.
The sphere ${\mathtt S}^2_v$ intersects the edges in $e_v(S)$ at pairwise distinct
points and the regions bounding these edges at strands that connect the corresponding
points. Thereby the intersection determines a graph $\mathcal{G}_v$ on ${\mathtt S}^2_v$.
Given a labeling $\varphi$, the labels on the regions of the skeleton $\Reg(S)$
induce a labeling of the edges of $\mathcal{G}_v$, and any vector
$\gamma\iN V(M,S,\varphi)$ induces a labeling of the vertices of $\mathcal{G}_v$. An
intersecting $3$-cell of the
skeleton meets the sphere at a \emph{patch}, i.e.\ at a connected component of 
${\tt S}^2_v\setmin \mathcal{G}_v$. Similarly, the labels of the $3$-cells induce labels
on the patches. Altogether we obtain a closed $\mathbf{Mod}^{\rm sph}(\cala)$-labeled
graph on the sphere. The following picture shows an example of such a graph:
    \def\locR  {2.3}   % radius
  \be
  \raisebox{-4em}{ \begin{tikzpicture}
  \coordinate (p1) at (25:0.65*\locR) ;
  \coordinate (p2) at (105:0.58*\locR) ;
  \coordinate (p3) at (245:0.54*\locR) ;
  \coordinate (q0) at (176:0.65*\locR) ;
  \coordinate (q1) at (25:0.99*\locR) ;
  \coordinate (q2) at (135:\locR) ;
  \coordinate (q3) at (220:\locR) ;
  \shade[top color=gray!22!white,bottom color=white,draw=gray,thin] (0,0) circle (\locR) ;
  \draw[gray,thin] (-\locR,0) arc (180:360:\locR cm and 0.19*\locR cm) ;
  \draw[gray!66!white,thin,densely dashed]
       (-\locR,0) arc (180:0:\locR cm and 0.19*\locR cm) ;
  \scopeArrow{0.43}{\arrowObj} 
  \draw[line width=\widthObj,\colorObj,postaction={decorate}]
       (p1) [out=125,in=20] to node[near end,above,sloped,yshift=-2pt]{$\scriptstyle K$} (p2) ;
  \draw[line width=\widthObj,\colorObj,postaction={decorate}]
       (p3) [out=5,in=260] to node[near start,below,sloped,yshift=2pt]{$\scriptstyle J$} (p1) ;
  \end{scope}
  \scopeArrow{0.88}{\arrowObj} 
  \draw[line width=\widthObj,\colorObj,postaction={decorate}]
       (p3) [out=125,in=230] to node[very near start,below,sloped,yshift=2pt]{$\scriptstyle H$} (p2) ;
  \end{scope}
  \scopeArrow{0.52}{\arrowObj} 
  \draw[line width=\widthObj,\colorObj,postaction={decorate}]
       (p1) [out=40,in=130] to node[midway,above,sloped,yshift=-2pt]{$\scriptstyle G$} (q1) ;
  \end{scope}
  \scopeArrow{0.66}{\arrowObj} 
  \draw[line width=\widthObj,\colorObj,postaction={decorate}]
       (p2) [out=120,in=40] to node[near start,above,sloped,yshift=-2pt]{$\scriptstyle F$} (q2) ;
  \draw[line width=\widthObj,\colorObj,postaction={decorate}]
       (q3) [out=-40,in=240] to node[near start,above,sloped,yshift=-2pt]{$\scriptstyle I$} (p3) ;
  \draw[line width=0.8*\widthObj,\colorObj,densely dashed,postaction={decorate}]
       (q2) [out=220,in=145] to (q0) ;
  \end{scope}
  \draw[line width=0.8*\widthObj,\colorObj,dash pattern=on 2.65pt off 2.55pt,dash phase=-0.1pt]
       (q1) [out=-70,in=-25] to (q0) ;
  \draw[line width=0.8*\widthObj,\colorObj,densely dashed]
       (q0) [out=220,in=130] to (q3) ;
  \filldraw[black]
       (p1) circle (2.2*\radvtx) (p2) circle (2.2*\radvtx) (p3) circle (2.2*\radvtx)
       (q0) circle (1.8*\radvtx)
       (0,0.1) node[xshift=4pt,yshift=2pt] {\scriptsize $v$} circle (\radvtx);  
  \node[\colorCat] at (140:0.77*\locR) {\small $\calm_{\!1}$} ;
  \node[\colorCat] at (60:0.83*\locR) {\small $\calm_{\!2}$} ;
  \node[\colorCat] at (75:0.36*\locR) {\small $\calm_{\!3}$} ;
  \node[\colorCat] at (310:0.79*\locR) {\small $\calm_{\!4}$} ;
  \end{tikzpicture}
 %\hspace*{5em}\includegraphics[width=4cm]{Graphs/graph_sphere.jpg}
  }
  \label{eq:pic:G}
  \ee
The structure on the bicategory $\mathbf{Mod}^{\rm sph}(\cala)$ described in Section \ref{sec:sph_module} allows us to 
define a graphical calculus, by which a closed graph on the sphere is evaluated
to a scalar in the following manner. Let $\mathcal{G}_v$ be such a labeled graph. 
First we select one of the patches on the sphere, i.e.\ a specific connected 
component of ${\tt S}^2_v\setmin \mathcal{G}_v$. By puncturing the chosen patch, we 
obtain a closed graph $\Tilde{\mathcal{G}}_v$ on the plane. The labeling of the graph
determines a planar string diagram in $\mathbf{Mod}^{\rm sph}(\cala)$.
This string diagram represents a $2$-endomorphism of the identity $1$-morphism, which is invariant under 
isotopies of the graph $\Tilde{\mathcal{G}}_v$ on the plane. (This invariance 
is established, for general pivotal bicategories, in \Cite{Prop.\,2.2}{FuSY}.)
Further, we can trace the associated $2$-endomorphism into an scalar by means of the
local Calabi-Yau structure, i.e.\ of the trace map of the corresponding Hom-category:
    \def\locR  {2.5}   % radius
  \be
  \raisebox{-5.6em}{ \begin{tikzpicture}
  \coordinate (p1) at (25:0.65*\locR) ;
  \coordinate (p2) at (105:0.58*\locR) ;
  \coordinate (p3) at (245:0.54*\locR) ;
  \coordinate (q0) at (350:1.13*\locR) ;
  \coordinate (q2) at (15:1.3*\locR) ;
  \coordinate (q3) at (290:0.85*\locR) ;
  \scopeArrow{0.47}{\arrowObj} 
  \draw[line width=\widthObj,\colorObj,postaction={decorate}]
       (p1) [out=125,in=20] to node[near start,above,sloped,yshift=-2pt]{$\scriptstyle K$} (p2) ;
  \draw[line width=\widthObj,\colorObj,postaction={decorate}]
       (p3) [out=5,in=260] to node[near start,below,sloped,yshift=2pt]{$\scriptstyle J$} (p1) ;
  \draw[line width=\widthObj,\colorObj,postaction={decorate}]
       (p2) [out=100,in=110] to node[near end,above,sloped,yshift=-2pt]{$\scriptstyle F$} 
       (q2) [out=290,in=35] to (q0) ;
  \draw[line width=\widthObj,\colorObj,postaction={decorate}]
       (p3) [out=125,in=230] to node[near start,below,sloped,yshift=2pt]{$\scriptstyle H$} (p2) ;
  \draw[line width=\widthObj,\colorObj,postaction={decorate}]
       (q0) [out=240,in=0] to
       (q3) [out=180,in=295] to node[near start,above,sloped,yshift=-2pt]{$\scriptstyle I$} (p3) ;
  \end{scope}
  \scopeArrow{0.61}{\arrowObj} 
  \draw[line width=\widthObj,\colorObj,postaction={decorate}]
       (p1) [out=20,in=110] to node[near end,above,sloped,yshift=-2pt]{$\scriptstyle G$}  (q0) ;
  \end{scope}
  \filldraw[black]
       (p1) circle (2.2*\radvtx) (p2) circle (2.2*\radvtx) (p3) circle (2.2*\radvtx)
       (q0) circle (2.2*\radvtx)
       (0,0.1);
  \node[\colorCat] at (150:0.62*\locR) {\small $\calm_{\!1}$} ;
  \node[\colorCat] at (33:0.91*\locR) {\small $\calm_{\!2}$} ;
  \node[\colorCat] at (45:0.35*\locR) {\small $\calm_{\!3}$} ;
  \node[\colorCat] at (330:0.71*\locR) {\small $\calm_{\!4}$} ;
  \draw[thick,gray,dashed] (232:1.17*\locR) rectangle (34:1.72*\locR) ;
  \end{tikzpicture}
  }
  ~~~\longmapsto~~ \Nat_{\rm mod}\left(1_{\!\calm_1},1_{\!\calm_1}\right)
  \,\xrightarrow{~\tr^{\scriptscriptstyle\Cala_{\Calm}^*}
  _{\scriptscriptstyle\mathbf{1}_{\Calm}}~}\, \ko \,. \qquad
 %\hspace*{5em} \includegraphics[width=4cm]{Graphs/graph_plane.jpg} 
  \ee
As we will show in Lemma \ref{invariance_sphere_graph} below, the so obtained scalar
is independent of the choice of patch used in the transition from the graph
$\mathcal{G}_v$ on the sphere to the graph $\Tilde{\mathcal{G}}_v$ on the
plane. Altogether, by tensoring over all the vertices of the skeleton,
the graphical calculus on the sphere determines a linear form
  \be
  \tr_\varphi\Colon V(M,S,\varphi)=\bigotimes_{v\in v(S)}\,\bigotimes_{l\in e_v(S)}
  V_{l} \rarr~ \ko\,.
  \label{vertex_evaluation}
  \ee

\begin{lem}\label{invariance_sphere_graph}
The scalar assigned by the prescription above to a closed
$\mathbf{Mod}^{\rm sph}(\cala)$-labeled graph $\mathcal{G}$ on the sphere ${\tt S}^2$
is invariant under isotopies of $\mathcal{G}$ and independent of the patch chosen to
define it.
\end{lem}

\begin{proof}
Consider first the situation that the graph $\mathcal{G}$ determines precisely two 
patches on the sphere. Then \eqref{un_triangle} ensures that by puncturing
either patch we obtain the same scalar. In case of a larger number of patches, the
result follows by iteratively applying the previous argument to any pair of patches 
adjacent to a common strand in the graph.
\end{proof}

   \item 
\textbf{State sum:} 
Finally we take a weighted average $\sum_{\varphi_3}\omega_3\,\sum_{\varphi_2}
\omega_2 \,{\cdot}\, {\rm tr}_\varphi(*_\varphi)$
over the scalars obtained for all possible labelings on $(M,S)$. In more detail,
given a labeling $\varphi \eq (\varphi_3,\varphi_2)$ we define the weight $\omega_2$
associated to the function $\varphi_2$ in \eqref{label_2} as
  \be
  \omega_2 \equiv \omega_2(\varphi_2)
  := \prod_{r\in\Reg(S)}\! \dim\left(\varphi_2(r)\right)^{\chi_r} ,
  \label{omega_2}
  \ee
where for a label $\varphi_2(r) \eq H \iN \Fun_\cala(\calm,\caln)$, the number
$\dim(H)$ is given by \eqref{ob_dim} and $\chi_r$ is the Euler characteristic of 
the region $r\iN\Reg(S)$. To the function $\varphi_3$ in \eqref{label_3} we designate a weight 
  \be 
  \omega_3 \equiv \omega_3(\varphi_3)
  := \prod_{\Gamma\in|M\backslash S|}\! \left(\dim\,\cala_{\varphi_3(\Gamma)}^*
  \cdot \#\varphi_3(\Gamma) \right)^{\!-1} ,
  \label{eq:omega3}
  \ee
where for $\calm\iN\mathbf{Mod}^{\rm sph}(\cala)$ the value $\dim\,\dualcat$ is 
the categorical dimension of the fusion category $\dualcat$ and $\#\calm$ is the 
number of equivalence classes of simple objects in the connected component of
$\calm$, i.e.\ the number of equivalence classes of simple objects
$\caln\iN\mathbf{Mod}^{\rm sph}(\cala)$ for which there are non-zero $1$-morphisms
$\calm\Rarr~\caln$. 

In our setting, the two types of factors entering the definition of $\omega_3$
simplify: First, the categorical dimension of any indecomposable $\cala$-module $\calm$ 
satisfies $\dim\,\dualcat \eq \dim\,\cala$ \Cite{Prop.\,9.3.9}{EGno}. And second,
since $\cala$ is fusion, the number $\#\calm \,{=:}\, \#\cala$ equals the number 
of equivalence classes of indecomposable $\cala$-module categories, which is a finite
number \Cite{Cor.\,9.1.6(ii)}{EGno}. It follows that we have in fact
  \be
  \omega_3=\big(\dim\,\cala\cdot \#\cala\big)^{\!-|M\backslash S|},
  \ee
which is a number independent of $\varphi_3$.
\end{enumerate}

\begin{defi}\label{state_sum_inv}
Let $M$ be a closed oriented $3$-manifold, $S$ a skeleton of $M$,
$\cala$ a spherical fusion category, and
$\mathbf{Mod}^{\rm sph}(\cala)$ the bicategory of spherical $\cala$-modules.
 % \\
The \emph{bicategorical state sum} assigned to $M$ is the scalar
  \be
  \label{the_sum}
  {\rm St}_{{\mathbf{Mod}^{\rm sph}(\cala)}} (M,S)
  := \, \big(\dim\,\cala\cdot \#\cala\big)^{\!-|M{\setminus}S|}\!\!\!\!\!
  \sum_{\varphi=(\varphi_3,\varphi_2)}\!
  \omega_2\cdot {\rm tr}_\varphi(*_\varphi) \,\in\ko \,,
  \ee
where the summation extends over all possible labelings \eqref{label_3} and
\eqref{label_2}, the value of $\omega_2$ is given by \eqref{omega_2}, 
$\dim\,\cala$ denotes the categorical dimension of $\cala$, and $\#\cala$ is 
the number of equivalence classes of indecomposable $\cala$-module categories.
\end{defi}

\begin{rem}\label{TV_term}
The Turaev-Viro invariant based on a spherical fusion category $\cala$ appears in
the state sum \eqref{the_sum} as the term for which $\varphi_3$ is the constant
labeling whose constant value is the regular spherical module category ${}_\cala\cala$, once one recalls that $\cala\simeq \Fun_\cala(\cala,\cala)$. 
\end{rem}

%%%%%%%%%%%%%%%%%%%%%%%%%%%%%%%%%%%%%%%%%%%%%%%%%%%%%%%%%%%%%%%%%%%%%%%%

\subsection{Skeleton independence of the state sum}\label{invariance}

While the state sum \eqref{the_sum} is defined for a $3$-manifold together with a 
skeleton, in fact it only depends on the manifold:

\begin{thm}\label{skeleton_inv}
Given a closed oriented $3$-manifold $M$, the bicategorical state sum scalar
 \\
${\rm St}_{{\mathbf{Mod}^{\rm sph}(\cala)}}(M,S)$ is independent of the choice of 
skeleton $S$ for $M$.
\end{thm}

\begin{proof}
According to Theorem \ref{skeletons}, any two skeletons on $M$ can be related by a 
finite sequence of the primary moves listed in Section \ref{sec:primary_moves}. 
Consequently, it suffices to show invariance of ${\rm St}_{} (M,S)$ under all primary
moves. For each of the primary moves we denote the skeleton before the move by $S$
and the one after the move by $S'$, and we fix a labeling
$\varphi \eq (\varphi_3,\varphi_2)$ of the skeleton $S$.
 \\[2pt]
We first consider the percolation move $T_3$, which creates a new edge both of
whose end points are a vertex $v$ and which bounds a new region. 
We restrict our attention to the situation that four regions meet at the vertex $v$
before the move; the general case is treated analogously.
Assigning any simple object $J\iN\Fun_\cala(\caln,\caln)$ to the new 
region created by the move extends $\varphi$ to a labeling of the new skeleton $S'$: 
  \be
  \raisebox{-3.6em}
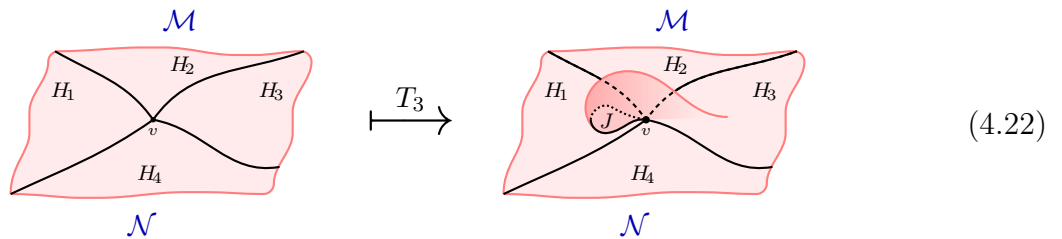
{ \begin{tikzpicture}
  \tikzpicTwocellGGfourlabeled
  \node[\colorCat] at (0.5*\stdx,-0.42) {\small $\caln$} ;
  \node[\colorCat] at (0.5*\stdx+\stdxx,\stdy+0.42) {\small $\calm$} ;
  \node at (0.541*\stdx,0.54*\stdy) [anchor=north]{$\scriptscriptstyle v$} ;
  \end{tikzpicture}
  }
  \hspace*{1.4em} \nicemapsto{T_3} \hspace*{1.2em}
  \raisebox{-3.6em}{ \begin{tikzpicture}
  \coordinate (p1n)  at (0.26*\stdx+0.5*\stdxx,0.52*\stdy) ;
  \tikzpicTwocellGGfour
  \tikzpicpercolfour
  \node[\colorCat] at (0.5*\stdx,-0.42) {\small $\caln$} ;
  \node[\colorCat] at (0.5*\stdx+\stdxx,\stdy+0.42) {\small $\calm$} ;
  \node at (0.06*\stdx+\stdxx,0.72*\stdy) {$\scriptstyle H_{\!1}$} ;
  \node at (0.54*\stdx+0.8*\stdxx,0.9*\stdy) {$\scriptstyle H_{\!2}$} ;
  \node at (0.85*\stdx+\stdxx,0.72*\stdy) {$\scriptstyle H_{\!3}$} ;
  \node at (0.53*\stdx,0.14*\stdy) {$\scriptstyle H_{\!4}$} ;
  \node at (0.392*\stdx,0.515*\stdy) {$\scriptstyle J$} ;
  \node at (0.541*\stdx,0.54*\stdy) [anchor=north]{$\scriptscriptstyle v$} ;
  \filldraw[black] (p1m) circle (1.5*\radvtx) ;
  \end{tikzpicture}
  }
  \ee
The contribution of the vertex $v$ to the state sum ${\rm St}_{} (M,S)$ is
the value of the corresponding graph on a sphere that surrounds $v$, i.e.\
the trace of the string diagram
  \def\locH    {0.85}   %  height  
  \be
  \label{trace_abcd}
  \raisebox{-3em}{ \begin{tikzpicture}[scale=\locscale]
  \draw[line width=\widthObjSm,\colorObj]
       (0,0) -- 
       +(0,\locH)
       (2*\locC,0) --  +(0,\locH)
       (0,0) arc (180:360:\locC)
       (0,\locH) arc (180:0:\locC) ;
  \scopeArrow{0.5}{\arrowObj}
  \draw[line width=\widthObjSm,\colorObj,postaction={decorate}]
       (\hsA+\locC,\locH+\locC) -- + (0.01,0) ;
  \draw[line width=\widthObjSm,\colorObj,postaction={decorate}]
       (-\hsA+\locC,-\locC) -- +(-0.01,0) ;
  \end{scope}
  \node[\colorCat] at (1*\locC,0.5*\locH) {\tiny $\calm\,$} ;
  \node[\colorCat] at (-3*\locC,0.5*\locH+0*\locC) {\small $\caln$} ;
  \node[\colorCat] at (5*\locC,0.5*\locH+0*\locC) {\small $\caln$} ;
  \node[black] at (-0.65*\locC,0.5*\locH)  {\scriptsize $H_1$} ;
  \node[black] at (0.5*\locC,1.75*\locC+1*\locH)  {\scriptsize $H_2$} ;
  \node[black] at (2.65*\locC,0.5*\locH)  {\scriptsize $H_3$} ;    
  \node[black] at (0.5*\locC,-\locC) [anchor=north west] {\scriptsize $H_4$} ;    
  \node[teal] at (0,1.2*\locH) [anchor=east] {\small $\alpha$} ;
  \node[teal] at (2*\locC,1.2*\locH) [anchor=west] {\small $\beta$} ;
  \node[teal] at (0,-0.2*\locH) [anchor=east] {\small $\delta$} ;
  \node[teal] at (2*\locC,-0.2*\locH) [anchor=west] {\small $\gamma$} ;    
  \filldraw[black] (0,0.1*\locH)  circle (1.3pt);
  \filldraw[black] (0,0.9*\locH) circle (1.3pt);
  \filldraw[black] (2*\locC,0.1*\locH) circle (1.3pt);
  \filldraw[black] (2*\locC,0.9*\locH) circle (1.3pt);
  \end{tikzpicture}
  }
  \ee
where $\alpha, \beta, \gamma, \delta$ are the module natural transformations coming
from the distinguished vector $*_\varphi$. The contribution of $v$ to
${\rm St}_{} (M,S')$ is the trace of the string diagram
      \def\locC {0.5}
  \be
  \sum_{J\in\Irr\cala_{\!\caln}^*} \dim(J)\;~\sum_i~~
  \raisebox{-4.5em}{ \begin{tikzpicture}[scale=\locscale]
  \draw[line width=\widthObjSm,\colorCat]
       (\locC,-0.5*\locH) -- 
       +(0,0.75*\locH)       ;
  \draw[line width=\widthObjSm,\colorObj]
       (\locC,0.25*\locH) to  [out=150,in=270] 
       (0,\locH)
       (\locC,0.25*\locH) to  [out=30,in=270] 
       (2*\locC,\locH)       
       (\locC,-0.5*\locH) to  [out=-150,in=90] 
       (0,-1.25*\locH)
       (\locC,-0.5*\locH) to  [out=-30,in=90] 
       (2*\locC,-1.25*\locH)       
       (0,-1.25*\locH) arc (180:360:\locC)
       (0,\locH) arc (180:0:\locC) ;
  \scopeArrow{0.5}{\arrowObj}
  \draw[line width=\widthObjSm,\colorObj,postaction={decorate}]
       (\hsA+\locC,\locH+\locC) -- + (0.01,0) ;
  \draw[line width=\widthObjSm,\colorObj,postaction={decorate}]
       (-\hsA+\locC,-1.25*\locH-1*\locC) -- +(-0.01,0) ;
  \end{scope}
    \node[\colorCat] at (1*\locC,-1.25*\locH) {\tiny $\calm\,$} ;
    \node[\colorCat] at (1*\locC,1*\locH) {\tiny $\calm\,$} ;
    \node[\colorCat] at (-2.5*\locC,0) {\small $\caln$};
    \node[\colorCat] at (4.5*\locC,0) {\small $\caln$};
    \node[\colorCat] at (0.5*\locC,-0.15*\locH)  {\scriptsize $J$};
    \node[black] at (-0.3*\locC,-0.75*\locH)  {\scriptsize $H_1$};
    \node[black] at (-0.3*\locC,0.5*\locH)  {\scriptsize $H_1$} ;
    \node[black] at (0.5*\locC,1.75*\locC+1*\locH)  {\scriptsize $H_2$} ;
    \node[black] at (2.3*\locC,-0.75*\locH)  {\scriptsize $H_3$};
    \node[black] at (2.3*\locC,0.5*\locH)  {\scriptsize $H_3$};
    \node[black] at (0.5*\locC,-1.25*\locH-1*\locC) [anchor=north west] {\scriptsize $H_4$} ;    
    \node[teal] at (\locC,-0.3*\locH) [anchor=west] {\tiny $\psi_i$};
    \node[teal] at (\locC,0.15*\locH) [anchor=west] {\tiny $\psi^*_i$};
    \node[teal] at (0,1.2*\locH) [anchor=east] {\small $\alpha$};
    \node[teal] at (2*\locC,1.2*\locH) [anchor=west] {\small $\beta$};
    \node[teal] at (0,-1.25*\locH) [anchor=east] {\small $\delta$};
    \node[teal] at (2*\locC,-1.25*\locH) [anchor=west] {\small $\gamma$};
    \filldraw[black] (0,-1.25*\locH)  circle (1.3pt);
    \filldraw[black] (\locC,-0.5*\locH)  circle (1.3pt);
    \filldraw[black] (\locC,0.25*\locH)  circle (1.3pt);    
    \filldraw[black] (0,\locH) circle (1.3pt);
    \filldraw[black] (2*\locC,-1.25*\locH) circle (1.3pt);
    \filldraw[black] (2*\locC,\locH) circle (1.3pt);
  \end{tikzpicture}
  }
  \ee
where $\psi_i^{}$ and $\psi_i^*$ are dual bases because the two half-edges associated 
to the new edge in $S'$ share their end point. The factor $\dim(J)$ comes from the 
factor in $\omega_2$ associated to the new region, while the $J$-summation accounts 
for the different possible labelings of $S'$ that extend $\varphi$. Now it follows
from Proposition \ref{properties_inv}(iii) that the so obtained expressions coincide:
  \be
  \def\locC {0.4}
  \tr^{\cala_{\!\caln}^*}_{\mathbf{1}_{\!\caln}}
  \Bigg(
  \hspace*{0.2em}
  \raisebox{-2.75em}{ \begin{tikzpicture}[scale=0.9*\locscale]
  \draw[line width=\widthObjSm,\colorObj]
       (0,0) -- +(0,\locH)
       (2*\locC,0) --  +(0,\locH)
       (0,0) arc (180:360:\locC)
       (0,\locH) arc (180:0:\locC) ;
  \scopeArrow{0.5}{\arrowObj}
  \draw[line width=\widthObjSm,\colorObj,postaction={decorate}]
       (\hsA+\locC,\locH+\locC) -- + (0.01,0) ;
  \draw[line width=\widthObjSm,\colorObj,postaction={decorate}]
       (-\hsA+\locC,-\locC) -- +(-0.01,0) ;
  \end{scope}
  \node[\colorCat] at (1*\locC,0.5*\locH) {\tiny $\calm\,$} ;
  \node[\colorCat] at (-2.5*\locC,0.5*\locH+0*\locC) {\small $\caln$} ;
  \node[\colorCat] at (4.5*\locC,0.5*\locH+0*\locC) {\small $\caln$} ;
  \node[black] at (-0.65*\locC,0.5*\locH)  {\scriptsize $H_1$} ;
  \node[black] at (0.5*\locC,1.75*\locC+1*\locH)  {\scriptsize $H_2$} ;
  \node[black] at (2.65*\locC,0.5*\locH)  {\scriptsize $H_3$} ;    
  \node[black] at (0.5*\locC,-\locC) [anchor=north west] {\scriptsize $H_4$} ;    
  \node[teal] at (0,1.2*\locH) [anchor=east] {\small $\alpha$} ;
  \node[teal] at (2*\locC,1.2*\locH) [anchor=west] {\small $\beta$} ;
  \node[teal] at (0,-0.2*\locH) [anchor=east] {\small $\delta$} ;
  \node[teal] at (2*\locC,-0.2*\locH) [anchor=west] {\small $\gamma$} ;    
  \filldraw[black] (0,0.1*\locH)  circle (1.3pt);
  \filldraw[black] (0,0.9*\locH) circle (1.3pt);
  \filldraw[black] (2*\locC,0.1*\locH) circle (1.3pt);
  \filldraw[black] (2*\locC,0.9*\locH) circle (1.3pt);
  \end{tikzpicture}
  }  
  \hspace*{0.2em}
  \Bigg) \;
  = \sum_{J\in\Irr\cala_{\!\caln}^*}\!\! \dim(J)\,\sum_i\,
  \tr^{\cala_{\!\caln}^*}_{\mathbf{1}_{\!\caln}}
  \Bigg( \hspace*{0.2em}
       \def\locC {0.5}  
  \raisebox{-4.5em}{ \begin{tikzpicture}[scale=\locscale]
  \draw[line width=\widthObjSm,\colorCat]
       (\locC,-0.5*\locH) -- 
       +(0,0.75*\locH)       ;
  \draw[line width=\widthObjSm,\colorObj]
       (\locC,0.25*\locH) to  [out=150,in=270] (0,\locH)
       (\locC,0.25*\locH) to  [out=30,in=270] (2*\locC,\locH)       
       (\locC,-0.5*\locH) to  [out=-150,in=90] (0,-1.25*\locH)
       (\locC,-0.5*\locH) to  [out=-30,in=90] (2*\locC,-1.25*\locH)       
   (0,-1.25*\locH) arc (180:360:\locC)
       (0,\locH) arc (180:0:\locC) ;
  \scopeArrow{0.5}{\arrowObj}
  \draw[line width=\widthObjSm,\colorObj,postaction={decorate}]
       (\hsA+\locC,\locH+\locC) -- + (0.01,0) ;
  \draw[line width=\widthObjSm,\colorObj,postaction={decorate}]
       (-\hsA+\locC,-1.25*\locH-1*\locC) -- +(-0.01,0) ;
  \end{scope}
  \node[\colorCat] at (1*\locC,-1.25*\locH) {\tiny $\calm\,$} ;
  \node[\colorCat] at (1*\locC,1*\locH) {\tiny $\calm\,$} ;
  \node[\colorCat] at (-2*\locC,0) {\small $\caln$};
  \node[\colorCat] at (4*\locC,0) {\small $\caln$};
  \node[\colorCat] at (0.5*\locC,-0.15*\locH)  {\scriptsize $J$};
  \node[black] at (-0.3*\locC,-0.75*\locH)  {\scriptsize $H_1$};
  \node[black] at (-0.3*\locC,0.5*\locH)  {\scriptsize $H_1$} ;
  \node[black] at (0.5*\locC,1.75*\locC+1*\locH)  {\scriptsize $H_2$} ;
  \node[black] at (2.3*\locC,-0.75*\locH)  {\scriptsize $H_3$};
  \node[black] at (2.3*\locC,0.5*\locH)  {\scriptsize $H_3$};
  \node[black] at (0.5*\locC,-1.25*\locH-1*\locC) [anchor=north west] {\scriptsize $H_4$} ;    
  \node[teal] at (\locC,-0.3*\locH) [anchor=west] {\tiny $\psi_i$};
  \node[teal] at (\locC,0.15*\locH) [anchor=west] {\tiny $\psi^*_i$};
  \node[teal] at (0,1.2*\locH) [anchor=east] {\small $\alpha$};
  \node[teal] at (2*\locC,1.2*\locH) [anchor=west] {\small $\beta$};
  \node[teal] at (0,-1.25*\locH) [anchor=east] {\small $\delta$};
  \node[teal] at (2*\locC,-1.25*\locH) [anchor=west] {\small $\gamma$};
  \filldraw[black] (0,-1.25*\locH)  circle (1.3pt);
  \filldraw[black] (\locC,-0.5*\locH)  circle (1.3pt);
  \filldraw[black] (\locC,0.25*\locH)  circle (1.3pt);    
  \filldraw[black] (0,\locH) circle (1.3pt);
  \filldraw[black] (2*\locC,-1.25*\locH) circle (1.3pt);
  \filldraw[black] (2*\locC,\locH) circle (1.3pt);
  \end{tikzpicture}
  }
  \hspace*{0.4em} \Bigg)
  \ee
This establishes the invariance of ${\rm St}_{{\mathbf{Mod}^{\rm sph}(\cala)}}(M,S)$ 
under the percolation move $T_3$.
 \\[2pt]
Next consider the contraction move $T_2$ which
collapses an edge $l \iN e(S)$ into a single vertex $v\iN v(S')$,
again for the situation of four regions meeting at $v$:
  \be
  \raisebox{-3.6em}{ \begin{tikzpicture}
  \tikzpicTwocellEEfour
  \node[\colorCat] at (0.5*\stdx,-0.42) {\small $\caln$} ;
  \node[\colorCat] at (0.5*\stdx+\stdxx,\stdy+0.42) {\small $\calm$} ;
  \draw[black,thick] (p11) [out=0,in=180] to (p12) ;
  \node at (0.06*\stdx+\stdxx,0.72*\stdy) {$\scriptstyle H_{\!1}$} ;
  \node at (0.54*\stdx+0.8*\stdxx,0.9*\stdy) {$\scriptstyle H_{\!2}$} ;
  \node at (0.85*\stdx+\stdxx,0.58*\stdy) {$\scriptstyle H_{\!3}$} ;
  \node at (0.53*\stdx,0.14*\stdy) {$\scriptstyle H_{\!4}$} ;
  \node at (0.45*\stdx+0.8*\stdxx,0.61*\stdy) {$\scriptstyle l$} ;
  \filldraw[black] (p11) circle (1.5*\radvtx) (p12) circle (1.5*\radvtx) ;

  \end{tikzpicture}
  }
  \hspace*{1.4em} \nicemapsto{T_2} \hspace*{1.2em}
  \raisebox{-3.6em}{ \begin{tikzpicture}
  \tikzpicTwocellFFfour
  \node[\colorCat] at (0.5*\stdx,-0.42) {\small $\caln$} ;
  \node[\colorCat] at (0.5*\stdx+\stdxx,\stdy+0.42) {\small $\calm$} ;
  \node at (0.06*\stdx+\stdxx,0.72*\stdy) {$\scriptstyle H_{\!1}$} ;
  \node at (0.54*\stdx+0.8*\stdxx,0.9*\stdy) {$\scriptstyle H_{\!2}$} ;
  \node at (0.85*\stdx+\stdxx,0.58*\stdy) {$\scriptstyle H_{\!3}$} ;
  \node at (0.53*\stdx,0.14*\stdy) {$\scriptstyle H_{\!4}$} ;
  \node at (0.53*\stdx,0.52*\stdy) [anchor=north]{$\scriptscriptstyle v$} ;
  \filldraw[black] (p1m) circle (1.5*\radvtx) ;
  \end{tikzpicture}
  }
  \ee
The contribution of the two end points of the edge $l$ to ${\rm St}_{}( M,S)$ is
  \def\locC {0.4}
  \def\locH {0.75}   %  height  
  \be
  \sum_i~\tr^{\cala_{\!\caln}^*}_{\mathbf{1}_{\!\caln}}
  \Bigg(
  \hspace*{0.2em}
  \raisebox{-2.8em}{ \begin{tikzpicture}[scale=\locscale]
  \draw[line width=\widthObjSm,\colorObj]
       (0,0) -- +(0,\locH)
       (2*\locC,0) --  +(0,\locH)
       (0,0) arc (180:360:\locC)
       (0,\locH) arc (180:0:\locC) ;
  \scopeArrow{0.5}{\arrowObj}
  \draw[line width=\widthObjSm,\colorObj,postaction={decorate}]
       (\hsA+\locC,\locH+\locC) -- + (0.01,0) ;
  \draw[line width=\widthObjSm,\colorObj,postaction={decorate}]
       (-\hsA+\locC,-\locC) -- +(-0.01,0) ;
  \end{scope}
  \node[\colorCat] at (1*\locC,0.5*\locH) {\tiny $\calm\,$} ;
  \node[\colorCat] at (-2.5*\locC,0.5*\locH+0*\locC) {\small $\caln$} ;
  \node[black] at (-0.65*\locC,0.5*\locH)  {\scriptsize $H_1$} ;
  \node[black] at (0.5*\locC,1.75*\locC+1*\locH)  {\scriptsize $H_2$} ; 
  \node[black] at (0.5*\locC,-\locC) [anchor=north west] {\scriptsize $H_4$} ;    
  \node[teal] at (0,1.2*\locH) [anchor=east] {\small $\alpha$} ;
  \node[teal] at (2*\locC,0.5*\locH) [anchor=west] {\footnotesize $\psi_i$} ;
  \node[teal] at (0,-0.2*\locH) [anchor=east] {\small $\delta$} ;
  \filldraw[black] (0,0.1*\locH)  circle (1.3pt);
  \filldraw[black] (0,0.9*\locH) circle (1.3pt);
  \filldraw[black] (2*\locC,0.5*\locH) circle (1.3pt);
  \end{tikzpicture}
  }
  \hspace*{0.2em}
  \Bigg) \; \cdot \;
  \tr^{\cala_{\!\caln}^*}_{\mathbf{1}_{\!\caln}}
  \Bigg(
  \hspace*{0.2em}
  \raisebox{-2.8em}{ \begin{tikzpicture}[scale=\locscale]
  \draw[line width=\widthObjSm,\colorObj]
       (0,0) -- 
       +(0,\locH)
       (2*\locC,0) --  +(0,\locH)
       (0,0) arc (180:360:\locC)
       (0,\locH) arc (180:0:\locC) ;
  \scopeArrow{0.5}{\arrowObj}
  \draw[line width=\widthObjSm,\colorObj,postaction={decorate}]
       (\hsA+\locC,\locH+\locC) -- + (0.01,0) ;
  \draw[line width=\widthObjSm,\colorObj,postaction={decorate}]
       (-\hsA+\locC,-\locC) -- +(-0.01,0) ;
  \end{scope}
  \node[\colorCat] at (1*\locC,0.5*\locH) {\tiny $\calm\,$} ;
  \node[\colorCat] at (4.5*\locC,0.5*\locH+0*\locC) {\small $\caln$} ;
  \node[black] at (0.5*\locC,1.75*\locC+1*\locH)  {\scriptsize $H_2$} ;
  \node[black] at (2.65*\locC,0.5*\locH)  {\scriptsize $H_3$} ;    
  \node[black] at (0.5*\locC,-\locC) [anchor=north west] {\scriptsize $H_4$} ;    
  \node[teal] at (0,0.5*\locH) [anchor=east] {\footnotesize $\psi_i^*$} ;
  \node[teal] at (2*\locC,1.2*\locH) [anchor=west] {\small $\beta$} ;
  \node[teal] at (2*\locC,-0.2*\locH) [anchor=west] {\small $\gamma$} ;
  \filldraw[black] (0,0.5*\locH)  circle (1.3pt);
  \filldraw[black] (2*\locC,0.1*\locH) circle (1.3pt);
  \filldraw[black] (2*\locC,0.9*\locH) circle (1.3pt);
  \end{tikzpicture}
  }
  \hspace*{0.2em}
  \Bigg)
  \label{contr_endpoints}
  \ee
where $\psi_i$ and $\psi_i^*$, coming from the two half-edges associated to $l$, are 
again dual bases. By the multiplicativity of the trace the number
\eqref{contr_endpoints} equals 
        \def\locH    {0.75}   %  height  
  \be
  \sum_i~\tr^{\cala_{\!\caln}^*}_{\mathbf{1}_{\!\caln}}
  \Bigg(
  \hspace*{0.2em}
  \raisebox{-2.8em}{ \begin{tikzpicture}[scale=\locscale]
  \draw[line width=\widthObjSm,\colorObj]
       (0,0) -- 
       +(0,\locH)
       (2*\locC,0) --  +(0,\locH)
       (0,0) arc (180:360:\locC)
       (0,\locH) arc (180:0:\locC) ;
  \scopeArrow{0.5}{\arrowObj}
  \draw[line width=\widthObjSm,\colorObj,postaction={decorate}]
       (\hsA+\locC,\locH+\locC) -- + (0.01,0) ;
  \draw[line width=\widthObjSm,\colorObj,postaction={decorate}]
       (-\hsA+\locC,-\locC) -- +(-0.01,0) ;
  \end{scope}
  \node[\colorCat] at (1*\locC,0.5*\locH) {\tiny $\calm\,$} ;
  \node[\colorCat] at (-2.5*\locC,0.5*\locH+0*\locC) {\small $\caln$} ;
  \node[black] at (-0.65*\locC,0.5*\locH)  {\scriptsize $H_1$} ;
  \node[black] at (0.5*\locC,1.75*\locC+1*\locH)  {\scriptsize $H_2$} ; 
  \node[black] at (0.5*\locC,-\locC) [anchor=north west] {\scriptsize $H_4$} ;    
  \node[teal] at (0,1.2*\locH) [anchor=east] {\small $\alpha$} ;
  \node[teal] at (2*\locC,0.5*\locH) [anchor=west] {\footnotesize $\psi_i$} ;
  \node[teal] at (0,-0.2*\locH) [anchor=east] {\small $\delta$} ;
  \filldraw[black] (0,0.1*\locH)  circle (1.3pt);
  \filldraw[black] (0,0.9*\locH) circle (1.3pt);
  \filldraw[black] (2*\locC,0.5*\locH) circle (1.3pt);
  \end{tikzpicture}
  }
  \hspace*{0.2em}
  \raisebox{-2.8em}{ \begin{tikzpicture}[scale=\locscale]
  \draw[line width=\widthObjSm,\colorObj]
       (0,0) -- +(0,\locH)
       (2*\locC,0) --  +(0,\locH)
       (0,0) arc (180:360:\locC)
       (0,\locH) arc (180:0:\locC) ;
  \scopeArrow{0.5}{\arrowObj}
  \draw[line width=\widthObjSm,\colorObj,postaction={decorate}]
       (\hsA+\locC,\locH+\locC) -- + (0.01,0) ;
  \draw[line width=\widthObjSm,\colorObj,postaction={decorate}]
       (-\hsA+\locC,-\locC) -- +(-0.01,0) ;
  \end{scope}
  \node[\colorCat] at (1*\locC,0.5*\locH) {\tiny $\calm\,$} ;
  \node[\colorCat] at (4.5*\locC,0.5*\locH+0*\locC) {\small $\caln$} ;
  \node[black] at (0.5*\locC,1.75*\locC+1*\locH)  {\scriptsize $H_2$} ;
  \node[black] at (2.65*\locC,0.5*\locH)  {\scriptsize $H_3$} ;    
  \node[black] at (0.5*\locC,-\locC) [anchor=north west] {\scriptsize $H_4$} ;    
  \node[teal] at (0,0.5*\locH) [anchor=east] {\footnotesize $\psi_i^*$} ;
  \node[teal] at (2*\locC,1.2*\locH) [anchor=west] {\small $\beta$} ;
  \node[teal] at (2*\locC,-0.2*\locH) [anchor=west] {\small $\gamma$} ;
  \filldraw[black] (0,0.5*\locH)  circle (1.3pt);
  \filldraw[black] (2*\locC,0.1*\locH) circle (1.3pt);
  \filldraw[black] (2*\locC,0.9*\locH) circle (1.3pt);
  \end{tikzpicture}
  }
  \hspace*{0.2em}
  \Bigg)\label{contr_endpoints_2}
  ~=~\,
  \tr^{\cala_{\!\caln}^*}_{\mathbf{1}_{\!\caln}}
  \Bigg(
  \hspace*{0.2em}
  \raisebox{-2.8em}{ \begin{tikzpicture}[scale=\locscale]
  \draw[line width=\widthObjSm,\colorObj]
       (0,0) -- +(0,\locH)
       (2*\locC,0) --  +(0,\locH)
       (0,0) arc (180:360:\locC)
       (0,\locH) arc (180:0:\locC) ;
  \scopeArrow{0.5}{\arrowObj}
  \draw[line width=\widthObjSm,\colorObj,postaction={decorate}]
       (\hsA+\locC,\locH+\locC) -- + (0.01,0) ;
  \draw[line width=\widthObjSm,\colorObj,postaction={decorate}]
       (-\hsA+\locC,-\locC) -- +(-0.01,0) ;
  \end{scope}
  \node[\colorCat] at (1*\locC,0.5*\locH) {\tiny $\calm\,$} ;
  \node[\colorCat] at (-3*\locC,0.5*\locH+0*\locC) {\small $\caln$} ;
  \node[\colorCat] at (5*\locC,0.5*\locH+0*\locC) {\small $\caln$} ;
  \node[black] at (-0.65*\locC,0.5*\locH)  {\scriptsize $H_1$} ;
  \node[black] at (0.5*\locC,1.75*\locC+1*\locH)  {\scriptsize $H_2$} ;
  \node[black] at (2.65*\locC,0.5*\locH)  {\scriptsize $H_3$} ;    
  \node[black] at (0.5*\locC,-\locC) [anchor=north west] {\scriptsize $H_4$} ;    
  \node[teal] at (0,1.2*\locH) [anchor=east] {\small $\alpha$} ;
  \node[teal] at (2*\locC,1.2*\locH) [anchor=west] {\small $\beta$} ;
  \node[teal] at (0,-0.2*\locH) [anchor=east] {\small $\delta$} ;
  \node[teal] at (2*\locC,-0.2*\locH) [anchor=west] {\small $\gamma$} ;    
  \filldraw[black] (0,0.1*\locH)  circle (1.3pt);
  \filldraw[black] (0,0.9*\locH) circle (1.3pt);
  \filldraw[black] (2*\locC,0.1*\locH) circle (1.3pt);
  \filldraw[black] (2*\locC,0.9*\locH) circle (1.3pt);
  \end{tikzpicture}
  }
  \hspace*{0.2em} \Bigg)  
  \ee
where the equality follows from Proposition \ref{properties_inv}(iv). On the other
hand, the graph on the sphere around $v$ after the move is again of the 
form \eqref{trace_abcd}. Hence the contribution of $v$ to ${\rm St}(M,S')$ 
is precisely given by the right hand side of \eqref{contr_endpoints_2}.
This proves invariance of the state sum under $T_2$.
 \\[2pt]
For the bubble move $T_0$, the interior of the newly created bubble is an additional
$3$-cell of $S'$. Every choice of a spherical module category 
$\caln\iN\Irr\,\mathbf{Mod}^{\rm sph}(\cala)$ extends the function $\varphi_3$ to the
first component of a labeling $\varphi' \eq (\varphi_3',\varphi_2')$ of $S'$.
Now consider a region $r$ labeled by $K$ and
fix such a spherical module $\caln$ and thus the function $\varphi_3'$. Then the
choice of simple module functors $H_1\iN\Fun_\cala(\calm,\caln)$ and 
$H_2\iN\Fun_\cala(\caln,\call)$ extends the function $\varphi_2$ to the second 
component $\varphi_2'$ of a labeling of $S'$:
  \be
  \raisebox{-3.6em}{ \begin{tikzpicture}
  \tikzpicTwocell
  \node at (0.14*\stdx,0.18*\stdy) {$\scriptstyle K$} ;
  \node[\colorMlab]  at (0.9*\stdx+\stdxx,0.83*\stdy) {$\scriptstyle r$} ;
  \node[\colorCat] at (0.5*\stdx,-0.42) {\small $\call$} ;
  \node[\colorCat] at (0.5*\stdx+\stdxx,\stdy+0.42) {\small $\calm$} ;
  \end{tikzpicture}
  }
  \hspace*{1.4em} \nicemapsto{T_0} \hspace*{1.2em}
  \raisebox{-3.6em}{ \begin{tikzpicture}
  \tikzpicTwocell
  \node at (0.14*\stdx,0.18*\stdy) {$\scriptstyle K$} ;
  \node[\colorMlab] at (0.9*\stdx+\stdxx,0.83*\stdy) {$\scriptstyle r$} ;
  \node[\colorCat] at (0.5*\stdx,-0.42) {\small $\call$} ;
  \node[\colorCat] at (0.5*\stdx+\stdxx,\stdy+0.42) {\small $\calm$} ;
  \shade[top color=\colorMd,bottom color=\colorMl]
            (0.5*\stdx+0.5*\stdxx,0.5*\stdy) circle (\radcir);
  \draw[black,thick] (0.5*\stdx+0.5*\stdxx+\radcir,0.5*\stdy) arc (0:180:\radcir);
  \draw[black,thick,densely dotted] (0.5*\stdx+0.5*\stdxx+\radcir,0.5*\stdy) arc (0:-180:\radcir);
  \draw[black] (0.5*\stdx+0.5*\stdxx+\radcir,0.5*\stdy) arc (0:-180:\radcir cm and 0.4*\radcir cm);
  \draw[black,dashed] (0.5*\stdx+0.5*\stdxx+\radcir,0.5*\stdy) arc (0:180:\radcir cm and 0.4*\radcir cm);
  \filldraw[black] (0.5*\stdx,0.5*\stdy-0.372*\radcir) circle (\radvtx) ;
  \node[\colorCat] at (0.5*\stdx+0.5*\stdxx,0.51*\stdy) {$\scriptstyle \caln$} ;
  \node at (0.54*\stdx+0.5*\stdxx,0.74*\stdy) {$\scriptstyle H_{\!1}$} ;
  \node at (0.54*\stdx+0.5*\stdxx,0.25*\stdy) {$\scriptstyle H_{\!2}$} ;
  \node at (0.38*\stdx+0.5*\stdxx,0.47*\stdy) {$\scriptstyle l$} ;
  \end{tikzpicture}
  }
  \ee
The new edge $l\iN e(S')$ created along the equator of the bubble has coinciding 
end points, namely the new vertex $v\iN v(S')$, and thus dual bases label its 
half-edges. Therefore the value of the graph around $v$ equals the dimension of 
the vector space $\Nat_{\rm mod}(\mathbf{1}_\Caln, K^\delta\,{\circ}\, H_2^{\delta_2}
\,{\circ}\, H_1^{\delta_1})$ associated to the edge $l$, where the
exponents $\delta$ are the signs
given by \eqref{delta_def}, which come 
from the relative orientation of the respective region and the edge $l$. Hence for a 
fixed $\varphi_3'$ the contribution of $v$ to the state sum ${\rm St}_{}(M,S)$ is
  \be
  \begin{aligned}
  \sum_{H_1,H_2} & \dim(H_1^{\delta_1})\,\dim(H_2^{\delta_2})\, \dim_\ko
  \Nat_{\rm mod}(\mathbf{1}_\Caln, K^\delta\cir H_2^{\delta_2}\cir H_1^{\delta_1})
  \\[-4pt]
  & \! = \sum_{H_1,H_2} \dim(H_1)\,\dim(H_2)\,\dim_\ko
  \Nat_{\rm mod}(\mathbf{1}_\Caln, K^\delta\cir H_2\cir H_1\ra)
  \\
  & = \sum_{H_2} \dim(H_2)\,\sum_{H_1} \dim(H_1)\,
  \dim_\ko \Nat_{\rm mod}(H_1, K^\delta\cir H_2)
  \\
  & = \sum_{H_2} \dim(H_2)\,\dim(K^\delta\cir H_2) = \dim(K^\delta)\,\sum_{H_2} (\dim(H_2))^2
  \\
  & = \dim(K)\,\dim\,\Fun_\Cala(\caln,\call) = \dim(K)\,\dim\,\cala_{\!\caln}^* \,.
  \end{aligned}
  \label{long_chain}
  \ee
Here the first equality makes use of Proposition \ref{multi_sph_mod}(i)
and the second equality holds because $H_1$ and $H_1\ra$ are duals; the third step
follows from the sum rule \eqref{dim_expansion}, while the fourth equality uses 
Proposition \ref{multi_sph_mod}(iii). Finally we use the definition of the
categorical dimension and the fact \Cite{Prop.\,2.17}{etno} that 
$\dim\,\Fun_\Cala(\caln,\call) \eq \dim\,\cala_{\!\caln}^*$ and once more
Proposition \ref{multi_sph_mod}(i).
 \\
Further, summing over all $\varphi_3'$ amounts to a factor of $\#\cala$, which
together with the factor $\dim\,\cala_{\!\caln}^* \eq \dim\,\cala$ from 
\eqref{long_chain} cancels out the additional factor in $\omega_3(\varphi_3')$ that according to \eqref{eq:omega3} arises from the presence of the new $3$-cell.
The factor of $\dim(K)$ in \eqref{long_chain} accounts for the change in 
Euler characteristic of the region $r$ labeled by $K$ (removing a disk from $r$
reduces $\chi_r$ by 1). Taken together, this demonstrates the invariance of the
state sum under $T_0$. 
 \\[2pt]
Finally, invariance under the move $T_1$ is proven in complete analogy to the proof in
\Cite{Thm.\ 13.1}{TV17} by using Proposition \ref{multi_sph_mod}(ii) and Proposition
\ref{properties_inv}(i) and (iv).
\end{proof}

%%%%%%%%%%%%%%%%%%%%%%%%%%%%%%%%%%%%%%%%%%%%%%%%%%%%%%%%%%%%%%%%%%%%%%%%

\subsection{Comparison to the Turaev-Viro invariant}

Let $\cala$ be a spherical fusion category. The main purpose of this section is to
show that the Turaev-Vi\-ro invariant associated to $\cala$ coincides with the 
state sum invariant constructed in Section \ref{sec:state_sum_inv} from the
pivotal bicategory $\mathbf{Mod}^{\rm sph}(\cala)$ of spherical $\cala$-modules.

\begin{thm}\label{thm:TV=St}
The bicategorical state sum invariant for $\mathbf{Mod}^{\rm sph}(\cala)$ coincides 
with the Turaev-Viro invariant for $\cala$:
  \be
  {\rm St}_{\mathbf{Mod}^{\rm sph}(\cala)} (M)={\rm TV}_{\Cala}(M) 
  \ee
for every closed oriented $3$-manifold $M$.	
\end{thm}

In order to prove this statement, first notice that by Theorem 
\ref{skeleton_inv} the value ${\rm St}_{\mathbf{Mod}^{\rm sph}(\cala)} (M)$ 
is independent of the skeleton chosen to compute it. Since every $3$-manifold admits 
a triangulation \Cite{Ch.\,33}{Moi}, we can fix a triangulation on $M$ and calculate
the state sum values by using its associated skeleton $T$. Explicitly, we would like 
to compare the state sum expression 
  \be \label{TV_A}
  {\rm TV}_{\cala} (M,T)= (\dim\,\cala)^{-|M\backslash T|}\!\!
  \sum_{\varphi_2:\Reg(T)\to\Irr\cala}\,\prod_{r\in\Reg(T)}
  \dim\,\varphi_2(r)\cdot {\rm tr}_{\varphi_2}(*_{\varphi_2})
  \ee
coming from the Turaev-Viro invariant 
for the spherical fusion category $\cala$ and the bicategorical state sum value
  \be \label{ST_MOD_A}
  {\rm St}_{\mathbf{Mod}^{\rm sph}(\cala)} (M,T)
  = \big(\dim\,\cala\cdot \#\cala\big)^{\!-|M\backslash T|}\!\!
  \sum_{\varphi=(\varphi_3,\varphi_2)}\,\prod_{r\in\Reg(T)}
  \dim\,\varphi_2(r)\cdot {\rm tr}_\varphi(*_\varphi)
  \ee
associated to the pivotal bicategory $\mathbf{Mod}^{\rm sph}(\cala)$

In \eqref{ST_MOD_A}, $\#\cala$ is the number of equivalence classes of simple objects
in $\mathbf{Mod}^{\rm sph}(\cala)$, which is finite \Cite{Cor.\ 9.1.6 (ii)}{EGno}.
Since each $3$-cell of $T$ is labeled by such an object via $\varphi_3$, the number
$\#\cala^{|M\backslash T|}$ is exactly the number of possible $3$-labelings 
  \be
  \label{labeling_phi3}
  \varphi_3\Colon |M\backslash T|\longrightarrow \Irr\,\mathbf{Mod}^{\rm sph}(\cala)\,.
  \ee
For a fixed choice of $\varphi_3$ define the partial sum
  \be
  \mathrm{P}_{\Varphi_3} := \sum_{\varphi_2}\prod_{r\in\Reg(T)}\!\dim\,\varphi_2(r)
  \cdot {\rm tr}_\varphi(*_\varphi)\,.
  \label{partial_sum_phi}
  \ee
This quantity involves two main ingredients of the state sum construction: On the
one hand, the trace form ${\rm tr}_\varphi$ from equation \eqref{vertex_evaluation}
is built up from graph evaluations at spheres organized according to the vertices 
of $T$. On the other hand, the distinguished vector $*_\varphi$ is governed by 
contributions that come from the edges of $T$. We can regard the expression 
\eqref{partial_sum_phi} as a linear form
  \be
  \overline{\mathrm{P}}_{\Varphi_3}\Colon\bigoplus_{\varphi_2} V(M,T,\varphi)
  \rarr~ \ko \,,\qquad \sum_{\varphi_2}\gamma \xmapsto{~~} \sum_{\varphi_2}
  \prod_{r\in\Reg(T)}\!\dim\,\varphi_2(r)\cdot {\rm tr}_\varphi(\gamma)
  \label{partial_sum_form}
  \ee
defined on the space of all possible $2$-labelings associated to $\varphi_3$,
evaluated at the sum $\sum_{\varphi_2}*_\varphi$ of the distinguished vectors. With
this consideration in mind, we next prove that the value of the partial sum 
$\mathrm{P}_{\Varphi_3}$ is in fact independent of the $3$-labeling $\varphi_3$.

\begin{lem}\label{partial_sums_invariant}
For any two $3$-labelings $\varphi_3$ and $\psi_3$, the partial sums defined by
\eqref{partial_sum_phi} are equal: $\mathrm{P}_{\Varphi_3} \eq \mathrm{P}_{\!\psi_3}$.
\end{lem}

\begin{proof}
Since the triangulation $T$ (or, more generally, any skeleton) has a finite number
of $3$-cells, it is enough to check that the partial sums $\mathrm{P}_{\Varphi_3}$ 
and $\mathrm{P}_{\!\psi_3}$ coincide if the $3$-labelings $\varphi_3$ and $\psi_3$
differ only in the assignment for a single $3$-cell $c\iN|M\backslash T|$. As we are 
considering a triangulation, the boundary of the closure of the $3$-cell $c$ is a 
tetrahedron $\Theta$; we depict this situation as follows:
  % 
  % $$ \includegraphics[width=6cm]{tetrahedron.jpeg} $$
    \def\locH  {3.4}   %  scale of tetrahedron
    \def\locL  {1.1}   %  scale of outer edges
  \be\label{tetrahedron}
  \raisebox{-6em}
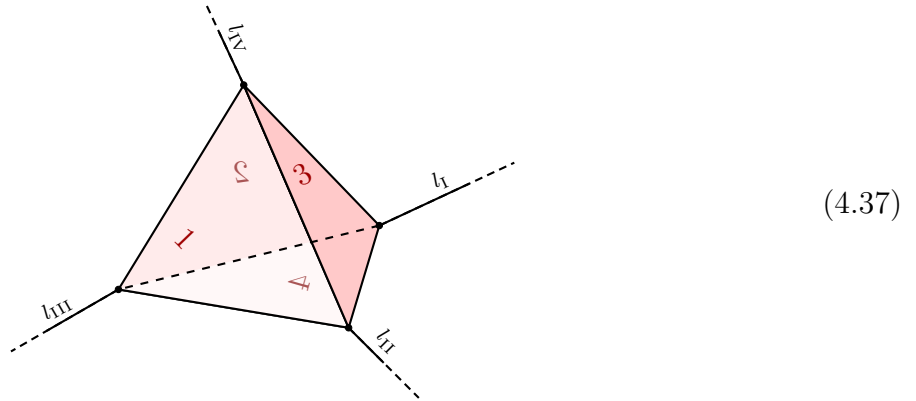
{
  \begin{tikzpicture}
  \coordinate (p1) at (1.02*\locH,0.25*\locH) ;
  \coordinate (p2) at (0.9*\locH,-0.15*\locH) ;
  \coordinate (p3) at (0,0) ;
  \coordinate (p4) at (0.49*\locH,0.8*\locH) ;
  \coordinate (r1) at (0.26*\locH,0.20*\locH) ;
  \coordinate (r2) at (0.48*\locH,0.46*\locH) ;
  \coordinate (r3) at (0.72*\locH,0.45*\locH) ;
  \coordinate (r4) at (0.71*\locH,0.03*\locH) ;
  \filldraw[draw=black,thick,fill=\colorM] (p3) -- (p2) -- (p4)  --
       (p3) ;
  \filldraw[draw=\colorN,fill=\colorN] (p3) -- (p2) -- (p1) -- (p3) ;
  \draw[black,thick]
       (p3) -- (p2) -- (p4)
       (p1) -- node[near end,sloped,above,yshift=-2pt] {$\scriptstyle l_{\mathrm{I}}$} + (25:1.2*\locL)
       (p2) -- node[near end,sloped,above,yshift=-1pt] {$\scriptstyle l_{\mathrm{II}}$} + (-45:0.6*\locL)
       (p3) -- node[near end,sloped,above,yshift=-2pt] {$\scriptstyle l_{\mathrm{III}}$} + (210:\locL)
       (p4) -- node[near end,sloped,above,yshift=-2pt] {$\scriptstyle l_{\mathrm{IV}}$} + (115:0.7*\locL)
  ;
  \draw[black,thick,densely dashed] (p1) -- + (25:1.8*\locL)
       (p2) -- + (-45:1.2*\locL) (p3) -- + (210:1.5*\locL) (p4) -- + (115:1.1*\locL) ;
  \filldraw[draw=black,thick,fill=\colorMJ]
       (p2) -- (p1) -- (p4) -- (p2) ;
  \draw[black,thick,dashed] (p1) -- (p3) ;
  \filldraw[black]
       (p1) circle (1.5*\radvtx) (p2) circle (1.5*\radvtx)
       (p3) circle (1.5*\radvtx) (p4) circle (1.5*\radvtx) ;
  \node[\colorMlab,rotate=-46] at (r1) {$1$} ;
  \node[\colorMlabb,rotate=8] at (r2) {\reflectbox{$2$}} ;
  \node[\colorMlab,rotate=32] at (r3) {$3$} ;
  \node[\colorMlabb,rotate=74] at (r4) {\reflectbox{$4$}} ;
  \end{tikzpicture}
  }
  \ee
The general idea of the proof is now the following: The tetrahedron $\Theta$ splits
the triangulation into two parts $T\backslash\Theta$ and $\Theta$. The value of
the partial sum \eqref{partial_sum_phi} can be obtained by evaluating the linear
form $\overline{\mathrm{P}}_{\Varphi_3}$ defined in \eqref{partial_sum_form} at the
distinguished vector. This linear form can be factored into the contribution
$\overline{\mathrm{P}}_{\Varphi_3}^{\Theta}$ of the four vertices of $\Theta$ and
the contribution $\overline{\mathrm{P}}_{\Varphi_3}^{\,\rm out}$ coming from
vertices outside $\Theta$. Now the $3$-labelings $\varphi_3$ and $\psi_3$ coincide
in $T\backslash\Theta$ by assumption and, moreover, the sets of possible $2$-labelings
of regions agree. Thus we have $\overline{\mathrm{P}}_{\Varphi_3}^{\,\rm out}
\eq \overline{\mathrm{P}}_{\!\psi_3}^{\,\rm out}$. On the other hand, the colors
assigned to the $3$-cell inside $\Theta$ by $\varphi_3$ and by $\psi_3$ are different
by assumption; as a consequence the possible $2$-labelings on the faces of the
tetrahedron are of a different type. Nevertheless, as we will show, the partial
evaluation of $\overline{\mathrm{P}}_{\Varphi_3}^{\Theta}$ at the factor of the
distinguished vector that comes from the edges of the tetrahedron is the same for
$\varphi_3$ and $\psi_3$, and hence the partial sums coincide
$\mathrm{P}_{\Varphi_3} \eq \mathrm{P}_{\!\psi_3}$. The analysis on the factorization
of the distinguished vector according to different types of edge contributions
requires introducing additional terminology and notation, and further details which
are described in what follows.

 \medskip

Let us now present the proof in more detail. 
\Enumerate
   \item 
At the level of $3$-cells: Since the
$3$-labelings $\varphi_3$ and $\psi_3$ differ only in the assignment given to the
$3$-cell $c$, we have $\varphi_3(b) \eq \psi_3(b) \,{=:}\, \lambda_3(b)$ for every
$3$-cell $b\,{\neq}\, c$. Denote by $\calm$ the spherical module category assigned 
to $c$ by the $3$-labeling $\varphi_3$ and by $\caln$  the spherical module assigned 
to $c$ by $\psi_3$.

  \item 
At the level of $2$-cells or regions: The tetrahedron has four faces, which 
constitute regions of the skeleton $T$. Let $\call_i \,{\coloneqq}\, \lambda_3(b_i)$
be the common label assigned by $\varphi_3$ and $\psi_3$ to the $3$-cell $b_i$ outside
the tetrahedron adjacent to the face $i$, for $i \eq 1,2,3,4$. The four faces of $\Theta$
then have labels of the type $\varphi_2(i) \,{:=}\, H_i\iN\Fun_\cala(\call_i,\calm)$ 
and $\psi_2(i) \,{:=}\, J_i\iN\Fun_\cala(\call_i,\caln)$, respectively.
The choice of such colors on each of the four tetrahedron faces can be extended to
a $2$-labeling $\varphi_2$ of the form \eqref{label_2} associated to $\varphi_3$ and
to a $2$-labeling $\psi_2$ associated to $\psi_3$ by coloring any region
$r\iN\Reg(T\backslash\Theta)$ outside the tetrahedron with a common label
$\lambda_2(r) \,{:=}\, \psi_2(r) \eq \varphi_2(r)$.

  \item 
The situation for the edges is more involved. The triangulation $T$ has three
different types of edges. First, there are \emph{internal edges}, namely the six
edges contained in $\Theta$. Each of these six edges has vertices of $\Theta$ as its
end points. There are also \emph{external edges}, which are the edges in $T$ that do
not belong to $\Theta$ and neither of whose end points is contained in $\Theta$. And
finally there are \emph{mixed edges}, which are edges outside $\Theta$ that have at
least one vertex in $\Theta$ as an end point.

To each of the labelings $\varphi \eq (\varphi_3,\varphi_2)$ there is assigned the
tensor product $V_l^\varphi\otimes_\ko V_{\overline{l}}^\varphi$ of a pair of vector
spaces for every internal edge $l\iN e(\Theta)$, and similarly for the labeling
$\psi \eq (\psi_3,\psi_2)$. In the case of external and mixed edges, both of the
labelings lead to the same vector spaces: for $l\iN e(T\backslash\Theta)$ we have 
  \be
  V_l^\varphi\otimes_\ko V_{\overline{l}}^\varphi=V_l^\psi\otimes_\ko
  V_{\overline{l}}^\psi\equiv V_l\otimes_\ko V_{\overline{l}} \,. 
  \ee
We obtain a factorization of the distinguished vector in terms of the three edge types:
  \be
  *_\varphi=*^{\rm ext} \otik *^{\rm mix} \otik *^{\rm int}_\varphi\;
  \in \bigotimes_{l \in{\rm ext}} V_{l}\otik V_{\overline{l}}
  \bigotimes_{l \in{\rm mix}} V_{l}\otik V_{\overline{l}}
  \bigotimes_{l \in{\rm int}} V^\varphi_{l}\otik V^\varphi_{\overline{l}}
  \cong V(M,S,\varphi)\,.
  \ee
Analogously, there is an edge factorization of $*_\psi$.

  \item 
On the other hand, we can reorganize the tensorands of $V(M,S,\varphi)$ according 
to the vertices, whereby we obtain a factorization
  \be
  V(M,T,\varphi) \cong V_{\rm out} \otimes_\ko V_\Theta^\varphi \,,
  \ee
with $V_\Theta^\varphi$ the vector space associated to the four vertices of the 
tetrahedron $\Theta$ and $V_{\rm out}$ the space coming from the vertices 
outside $\Theta$.
Consequently, the map $\tr_\varphi$ from equation \eqref{vertex_evaluation} 
given by evaluation at vertices factors as well:
  \be
  \label{tr_factorization}
  \tr_\varphi(\gamma_{\rm out}\oti \gamma_\Theta)
  = \tr^{\rm out}(\gamma_{\rm out})\cdot \tr_\varphi^\Theta(\gamma_\Theta)
  \ee
with $\gamma_{\rm out}\iN V_{\rm out}$ and $\gamma_\Theta\iN V_\Theta^\varphi$. The
same holds for the labeling $\psi$.

\item 
The partial sum \eqref{partial_sum_phi} for $\varphi_3$ can be re-expressed by
taking the previous discussion into account. First, we can split the sum over 
$2$-labelings by multiple sums over the colors on the four tetrahedron faces and 
sums on the regions outside $\Theta$. If, additionally, we consider the trace 
factorization \eqref{tr_factorization}, we obtain
  \be
  \mathrm{P}_{\Varphi_3} = \!\! \sum_{\overset{\lambda_2(r)}{r\in\Reg(T\backslash \Theta)}}
  \,\prod_{r\in\Reg(T\backslash \Theta)} \!\! \dim\,\lambda_2(r) \!\!
  \sum_{H_1\in{\rm F}
  (\call_1,\calm)} \!\!\! \cdots \!\!\! \sum_{H_4\in{\rm F}(\call_4,\calm)}\,
  \prod_{i=1}^4\dim \,H_i \;
  ({\rm tr}^{\rm out}\ot\tr_\varphi^\Theta) (*_\varphi) \,,
  \ee
where we abbreviate $\Fun_\cala(\call_i,\calm)$ as ${\rm F}(\call_i,\calm)$. 
The associated form \eqref{partial_sum_form} can be re-expressed similarly. 
Furthermore, we can factor this expression by considering the contributions of 
outer and inner vertices separately:
  \be
  \overline{\mathrm{P}}_{\Varphi_3} = \overline{\mathrm{P}}_{\Varphi_3}^{\,\rm out}
  \otimes \overline{\mathrm{P}}_{\Varphi_3}^{\Theta} \,.
  \ee
Explicitly,
  \be
  \overline{\mathrm{P}}_{\Varphi_3}^{\,\rm out}\coloneqq \sum_{\overset{\lambda_2(r)}{r\in\Reg(T\backslash \Theta)}}\,\prod_{r\in\Reg(T\backslash \Theta)}\,\dim\,\lambda_2(r)\,{\rm tr}^{\rm out}
  \Colon \bigoplus_{\varphi_2}\,V_{\rm{out}}\longrightarrow\ko
  \ee
and
  \be\label{phi_form_theta}
  \overline{\mathrm{P}}_{\Varphi_3}^{\,\Theta}\coloneqq \sum_{H_1\in{\rm F}(\call_1,\calm)}\cdots\sum_{H_4\in{\rm F}(\call_4,\calm)}\,\prod_{i=1}^4\dim \,H_i \;
  \tr_\varphi^\Theta 
  \Colon \bigoplus_{\varphi_2}\,V_\Theta^\varphi\longrightarrow\ko\,.
  \ee
Similar expressions are obtained for the $3$-labeling $\psi_3$. Since all labels
outside the tetrahedron coincide for $\varphi_3$ and $\psi_3$, it follows that
$\overline{\mathrm{P}}_{\Varphi_3}^{\,\rm out} \eq 
\overline{\mathrm{P}}_{\!\psi_3}^{\,\rm out}$. Thus the relevant contribution to the
partial sum \eqref{partial_sum_phi} that is affected by the color of the $3$-cell
inside of $\Theta$ comes exclusively from the four vertices of the tetrahedron, i.e.\
from the evaluation of the linear form \eqref{phi_form_theta}. It still remains to 
show that this contribution agrees for $\varphi_3$ and $\psi_3$ as well.

  \item 
By partially evaluating this linear form at $*^{\rm int}_\varphi$ for each $\varphi_2$
and restricting to the space corresponding to labels of regions bounding mixed edges,
we obtain a linear form
  \be
  \overline{\overline{\rm P}}^{\,\Theta}_{\Varphi_3}\coloneqq
  \overline{\mathrm{P}}^{\,\Theta}_{\Varphi_3}
  \left( \sum_{\varphi_2}\,-\otimes_\ko*^{\rm int}_\varphi\right) \Colon
  \bigoplus_{\varphi_2}\,\bigotimes_{l \in{\rm mix}}\,V_{l}\longrightarrow \ko\,,
  \ee
and in a similar manner a linear form for $\psi_3$. Notice that mixed edges only bound
regions outside of $\Theta$. These are colored with the same type of values by the
possible $2$-labelings $\varphi_2$ and $\psi_2$, therefore
$\overline{\overline{\rm P}}^{\,\Theta}_{\Varphi_3}$ and
$\overline{\overline{\rm P}}^{\,\Theta}_{\!\psi_3}$ are forms defined on the same space.

\item 
Moreover, these two maps coincide: Every $\gamma$ in their common domain can be
factored into a tensor product of vectors $\gamma_v$ associated to each of the four tetrahedron
vertices $v\iN v(\Theta)$; these will appear in the diagrams on the spheres around
said vertices:
    \def\locr  {1.2pt} % radius vertex
    \def\locR  {1.2}   % radius ball
    \def\locs  {0.6}   % scale line width=\width
  \be
  % ball 1:
  \raisebox{-2.4em}
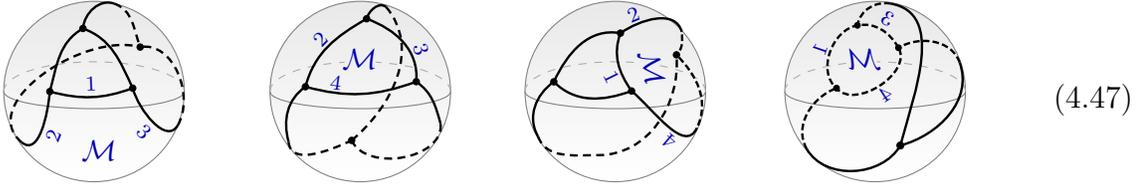
{ \begin{tikzpicture}
  \coordinate (p1) at (5:0.43*\locR) ;
  \coordinate (p2) at (100:0.71*\locR) ;
  \coordinate (p3) at (180:0.49*\locR) ;
  \coordinate (q0) at (44:0.71*\locR) ;
  \coordinate (q1) at (-22:\locR) ;
  \coordinate (q2) at (77:\locR) ;
  \coordinate (q3) at (215:0.98*\locR) ;
  \shade[top color=gray!18!white,bottom color=white,draw=gray,thin] (0,0) circle (\locR) ;
  \draw[gray,thin] (-\locR,0) arc (180:360:\locR cm and 0.19*\locR cm) ;
  \draw[gray!66!white,thin,dashed] (-\locR,0) arc (180:0:\locR cm and 0.33*\locR cm) ;
  \scopeArrow{0.59}{\arrowObj} 
  \draw[line width=\locs*\widthObj,\colorObj]
       (p1) [out=115,in=-20] to (p2)
       (p3) [out=-15,in=200] to
            node[midway,above,sloped,yshift=-2pt,\colorCat]{$\scriptstyle 1$} (p1)
       (p3) [out=75,in=210] to (p2)
       (p1) [out=-75,in=240] to 
            node[midway,below,sloped,yshift=2pt,\colorCat]{$\scriptstyle 3$} (q1)
       (p2) [out=85,in=160] to (q2) 
       (p3) [out=265,in=-40] to
            node[midway,below,sloped,yshift=2pt,\colorCat]{$\scriptstyle 2$} (q3) ;
  \draw[line width=\locs*\widthObj,\colorObj,densely dashed]
       (q1) [out=60,in=10] to (q0)
       (q2) [out=-20,in=110] to (q0)
       (q3) [out=140,in=165] to (q0) ;
  \end{scope}
  \filldraw[black]
       (p1) circle (\locr) (p2) circle (\locr) (p3) circle (\locr) (q0) circle (\locr) ;
  \node[\colorCat] at (275:0.65*\locR) {\small $\calm$} ;
  \end{tikzpicture}
  }
  \hspace*{1.7em}
  % ball 2:
  \raisebox{-2.4em}{ \begin{tikzpicture}
  \coordinate (p1) at (10:0.63*\locR) ;
  \coordinate (p2) at (85:0.81*\locR) ;
  \coordinate (p3) at (175:0.61*\locR) ;
  \coordinate (q0) at (260:0.55*\locR) ;
  \coordinate (q1) at (-22:0.96*\locR) ;
  \coordinate (q2) at (73:0.98*\locR) ;
  \coordinate (q3) at (215:0.97*\locR) ;
  \shade[top color=gray!18!white,bottom color=white,draw=gray,thin] (0,0) circle (\locR) ;
  \draw[gray,thin] (-\locR,0) arc (180:360:\locR cm and 0.19*\locR cm) ;
  \draw[gray!66!white,thin,dashed]
       (-\locR,0) arc (180:100:\locR cm and 0.33*\locR cm)
       (\locR,0) arc (0:80:\locR cm and 0.33*\locR cm) ;
  \draw[line width=\locs*\widthObj,\colorObj]
       (p1) [out=85,in=-20] to
       node[near start,above,sloped,xshift=-3pt,yshift=-3pt,\colorCat]{$\scriptstyle 3$} (p2)
       (p3) [out=-15,in=200] to
            node[near start,above,sloped,yshift=-3pt,\colorCat]{$\scriptstyle 4$} (p1)
       (p3) [out=75,in=210] to
            node[midway,above,sloped,yshift=-3pt,\colorCat]{$\scriptstyle 2$} (p2)
       (p1) [out=-35,in=105] to (q1)
       (p2) [out=75,in=140] to (q2) 
       (p3) [out=215,in=105] to (q3) ;
  \draw[line width=\locs*\widthObj,\colorObj,densely dashed]
       (q1) [out=-85,in=-65] to (q0)
       (q2) [out=-50,in=40] to (q0)
       (q3) [out=-85,in=205] to (q0) ;
  \filldraw[black]
       (p1) circle (\locr) (p2) circle (\locr) (p3) circle (\locr) (q0) circle (\locr) ;
  \node[\colorCat] at (91:0.36*\locR) {\small $\calm$} ;       
  \end{tikzpicture}
  }
  \hspace*{1.7em}
  % ball 3:
\raisebox{-2.4em}{ \begin{tikzpicture}
  \coordinate (p1) at (2:0.18*\locR) ;
  \coordinate (p2) at (85:0.65*\locR) ;
  \coordinate (p3) at (171:0.69*\locR) ;
  \coordinate (q0) at (31:0.79*\locR) ;
  \coordinate (q1) at (-24:0.99*\locR) ;
  \coordinate (q2) at (42:0.99*\locR) ;
  \coordinate (q3) at (210:0.99*\locR) ;
  \coordinate (q4) at (255:0.73*\locR) ;
  \shade[top color=gray!18!white,bottom color=white,draw=gray,thin] (0,0) circle (\locR) ;
  \draw[gray,thin] (-\locR,0) arc (180:360:\locR cm and 0.19*\locR cm) ;
  \draw[gray!66!white,thin,dashed]
       (-\locR,0) arc (180:75:\locR cm and 0.33*\locR cm)
       (\locR,0) arc (0:62:\locR cm and 0.33*\locR cm) ;
  \draw[line width=\locs*\widthObj,\colorObj]
       (p1) [out=125,in=250] to
       node[near start,below,sloped,xshift=-2pt,yshift=3pt,\colorCat]{$\scriptstyle 1$} (p2)
       (p3) [out=-45,in=200] to (p1)
       (p3) [out=65,in=170] to (p2)
       (p1) [out=-40,in=245] to
            node[midway,above,sloped,xshift=4pt,yshift=3pt,\colorCat,rotate=180]{$\scriptstyle 4$} (q1)
       (p2) [out=35,in=105] to
            node[near start,above,sloped,yshift=-3pt,\colorCat]{$\scriptstyle 2$} (q2)
       (p3) [out=215,in=120] to (q3) ;
  \draw[line width=\locs*\widthObj,\colorObj,densely dashed]
	  (q3) [out=290,in=180] to (q4) [out=0,in=275] to (q0) ;
  \draw[line width=\locs*\widthObj,\colorObj,dash pattern=on 2.5pt off 1.2,dash phase=-1pt]
       (q1) [out=65,in=-45] to (q0)
       (q2) [out=275,in=45] to (q0);
  \filldraw[black]
       (p1) circle (\locr) (p2) circle (\locr) (p3) circle (\locr) (q0) circle (\locr);
  \node[\colorCat,rotate=-75] at (34:0.49*\locR) {\small $\calm$};       
  \end{tikzpicture}
}
  \hspace*{1.7em}
  % ball 4:
  \raisebox{-2.4em}{ \begin{tikzpicture}
  \coordinate (p0) at (-65:0.66*\locR) ;
  \coordinate (q1) at (19:0.99*\locR) ;
  \coordinate (q2) at (80:0.99*\locR) ;
  \coordinate (q3) at (220:0.99*\locR) ;
  \coordinate (r1) at (62:0.55*\locR) ;
  \coordinate (r2) at (105:0.76*\locR) ;
  \coordinate (r3) at (175:0.43*\locR) ;
  \shade[top color=gray!18!white,bottom color=white,draw=gray,thin] (0,0) circle (\locR) ;
  \draw[gray,thin] (-\locR,0) arc (180:360:\locR cm and 0.19*\locR cm) ;
  \draw[gray!66!white,thin,dashed]
       (-\locR,0) arc (180:105:\locR cm and 0.33*\locR cm)
       (\locR,0) arc (0:89:\locR cm and 0.33*\locR cm) ;
  \draw[line width=\locs*\widthObj,\colorObj,dash pattern=on 2.5pt off 1.2]
       (r1) [out=120,in=10] to 
       node[midway,above,sloped,yshift=-2pt,\colorCat]{\small\reflectbox{$\scriptstyle 3$}} 
       (r2) [out=190,in=140] to 
       node[midway,above,sloped,yshift=-3pt,\colorCat]{\small\reflectbox{$\scriptstyle 1$}} 
       (r3) [out=-50,in=300] to
       node[midway,below,sloped,yshift=3pt,\colorCat]{\small\reflectbox{$\scriptstyle 4$}} 
       (r1)
       (q1) [out=110,in=30] to (r1)
       (q2) [out=170,in=80] to (r2)
       (q3) [out=120,in=200] to (r3) ;
  \draw[line width=\locs*\widthObj,\colorObj]
       (p0) [out=15,in=-70] to (q1)
       (p0) [out=75,in=-10] to (q2)
       (p0) [out=240,in=-60] to (q3) ;
  \filldraw[black]
       (r1) circle (0.9*\locr) (r2) circle (0.9*\locr) (r3) circle (0.9*\locr)
       (p0) circle (1.1*\locr) ;
  \node[\colorCat] at (109:0.38*\locR) {\small \reflectbox{$\calm$}} ;
  \end{tikzpicture}
  }
  \ee
(Here in order not to overburden the picture we omit the orientation of the edges
of the graph and abbreviate the labels $H_i$ by $i$.)
Puncturing each of these spheres at the $\calm$-labeled patch produces four graphs on
the plane. The multiplicativity of the trace from Proposition \ref{properties_inv}(ii)
allows us to compute the value of the four graphs as the trace of a single graph on 
the plane consisting of the disjoint union of these four graphs. Thus we have
   \def\locS  {1.5}   % length side
   \def\locv  {1.4pt} % radius vertex
  \begin{eqnarray}&&
  \overline{\overline{\rm P}}^{\,\Theta}_{\Varphi_3}(\gamma)
  = \! \sum_{H_1\in{\rm F}(\call_1,\calm)} \!\!\cdots\!\!\!
  \sum_{H_4\in{\rm F}(\call_4,\calm)}\,\prod_{i=1}^4 \dim \,H_i \;
  \tr^{\cala_{\!\calm}^*}_{\mathbf{1}_{\!\calm}}
  \Bigg(\hspace*{0.1em}
  \raisebox{-6.1em}{ \begin{tikzpicture}
    \draw[line width=0.8*\widthObj,\colorObj]
    %triangles
    (0,-0.5*\locS) to node[midway,xshift=-6pt] {\tiny $H_{\!1}$} (0,0.5*\locS)  
    (0,0.5*\locS) to node[midway,sloped,yshift=5pt] {\tiny $H_{\!3}$} (0.866*\locS,0)
    (0,-0.5*\locS) to node[midway,sloped,yshift=-6pt] {\tiny $H_{\!4}$} (0.866*\locS,0)
    (1.866*\locS+0.25*\locS,-0.5*\locS) to node[midway,xshift=6pt] {\tiny $H_{\!2}$} (1.866*\locS+0.25*\locS,0.5*\locS)
    (1.866*\locS+0.25*\locS,0.5*\locS) to node[midway,sloped,yshift=4pt] {\tiny $H_{\!3}$} (1*\locS+0.25*\locS,0)
    (1.866*\locS+0.25*\locS,-0.5*\locS) to node[midway,sloped,yshift=-5pt] {\tiny $H_{\!4}$} (1*\locS+0.25*\locS,0)
    %inner-lines
    (0,-0.5*\locS) to  (0.288*\locS,0)
    (0,0.5*\locS) to (0.288*\locS,0)
    (0.866*\locS,0) to (0.288*\locS,0)
    (1.866*\locS+0.25*\locS,-0.5*\locS) to (1.866*\locS+0.25*\locS-0.288*\locS,0)
    (1.866*\locS+0.25*\locS,0.5*\locS) to (1.866*\locS+0.25*\locS-0.288*\locS,0)
    (1*\locS+0.25*\locS,0) to (1.866*\locS+0.25*\locS-0.288*\locS,0);
    %top-graph
    \draw[line width=0.8*\widthObj,\colorObj]
    %inner
    (1.866*\locS+0.5*\locS,1.1*\locS) to  (1.866*\locS+0.25*\locS,0.75*\locS)
    (1.866*\locS+0.5*\locS,1.1*\locS) to (0,0.75*\locS)
    (1.866*\locS+0.5*\locS,1.1*\locS) to (1.866*\locS+1*\locS,0.25*\locS)    
    %outer
    (0,0.75*\locS) to node[midway,sloped,yshift=-5pt] {\tiny $H_{\!3}$} (1.866*\locS+0.25*\locS,0.75*\locS)
    (1.866*\locS+1*\locS,0.25*\locS) to node[midway,sloped,yshift=-5pt] {\tiny $H_{\!2}$} (1.866*\locS+0.25*\locS,0.75*\locS);
    \draw[line width=0.8*\widthObj,\colorObj]    (0,0.75*\locS)  .. controls 
    (1.866*\locS-0.5*\locS,1.75*\locS) and
    (1.866*\locS+1.25*\locS,1.75*\locS) .. (1.866*\locS+1*\locS,0.25*\locS)  
    ;
    \node[black] at (1.866*\locS+1.16*\locS,0.7*\locS) {\tiny $H_{\!1}$} ;
    %bottom-graph
    \draw[line width=0.8*\widthObj,\colorObj]
    %inner
    (1.866*\locS+0.5*\locS,-1.1*\locS) to  (1.866*\locS+0.25*\locS,-0.75*\locS)
    (1.866*\locS+0.5*\locS,-1.1*\locS) to (0,-0.75*\locS)
    (1.866*\locS+0.5*\locS,-1.1*\locS) to (1.866*\locS+1*\locS,-0.25*\locS)  
    %outer
    (0,-0.75*\locS) to node[midway,sloped,yshift=4pt] {\tiny $H_{\!4}$} (1.866*\locS+0.25*\locS,-0.75*\locS)
    (1.866*\locS+1*\locS,-0.25*\locS) to node[midway,sloped,yshift=4pt] {\tiny $H_{\!2}$} (1.866*\locS+0.25*\locS,-0.75*\locS);
    \draw[line width=0.8*\widthObj,\colorObj]    (0,-0.75*\locS) .. controls 
    (1.866*\locS-0.5*\locS,-1.75*\locS) and
    (1.866*\locS+1.25*\locS,-1.75*\locS) .. (1.866*\locS+1*\locS,-0.25*\locS)
    ;
    \node[black] at (1.866*\locS+1.16*\locS,-0.7*\locS) {\tiny $H_{\!1}$} ;   
    %vertices
  \filldraw[\colorObj]
    (0.288*\locS,0) circle (\locv) (1.866*\locS+0.25*\locS-0.288*\locS,0) circle (\locv)
        (0,-0.5*\locS) circle (\locv) (0,0.5*\locS) circle (\locv)
    (0,0.5*\locS) circle (\locv) (0.866*\locS,0) circle (\locv)
    (0,-0.5*\locS) circle (\locv) (0.866*\locS,0) circle (\locv)
    (1.866*\locS+0.25*\locS,-0.5*\locS) circle (\locv) (1.866*\locS+0.25*\locS,0.5*\locS) circle (\locv)
    (1.866*\locS+0.25*\locS,0.5*\locS) circle (\locv) (1*\locS+0.25*\locS,0) circle (\locv)
    (1.866*\locS+0.25*\locS,-0.5*\locS) circle (\locv) (1*\locS+0.25*\locS,0) circle (\locv)
    %top
    (1.866*\locS+0.5*\locS,1.1*\locS) circle (\locv) (1.866*\locS+0.25*\locS,0.75*\locS) circle (\locv)
    (1.866*\locS+0.5*\locS,1.1*\locS) circle (\locv) (0,0.75*\locS)    circle (\locv)
    (1.866*\locS+0.5*\locS,1.1*\locS) circle (\locv) (1.866*\locS+1*\locS,0.25*\locS)     circle (\locv)
    %bottom
    (1.866*\locS+0.5*\locS,-1.1*\locS) circle (\locv) (1.866*\locS+0.25*\locS,-0.75*\locS) circle (\locv)
    (1.866*\locS+0.5*\locS,-1.1*\locS) circle (\locv) (0,-0.75*\locS)    circle (\locv)
    (1.866*\locS+0.5*\locS,-1.1*\locS) circle (\locv) (1.866*\locS+1*\locS,-0.25*\locS)     circle (\locv)
    ;
  \node[\colorCat] at (\locS,0.35*\locS)  {\small $\calm$} ;
  \node[\colorCat] at (1.1*\locS,-0.35*\locS)  {\small $\calm$} ;
  \node[\colorCat] at (3.25*\locS,0.2)  {\small $\calm$} ;
  \node[\colorCat] at (0.12*\locS,0)  {\scriptsize $\call_{\!1}$} ;
  \node[\colorCat] at (1.866*\locS+0.12*\locS,0)  {\scriptsize $\call_{\!2}$} ; 
  \node[\colorCat] at (0.39*\locS,0.14*\locS)  {\scriptsize $\call_{\!3}$} ;
  \node[\colorCat] at (0.35*\locS,-0.14*\locS)  {\scriptsize $\call_{\!4}$} ;
  \node[\colorCat] at (1.876*\locS-0.15*\locS,0.14*\locS)  {\scriptsize $\call_{\!3}$} ;
  \node[\colorCat] at (1.879*\locS-0.15*\locS,-0.14*\locS)  {\scriptsize $\call_{\!4}$} ;
  \node[\colorCat] at (2*\locS,0.9*\locS)  {\footnotesize $\call_{\!3}$} ;
  \node[\colorCat] at (2.35*\locS,0.8*\locS)  {\footnotesize $\call_{\!2}$} ;   
  \node[\colorCat] at (1.7*\locS,1.23*\locS)  {\footnotesize $\call_{\!1}$} ;
  \node[\colorCat] at (2*\locS,-0.9*\locS)  {\footnotesize $\call_{\!4}$} ;
  \node[\colorCat] at (2.35*\locS,-0.8*\locS)  {\footnotesize $\call_{\!2}$} ;
  \node[\colorCat] at (1.7*\locS,-1.23*\locS)  {\footnotesize $\call_{\!1}$} ;  
\end{tikzpicture}
  }
  \hspace*{0.1em} \Bigg) . \hspace*{1.5em} 
  \nonumber \\[-2.5em] &&
  \end{eqnarray}
where the summation over the six pairs of dual bases coming from the partial
evaluation at $*^{\rm int}_\varphi$ is implicit. 
For each of these six pairs, the vertices labeled by the two basis morphisms in the
pair belong to disjoint components of the graph. When applying the dominance relation
from Proposition \ref{properties_inv}(iv) to each of these pairs of dual bases,
the four subgraphs get connected to each other, yielding
   \def\locaA {70}    % angle point p1
   \def\locaB {190}   % angle point p2
   \def\locaC {300}   % angle point p3
   \def\locR  {1.5}   % radius largish circle
   \def\locv  {2.1pt} % radius vertex
  \be
  \label{eq:cmp2}
  \overline{\overline{\rm P}}^{\,\Theta}_{\Varphi_3}(\gamma)
  = \sum_{H_1\in{\rm F}(\call_1,\calm)} \!\!\! \dim \,H_1\,\,\cdots \!\!\!\!\!
  \sum_{H_4\in{\rm F}(\call_4,\calm)}\!\!\! \dim \,H_4 \;
  \tr^{\cala_{\!\calm}^*}_{\mathbf{1}_{\!\calm}}
  \Bigg( \hspace*{0.5em} 
  \raisebox{-4.5em}{ \begin{tikzpicture}
  \coordinate (p0I) at (3:0.02) ;
  \mercedesgraph
  \draw[line width=0.8*\widthObj,\colorObj]
        (0,0) circle (1.31*\locR) 
        (p2) circle (0.27*\locR) (p3) circle (0.26*\locR) (p4) circle (0.26*\locR) ;
  \node[rotate=-27] at (55:1.44*\locR) {$\scriptstyle H_{\!1}$} ;
  \node[\colorCat,rotate=16] at (103:1.141*\locR) {\small $\call_{\!1}$} ;
  \begin{scope}[shift=(p2)]
        \node[rotate=-8] at (75:0.38*\locR) {$\scriptstyle H_{\!2}$} ;
        \node[\colorCat] at (282:0.55*\locR) {\small $\call_{\!2}$} ;
  \end{scope}
  \begin{scope}[shift=(p3)]
        \node[rotate=-33] at (50:0.37*\locR) {$\scriptstyle H_{\!3}$} ;
        \node[\colorCat] at (219:0.50*\locR) {\small $\call_{\!3}$} ;
  \end{scope}
  \begin{scope}[shift=(p4)]
        \node[rotate=63] at (150:0.36*\locR) {$\scriptstyle H_{\!4}$} ;
        \node[\colorCat] at (323:0.47*\locR) {\small $\call_{\!4}$} ;
  \end{scope}
  \scopeArrow{0.5}{\arrowObj}
    \draw[line width=\widthObjSm,\colorObj,postaction={decorate}]
       (\hsA,-1.31*\locR) -- + (0.01,0) ;
    \draw[line width=\widthObjSm,\colorObj,postaction={decorate},shift=(p3)]
       (\hsA,-0.26*\locR) -- + (0.01,0) ;
    \draw[line width=\widthObjSm,\colorObj,postaction={decorate},shift=(p2)]
       (\hsA,-0.26*\locR) -- + (0.01,0) ;
    \draw[line width=\widthObjSm,\colorObj,postaction={decorate},shift=(p4)]
       (\hsA,-0.26*\locR) -- + (0.01,0) ;       
  \end{scope}
  \node[\colorObj] at (\locaA:1.13*\locR){\tiny $\gamma_{\rm IV}$} ;
  \node[\colorObj] at (\locaB:1.15*\locR){\tiny $\gamma_{\rm II}$} ;
  \node[\colorObj] at (p0I) {\tiny $\gamma_{\rm I}$} ;
  \node[\colorObj] at (\locaC:1.15*\locR){\tiny $\gamma_{\rm III}$} ;  
  \node[\colorCat] at (p2) {\small $\calm$} ;
  \node[\colorCat] at (p3) {\small $\calm$} ;
  \node[\colorCat] at (p4) {\small $\calm$} ;
  \node[\colorCat] at (-28:1.3*\locR) [anchor=west] {\small $\calm$} ;  
  \end{tikzpicture}
  }
  \hspace*{-0.1em} \Bigg)\,.~~
  \ee
  % \raisebox{-4em}{\includegraphics[width=3.5cm]{Computation/cmp2.jpg}}
For each $i \iN \{1,2,3,4\}$ the so obtained graph in \eqref{eq:cmp2} contains
a $H_i$-labeled loop, which amounts to a factor of $\dim \,H_i$. Summing over all
possible values of $H_i$ then gives a factor of $\dim\,\Fun_\Cala(\call_i,\calm)$,
which according to \Cite{Prop.\,2.17}{etno} equals $\dim\,\cala$. Thereby 
\eqref{eq:cmp2} reduces to
   \def\locR  {1.3}   % radius largish circle
   \def\locv  {1.9pt} % radius vertex
  \be
  \label{eq:cmp3}
  \overline{\overline{\rm P}}^{\,\Theta}_{\Varphi_3}(\gamma)=\;(\dim\,\cala)^4\;
  \tr^{\cala_{\!\call_1}^*}_{\mathbf{1}_{\!\call_1}}
  \Bigg( \hspace*{0.1em} 
  \raisebox{-3.2em}{ \begin{tikzpicture}
  \coordinate (p0I) at (3:0.02) ;
  \mercedesgraph
  \node[\colorCat] at (150:1.3*\locR) {\small $\call_{\!1}$} ;
  \node[\colorCat] at (350:0.63*\locR) {\small $\call_{\!2}$} ;
  \node[\colorCat] at (130:0.63*\locR) {\small $\call_{\!3}$} ;
  \node[\colorCat] at (245:0.63*\locR) {\small $\call_{\!4}$} ;
  \node[\colorObj] at (\locaA:1.2*\locR){\tiny $\gamma_{\rm IV}$} ;
  \node[\colorObj] at (\locaB:1.22*\locR){\tiny $\gamma_{\rm II}$} ;
  \node[\colorObj] at (p0I) {\tiny $\gamma_{\rm I}$} ;
  \node[\colorObj] at (\locaC:1.22*\locR){\tiny $\gamma_{\rm III}$} ;  
  \end{tikzpicture}
  }
  \hspace*{0.7em} \Bigg)\,.
  \ee
The same argument works for the labeling $\psi_3$, and thus we have
$\overline{\overline{\rm P}}^{\,\Theta}_{\Varphi_3}(\gamma) \eq
\overline{\overline{\rm P}}^{\,\Theta}_{\!\psi_3}(\gamma)$.
It follows that $\mathrm{P}_{\Varphi_3} \eq \mathrm{P}_{\!\psi_3}$, as desired.
\end{enumerate}
\end{proof}

\begin{proof}[Proof of Theorem $\ref{thm:TV=St}$]
First, as mentioned in Remark \ref{TV_term}, the Turaev-Viro invariant
\eqref{TV_A} appears up to a scaling as one of the summands in \eqref{ST_MOD_A},
taking the equivalence $\cala \,{\simeq}\, \Fun_\cala(\cala,\cala)$ into account: 
  \be
  {\rm TV}_{\Cala} (M,T)=(\dim\,\cala)^{-|M\backslash T|}\;\mathrm{P}_{\Cala} \,,
  \ee
where $\mathrm{P}_{\Cala}$ is the partial sum associated to the labeling that colors
every $3$-cell with the regular spherical module category ${}_\cala\cala$.
Furthermore, by Lemma \ref{partial_sums_invariant} we have
$\mathrm{P}_{\Varphi_3} \eq \mathrm{P}_{\Cala}$ for every labeling $\varphi_3$.
It follows that
  \be 
  {\rm St}_{\mathbf{Mod}^{\rm sph}(\cala)} (M,T)
  = (\dim\,\cala\cdot \#\cala)^{-|M\backslash T|}\,
  \sum_{\varphi_3}\,\mathrm{P}_{\Varphi_3}
  = (\dim\,\cala\cdot \#\cala)^{-|M\backslash T|}\,
  \#\cala^{|M\backslash T|}\cdot\mathrm{P}_{\Cala} \,,
  \ee
where in the second equality we use the fact that $\#\cala^{|M\backslash T|}$ counts
the number of labelings \eqref{labeling_phi3}.
\end{proof}

\begin{rem}
A natural question that arises when dealing with a numerical manifold invariant
is whether it takes values in some particular subfield or subring of the field \ko.
Specifically, one may ask whether for every closed 3-manifold the Turaev-Viro
invariant based on a spherical fusion category for $\ko \eq \complex$ takes values in
the rationals, or in a real field, or in a cyclotomic field, etc. Now as a 
consequence of Theorem \ref{thm:TV=St}, to make sure that the Turaev-Viro invariant 
for a spherical fusion category \(\mathcal{A}\) takes only real values, it is
sufficient to show that \(\mathcal{A}\) is pivotal Morita equivalent to its complex 
conjugate, even if \(\mathcal{A}\) is not  tensor equivalent to its conjugate.
\end{rem}

The following example exhibits a family of spherical fusion categories that 
illustrates this point.

\begin{exa}
For \(G\) a finite group and \(\omega \iN Z^3(G,\ko^\times)\) a 3-cocycle on $G$, 
denote by \(\vect_G^\omega\) the pointed fusion category of finite-dimensional
\(G\)-graded vector spaces with associativity constraint twisted by \(\omega\) and
with a canonical spherical structure that ensures that the quantum dimension of every
simple object is $1$.
 \\[2pt]
Module categories over \(\vect_G^\omega\)
are in bijection with pairs \((S, \psi)\), where \(S \,{\subseteq}\, G\) is a 
subgroup and \(\psi \iN C^2(S,\ko^\times)\) is a 2-cochain on $S$ satisfying
$\mathrm d\psi \eq \omega|_S^{}$ \cite{ostr5}.
Given such a pair \((S,\psi)\), the twisted group algebra \(\ko_\psi[S]\) is an 
algebra in \(\vect_G^\omega\), and the associated module category \(\calm(S, \psi)\) 
is realized as the category of right \(\ko_\psi[S]\)-modules in \(\vect_G^\omega\).
\(\ko_\psi[S]\) is even a special symmetric Frobenius algebra, and hence \(\calm(S, \psi)\)
is a pivotal \(\vect_G^\omega\)-mo\-du\-le category.
 \\[2pt]
Let now \(A\) and \(B\) be finite abelian groups and \(\alpha \iN Z^2(A,\widehat{B})\)
be a 2-cocycle. We denote by \(\widehat{B} \,{\rtimes_{\alpha}}\, A\) the central
extension of \(A\) by \(\widehat{B}\) associated with \(\alpha\). The function
$\omega_{\alpha}$ defined by
  \be
  \omega_{\alpha}((x_1, a_1), (x_2, a_2), (x_3, a_3)) = \alpha(a_1, a_2)(x_3)
  \ee
is a 3-cocycle in \(Z^3(A \,{\oplus}\, B,\ko^\times)\). 
According Theorem 3.6 of \cite{urib} the fusion categories
\(\vect^{\omega_\alpha}_{B \oplus A}\) and \(\vect_{\widehat{B} \rtimes_{\alpha} A}\) 
are Morita equivalent.
 \\[2pt]
As a concrete example, take \(A \eq \Z/n\Z \,{\oplus}\, \Z/n\Z\) and
\(B \eq \Z/n\Z\), and set $\ko\eq\complex$ and
  \be
  \alpha \Colon A \Times A \rarr~ \hat{B} \,, \qquad
  \alpha((m_1, m_2), (n_1, n_2)) := \chi^{m_1 n_2} ,
  \ee
where \(\chi \colon \Z/n\Z \Rarr~ \ko^\times\) is a character such that 
\(\chi(1) \eq q\) is a primitive \(n\)th root of unity. The associated 3-cocycle is
given by
  \be
  \omega(\vec{x}, \vec{y}, \vec{z}) = q^{x_1 y_2 z_3}
  \ee
for \(\vec{x}, \vec{y}, \vec{z} \iN (\Z/n\Z)^{\oplus 3}\). The cohomology class of
\(\omega\) has order \(n\), and hence for \(n \,{>}\, 2\) the fusion category 
\(\vect^{\omega_\alpha}_{B \oplus A}\) is not tensor equivalent to its complex
conjugate. On the other hand, since \(\vect^{\omega_\alpha}_{B \oplus A}\) is pivotal 
Morita equivalent to \(\vect_{\widehat{B} \rtimes_{\alpha} A}\) and the latter 
underlying fusion category \emph{is} equal to its conjugate, 
\(\vect^{\omega_\alpha}_{B \oplus A}\) is pivotal Morita equivalent to its complex 
conjugate. Con\-sequently, the Turaev-Viro invariant for 
\(\vect^{\omega_\alpha}_{B \oplus A}\) takes values in the reals. More specifically, 
the invariant equals the Turaev-Viro invariant of \(\vect_{\widehat{B} 
\rtimes_{\alpha} A}\) which, having a trivial 3-cocycle, takes values in the rationals.
\end{exa}

%%%%%%%%%%%%%%%%%%%%%%%%%%%%%%%%%%%%%%%%%%%%%%%%%%%%%%%%%%%%%%%%%%%%%%%%

\vskip 4.5em

\noindent
{\sc Acknowledgments:}\\[.3em]
J.F.\ is supported by the Swedish Research Council VR under project no.\ 2022-02931.
C.G. was partially supported by Grant INV-2023-162-2830 from the School of Science of Universidad de los Andes.
D.J.\ is supported by The Research Council of Norway - project 324944.
C.S.\ is supported by the Deutsche Forschungsgemeinschaft (DFG, German Research
Foundation) by SFB 1624 and under Germany's Excellence Strategy - EXC 2121
``Quantum Universe'' - 390833306.

\newpage
%%%%%%%%%%%%%%%%%%%%%%%%%%%%%%%%%%%%%%%%%%%%%%%%%%%%%%%%%%%%%%%%%%%%%%%%
\appendix

\section{String diagrams in pivotal bicategories}\label{app_graph}

%%%%%%%%%%%%%%%%
\def\width{0.6}
\def\height{0.6}
%%%%%%%%%%%%%%%%

The graphical string calculus for morphisms in a pivotal bicategory is a 
useful tool for computations. A survey of this calculus can be found for instance 
in \cite{humm,hiMar} and in \Cite{Sect.\,2}{FuSY}. 
In this appendix we settle the conventions that are used in the present document.

For a bicategory $\mathscr{F}$ we employ the following graphical representations 
on a canvas that is a stratified rectangle in the plane, regarded as 2-framed 
(with vector fields parallel to the two coordinate axes):
\Enumerate%i
    \item 
Objects: $\calx\iN\mathscr{F}$ is represented as a label on a two-cell of the
canvas.
    \item 
For a $1$-morphism $H\iN\mathscr{F}(\calx,\caly)$ we draw a line labeled by $H$ between two-cells
labeled by $\calx$ and $\caly$. The diagram is to be read from right to left, i.e.\
the domain of $H$ is to the right of the line and the codomain is to its the left:
  \be
  \label{id_H}
  \raisebox{-4.7em}{ \begin{tikzpicture}[scale=\locscale]
  \draw[line width=0.5*\widthObj,gray,dashed] (-4*\locW,0) rectangle (5*\locW,\locH);
  \draw[line width=\widthObj,\colorObj]
       (0.5*\locW,0) node[below=1pt] {$H$} -- +(0,\locH) node[above=-1pt] {$H$};
  \node[\colorCat] at (-2*\locW,0.5*\locH) {$\caly$} ;
  \node[\colorCat] at (2.5*\locW,0.5*\locH) {$\calx$} ;
  \end{tikzpicture}}
  \ee
Here the dashed line indicate the canvas; in all diagrams below they will be 
suppressed.
  \item 
The identity $1$-morphism $\textbf{1}_\calx\iN\mathscr{F}(\calx,\calx)$ amounts to
a transparent, or invisible, line.
  \item 
A $2$-morphism $\alpha \colon H_1 \,{\xRightarrow{~~}}\, H_2$ is represented by a 
rectangular coupon inserted between the lines for the $1$-morphisms $H_1$ and $H_2$,
with the line for the domain of $\alpha$ attached to the bottom of the coupon and
the line for the codomain attached to its top. To simplify the graphical notation,
we sometimes replace the rectangular coupon by a labeled dot:
  \be
  \raisebox{-4.7em}{
    \begin{tikzpicture}[scale=\locscale]
  \draw[line width=\widthObj,\colorObj]
       (0.5*\locW,0) node[below=1pt] {$H_1$} -- +(0,\locH) node[above=-1pt] {$H_2$};
  \filldraw[line width=\widthMor,fill=white,draw=\colorMor]
       (-\locWM,0.5*\locH-0.5*\locHM) rectangle +(\locW+2*\locWM,\locHM) ;
  \node at (0.5*\locW,0.5*\locH) {$\alpha$} ;
  \node[\colorCat] at (-2*\locW,0.5*\locH) {$\caly$} ;
  \node[\colorCat] at (2.5*\locW,0.5*\locH) {$\calx$} ;  
  \end{tikzpicture}
  }
  \qquad\text{ or simply }\qquad
  \raisebox{-4.7em}{
    \begin{tikzpicture}[scale=\locscale]
  \draw[line width=\widthObj,\colorObj]
         (0.5*\locW,0) node[below=1pt] {$H_1$} -- +(0,\locH) node[above=-1pt] {$H_2$};
  \filldraw[black] (0.5*\locW,0.5*\locH) circle (2.5pt) node[anchor=east]{$\alpha\;$};
  \node[\colorCat] at (-2*\locW,0.5*\locH) {$\caly$} ;
  \node[\colorCat] at (2.5*\locW,0.5*\locH) {$\calx$} ;  
  \end{tikzpicture}
  }
  \ee
    \item 
The identity $2$-morphism $\id_H \colon H \,{\xRightarrow{~~}}\, H$ is represented 
by an invisible coupon, such that its graphical description is identical with 
\eqref{id_H}.
    \item 
Parallel lines labeled by $1$-morphisms $F\iN\mathscr{F}(\caly,\calz)$ and
$H\iN\mathscr{F}(\calx,\caly)$ evaluate to a single line labeled by the horizontal
composition $1$-morphism $F\cir H\iN\mathscr{F}(\calx,\calz)$.
For $2$-morphisms $\alpha \colon H_1 \,{\xRightarrow{~~}}\, H_2$,
$\beta\colon H_2 \,{\xRightarrow{~~}}\, H_3$ and $\gamma \colon F_1 \,{\xRightarrow{~~}}\, F_2$ 
in $\mathscr{F}$, the diagrams
  \be
  \raisebox{-4.7em}{ \begin{tikzpicture}[scale=\locscale]
  \draw[line width=\widthObj,\colorObj]
       (0.5*\locW,0) node[below=1pt] {$F_1$} -- +(0,\locH) node[above=-1pt] {$F_2$};
  \filldraw[line width=\widthMor,fill=white,draw=\colorMor]
       (-\locWM,0.5*\locH-0.5*\locHM) rectangle +(\locW+2*\locWM,\locHM) ;
  \node at (0.5*\locW,0.5*\locH) {$\gamma$} ;
  \node[\colorCat] at (-1.85*\locW,0.5*\locH) {$\calz$} ;
  \node[\colorCat] at (3*\locW,0.5*\locH) {$\caly$} ;
  \draw[line width=\widthObj,\colorObj]
       (0.5*\locW+5*\locW,0) node[below=1pt] {$H_1$} -- +(0,\locH)
       node[above=-1pt] {$H_2$};
  \filldraw[line width=\widthMor,fill=white,draw=\colorMor]
       (-\locWM+5*\locW,0.5*\locH-0.5*\locHM) rectangle +(\locW+2*\locWM,\locHM) ;
  \node at (0.5*\locW+5*\locW,0.5*\locH) {$\alpha$} ;
  \node[\colorCat] at (2.85*\locW+5*\locW,0.5*\locH) {$\calx$} ;
  \end{tikzpicture}
  }
  \hsp{2.1} \text{and} \hsp{2.1} 
  \raisebox{-4.7em}{        \begin{tikzpicture}[scale=\locscale]
  \draw[line width=\widthObj,\colorObj]
       (0.5*\locW,0) node[below=1pt] {$H_1$} -- +(0,\locH)
       node[above=-1pt] {$H_3$};
  \node[anchor=west] at (0.5*\locW,0.5*\locH) {$H_2$} ;
  \filldraw[line width=\widthMor,fill=white,draw=\colorMor]
       (-\locWM,0.25*\locH-0.5*\locHM) rectangle +(\locW+2*\locWM,\locHM) ;
  \node at (0.5*\locW,0.25*\locH) {$\alpha$} ;
  \filldraw[line width=\widthMor,fill=white,draw=\colorMor]
       (-\locWM,0.75*\locH-0.5*\locHM) rectangle +(\locW+2*\locWM,\locHM) ; 
  \node at (0.5*\locW,0.75*\locH) {$\beta$} ;
  \node[\colorCat] at (-1.85*\locW,0.5*\locH) {$\caly$} ;
  \node[\colorCat] at (3.35*\locW,0.5*\locH) {$\calx$} ;
  \end{tikzpicture}}
  \hsp{2.1}
  \ee
evaluate to the horizontal composition $2$-morphism 
$\gamma\cir\alpha \colon F_1\cir H_1 \,{\xRightarrow{~~}}\, F_2\cir H_2$ and to the
vertical composition $\beta\,{\cdot}\,\alpha \colon H_1 \,{\xRightarrow{~~}}\, H_3$,
respectively.
\end{enumerate}

\noindent 
Next recall from Section \ref{sec:module_duals} the notion of duals for a
$1$-morphism in a bicategory $\mathscr{F}$ (see also e.g.\ \Cite{App.\,A.3}{schaum2}).
We will use the following graphical description. 
   \\[-0.8em]
   \Enumerate \setcounter{enumi}{6}

\item 
A \emph{right dual} to a $1$-morphism $H\iN\mathscr{F}(\calx,\caly)$ consists 
of a $1$-morphism $H^\vee\iN\mathscr{F}(\caly,\calx)$ and $2$-morphisms 
${\rm ev}_H \colon H^\vee\cir H\,{\Ra}\, \mathbf{1}_\calx$ and 
${\rm coev}_H \colon \mathbf{1}_\caly\,{\Ra}\, H\cir H^\vee$, which
obey appropriate snake relations. We depict these $2$-morphisms by
  \be
  \raisebox{-2.6em}{ \begin{tikzpicture}[scale=\locscale]
  \draw[line width=\widthObj,\colorObj]
       (\locW,0*\locHs) node[below=-1pt] {$H^\vee$} -- +(0,\locHs) 
       (\locW+2*\locC,0*\locHs) node[below=-1pt] {$H$} -- +(0,\locHs) 
       (\locW,\locHs-0*\locHs) arc (180:0:\locC);
  \scopeArrow{0.5}{\arrowObj}
  \draw[line width=\widthObj,\colorObj,postaction={decorate}]
       (\locW-\hsA+\locC,\locHs-0*\locHs+\locC) -- + (-0.01,0) ;
  \end{scope}
   \node[\colorCat] at (\locW+\locC,0.5*\locHs) {$\caly$} ;
  \node[\colorCat] at (\locW+\locC,2*\locHs+0.7*\locC) {$\calx$} ;
  \end{tikzpicture}
  }
  \qquad\qquad\text{ and }\qquad\qquad
  \raisebox{-2.6em}{\begin{tikzpicture}[scale=\locscale]
  \draw[line width=\widthObj,\colorObj]
       (\locW,\locHs+\locC)  -- +(0,\locHs) node[above=1pt] {$H$}
       (\locW+2*\locC,\locHs+\locC)  -- +(0,\locHs) node[above=1pt] {$\,H^\vee$}
       (\locW,\locHs+\locC) arc (180:360:\locC);
  \scopeArrow{0.5}{\arrowObj}
  \draw[line width=\widthObj,\colorObj,postaction={decorate}]
       (\locW-\hsA+\locC,\locHs) -- + (-0.01,0) ;
  \end{scope}
  \node[\colorCat] at (\locW+\locC,\locC+1.25*\locHs) {$\calx$} ;
  \node[\colorCat] at (\locW+\locC,0.25*\locHs) {$\caly$} ;
  \end{tikzpicture}}
  \ee

\item 
Similarly, a \emph{left dual} to $H\iN\mathscr{F}(\calx,\caly)$ is a $1$-morphism
${}^\vee\! H\iN\mathscr{F}(\caly,\calx)$ together with $2$-mor\-phisms 
${\rm ev'}_{\!H} \colon H\cir {}^\vee\! H\,{\Ra}\, \mathbf{1}_\caly$ and 
${\rm coev'}_{\!H} \colon \mathbf{1}_\calx\,{\Ra}\, {}^\vee\! H\cir H$, obeying
snake relations. We portray them as
  \be
  \raisebox{-2.6em}{ \begin{tikzpicture}[scale=\locscale]
  \draw[line width=\widthObj,\colorObj]
       (\locW,0*\locHs) node[below=-1pt]  {$H$}-- +(0,\locHs) 
       (\locW+2*\locC,0*\locHs) node[below=-1pt] {${}^\vee\!H\,$} -- +(0,\locHs) 
       (\locW,\locHs-0*\locHs) arc (180:0:\locC);
  \scopeArrow{0.5}{\arrowObj}
  \draw[line width=\widthObj,\colorObj,postaction={decorate}]
       (\locW+\hsA+\locC,\locHs-0*\locHs+\locC) -- + (0.01,0) ;
  \end{scope}
  \node[\colorCat] at (\locW+\locC,0.5*\locHs) {$\calx$} ;
  \node[\colorCat] at (\locW+\locC,2*\locHs+0.7*\locC) {$\caly$} ;
  \end{tikzpicture}
  }
  \qquad\qquad\text{ and }\qquad\qquad
  \raisebox{-2.6em}{        \begin{tikzpicture}[scale=\locscale]
  \draw[line width=\widthObj,\colorObj]
       (\locW,\locHs+\locC)  -- +(0,\locHs) node[above=1pt] {${}^\vee\!H\,$}
       (\locW+2*\locC,\locHs+\locC)  -- +(0,\locHs) node[above=1pt] {$H$}
       (\locW,\locHs+\locC) arc (180:360:\locC);
  \scopeArrow{0.5}{\arrowObj}
  \draw[line width=\widthObj,\colorObj,postaction={decorate}]
       (\locW+\hsA+\locC,\locHs) -- + (0.01,0) ;
  \end{scope}
  \node[\colorCat] at (\locW+\locC,\locC+1.25*\locHs) {$\caly$} ;
  \node[\colorCat] at (\locW+\locC,0.25*\locHs) {$\calx$} ;
  \end{tikzpicture}
  }
  \ee
\end{enumerate}

\noindent
If the bicategory $\mathscr{F}$ admits dualities for every $1$-morphism,
it may further admit a \emph{pivotal structure}, i.e.\ a pseudo-natural equivalence 
  \be
  \mathbf{P} \Colon \id_\mathscr{F} \xRightarrow{\;\simeq~\,} (-)\dd
  \ee
obeying $\mathbf{P}_{\!\calx} \eq \id_\calx$ for every object $\calx\iN\mathscr{F}$.
   \\[-0.9em]
  \Enumerate%i
  \setcounter{enumi}{8}
  \item 
We depict the component $\mathbf{P}_H \colon H\Ra H\dd$ of a pivotal structure 
$\mathbf{P}$ at a $1$-morphism $H\iN\mathscr{F}(\calx,\caly)$ by
  \be
  \raisebox{-3.6em}{ \begin{tikzpicture}[scale=.7*\locscale]
  \draw[line width=\widthObj,\colorObj]
        (0.5*\locW,0) node[below=1pt] {$H$} -- +(0,\locH) node[above=-1pt] {$~H\dd$};
  \filldraw[line width=\widthMor,fill=white,draw=\colorMor]
        (0.5*\locW,0.5*\locH) circle (0.5*\locC) ;
  \node[\colorCat] at (-2*\locW,0.5*\locH) {$\caly$} ;
  \node[\colorCat] at (2.5*\locW,0.5*\locH) {$\calx$} ;  
  \end{tikzpicture}
  }
  \qquad\text{ or simply }\qquad
  \raisebox{-3.6em}{ \begin{tikzpicture}[scale=.7*\locscale]
  \draw[line width=\widthObj,\colorObj]
       (0.5*\locW,0) node[below=1pt] {$H$} -- +(0,\locH) node[above=-1pt] {$~H\dd$};
  \node[\colorCat] at (-2*\locW,0.5*\locH) {$\caly$} ;
  \node[\colorCat] at (2.5*\locW,0.5*\locH) {$\calx$} ;  
  \end{tikzpicture}
  }  
  \ee
in the case that the labeling of the lines is evident.
\end{enumerate}

%%%%%%%%%%%%%%%%%%%%%%%%%%%%%%%%%%%%%%%%%%%%%%%%%%%%%%%%%%%%%%%%%%%%%%%%

Let us look a bit closer into the graphical calculus for the bicategories that are
of prime interest to us. First, any tensor category $\cala$ can be regarded as a 
bicategory $\mathbf{\mathbbm{A}}$ with a single object, called the \emph{delooping} 
of $\cala$. The End-category of the unique object in $\mathbf{\mathbbm{A}}$ is the 
category underlying $\cala$, and the composition of $1$-morphisms $a,b\iN\cala$ is 
the tensor product, $a\cir b \,{:=}\, a\oti b$. A morphism $f \colon a\oti b\Rarr~ c$ 
in $\cala$ is graphically represented as the string diagram
  \be
\raisebox{-4.7em}{
  \begin{tikzpicture}[scale=0.9*\locscale]
  \draw[line width=\widthObj,\colorObj]
       (0,0) node[below=1pt] {$a$} -- +(0,0.5*\locH)
       (\locW,0) node[below=-2pt] {$b$} -- +(0,0.5*\locH)
       (0.5*\locW,0.5*\locH)  -- +(0,0.5*\locH) node[above=-1pt] {$c$} ;
  \filldraw[line width=\widthMor,fill=white,draw=\colorMor]
       (-\locWM,0.5*\locH-0.5*\locHM) rectangle +(\locW+2*\locWM,\locHM) ;
  \node at (0.5*\locW,0.5*\locH) {$f$} ;
  \end{tikzpicture}}
  \ee
in $\mathbf{\mathbbm{A}}$. Left and right duals in $\cala$ correspond to 
bicategorical left and right duals in $\mathbf{\mathbbm{A}}$. A pivotal structure 
on $\cala$ endows the bicategory $\mathbf{\mathbbm{A}}$ with a pivotal structure in
the bicategorical sense. The requirement that $\textbf{P}_{\!x} \eq \id_x$ imposed
on the pseudo-natural equivalence \eqref{pivotal_st_bicategory} ensures that every
pivotal structure on $\mathbf{\mathbbm{A}}$ comes from a pivotal structure 
on the tensor category $\cala$.

Another pivotal bicategory of interest to us is the bicategory 
$\textbf{Mod}^{{\rm piv}\!}(\cala)$ of pivotal modules over a pivotal tensor category
discussed in Section \ref{sec:PME}. In this bicategory the right dual of a module
functor is its left adjoint functor with its canonical module functor structure, 
and the left dual is the right adjoint module functor.

Finally recall that associated to an invertible bimodule category 
${}_\calc\call_\cald$ there is two-object bicategory $\Mor$ \Cite{Sect.\,3}{FGJS}.
We denote the two objects of $\Mor$ are by ``$+$'' and ``$-$''; their Hom-categories
are
  \be
  \begin{aligned}
  \Mor(\mathbf{+},\mathbf{+}) = \calc \,, \qquad
  \Mor(\mathbf{-},\mathbf{-}) & = \cald \simeq \Fun_\calc(\call,\call) \,,
  \Nxl2
  \Mor(\mathbf{-},\mathbf{+}) = \call \,,\qquad
  \Mor(\mathbf{+},\mathbf{-}) & =\Fun_\calc(\call,\calc)  \,.
  \end{aligned}
  \ee
The horizontal composition in $\Mor$ is given by the tensor products of $\calc$ and
$\cald$, by their module actions on $\call$ and $\Fun_\calc(\call,\calc)$, and by
the bimodule functors
  \be
  \begin{aligned}
  \bmixt{}{}\Colon & \call\Times\Fun_\calc(\call,\calc) \rarr~ \calc \,,
  & (x,H)\longmapsto H(x) \phantom{-\act\,.}
  \Nxl2
  \text{and} \qquad
  \bmixtd{}{}\Colon & \Fun_\calc(\call,\call)\Times\call \rarr~ \calc_\call^\vee
  \simeq \cald \,, \quad & (H,x)\longmapsto H(-)\act x \,.
  \end{aligned}
  \ee
Moreover, owing to \Cite{Thm.\,4.2}{FGJS}, $\Mor$ is a bicategory with dualities. As
a consequence, objects in $\call$ and the actions of $\calc$ and $\cald$ can be
graphically represented in terms of string diagrams. For example, given 
morphisms $f\iN\Hom_\call(x,y)$ and $g\iN\Hom_\cald(d_1,d_2)$, the diagram
  \be\label{inv_graph_example}
  \raisebox{-4.7em}{ \begin{tikzpicture}[scale=\locscale]
  \draw[line width=\widthObj,\colorObj]
       (0.5*\locW,0) node[below=1pt] {$x$} -- +(0,\locH-\locHs) node[right=1pt] {$y$}
       (-4*\locC+0.5*\locW,0) node[below=-4pt] {$\;y^{\!\vee}$} -- +(0,\locH-\locHs) 
       (0.5*\locW,\locH-\locHs) arc (0:180:2*\locC);
  \filldraw[line width=\widthMor,fill=white,draw=\colorMor]
       (-\locWM,0.5*\locH-0.5*\locHM) rectangle +(\locW+2*\locWM,\locHM) ;
  \node at (0.5*\locW,0.5*\locH) {$f$} ;
  \scopeArrow{0.5}{\arrowObj}
  \draw[line width=\widthObj,\colorObj,postaction={decorate}]
       (0.5*\locW-\hsA-2*\locC,\locH-\locHs+2*\locC) -- + (-0.01,0) ;
  \end{scope}
  \node[\colorCat] at (-1.5*\locW-4*\locC,0.5*\locH) {$\mathbf{-}$} ;
  \node[\colorCat] at (-1.5*\locW,0.5*\locH) {$\mathbf{+}$} ;
  \node[\colorCat] at (3*\locW,0.5*\locH) {$\mathbf{-}$} ;

  \draw[line width=\widthObj,\colorObj]
       (0.5*\locW+5*\locW,0) node[below=1pt] {$d_1$} -- +(0,\locH)
       node[above=-1pt] {$d_2$};
  \filldraw[line width=\widthMor,fill=white,draw=\colorMor]
       (-\locWM+5*\locW,0.5*\locH-0.5*\locHM) rectangle +(\locW+2*\locWM,\locHM) ;
  \node at (0.5*\locW+5*\locW,0.5*\locH) {$g$} ;
  \node[\colorCat] at (2.5*\locW+5*\locW,0.5*\locH) {$\mathbf{-}$} ;
  \end{tikzpicture}
  }
  \ee
then stands for the morphism $({\rm ev}_y\act \id_{d_2}) \cir (\bmixtd{\id_{y^{\!\vee}}}
{f\actr g}) \colon \bmixtd{y^\vee}{x\actr d_1} \Rarr~ \textbf{1}\otimes_\cald d_2$ 
in $\cald$.

We denote by $\Mor^{\textbf{1(op)}}$ the bicategory with the same objects as $\Mor$,
but for which the Hom-ca\-tegories are $\Mor^{\textbf{1(op)}}(x,y) \eq \Mor(y,x)$; 
thus as compared to $\Mor$, only the $1$-morphisms are reversed. The bicategory
$\Mor^{\textbf{1(op)}}$ can be embedded into the bicategory 
$\mathbf{Mod}^{{\rm ex}\!}(\calc)$ of exact $\calc$-mo\-du\-le categories, i.e.\
there is a fully faithful pseudo-functor
  \be
  \Mor^{\textbf{1(op)}}\longrightarrow\mathbf{Mod}^{{\rm ex}\!}(\calc)
  \label{Mor_embedding}
  \ee
that maps objects as $\mathbf{+} \,{\xmapsto{\phantom x\,}}\, \calc$ and 
$\mathbf{-}\,{\xmapsto{\phantom x\,}}\, \call$, and has equivalences at the level of 
Hom-categories given by the assignments
  \be
  \begin{aligned}
  \Mor^{\textbf{1(op)}}(\mathbf{+},\mathbf{+})=\calc
  & \xrightarrow{\,\simeq~} \Fun_\calc(\calc,\calc) \,, \qquad&
  \Mor^{\textbf{1(op)}}(\mathbf{+},\mathbf{-})=\call
  & \xrightarrow{\,\simeq~} \Fun_\calc(\calc,\call) \,,\\
  c&\xmapsto{\phantom x~~} {-}\oti c&  y&\xmapsto{\phantom x~~} -\act y
  \Nxl9
  \Mor^{\textbf{1(op)}}(\mathbf{-},\mathbf{-})=\cald
  & \xrightarrow{\,\simeq~} \Fun_\calc(\call,\call) \,, \qquad &
  \Mor^{\textbf{1(op)}}(\mathbf{-},\mathbf{+})=&\Fun_\calc(\call,\calc) \,.\\
  d& \xmapsto{\phantom x~~} -\actr d&   &
  \end{aligned}
  \ee
With these identifications one represents the morphism \eqref{inv_graph_example} 
by the diagram
  \be
  \raisebox{-4.7em}{ \begin{tikzpicture}[scale=\locscale]
  \draw[line width=\widthObj,\colorObj]
       (0.5*\locW,0) node[below=1pt] {$x$} -- +(0,\locH-\locHs) node[right=1pt] {$y$}
       (0.5*\locW-4*\locC,0) node[below=-4pt] {$\;y^{\!\vee}$} -- +(0,\locH-\locHs) 
       (0.5*\locW,\locH-\locHs) arc (0:180:2*\locC);
  \filldraw[line width=\widthMor,fill=white,draw=\colorMor]
       (-\locWM,0.5*\locH-0.5*\locHM) rectangle +(\locW+2*\locWM,\locHM) ;
  \node at (0.5*\locW,0.5*\locH) {$f$} ;
  \scopeArrow{0.5}{\arrowObj}
  \draw[line width=\widthObj,\colorObj,postaction={decorate}]
       (0.5*\locW-\hsA-2*\locC,\locH-\locHs+2*\locC) -- + (-0.01,0) ;
  \end{scope}
  \node[\colorCat] at (-1.5*\locW-4*\locC,0.5*\locH) {$\call$} ;
  \node[\colorCat] at (-1.5*\locW,0.5*\locH) {$\calc$} ;
  \node[\colorCat] at (3*\locW,0.5*\locH) {$\call$} ;
  \draw[line width=\widthObj,\colorObj]
       (0.5*\locW+5*\locW,0) node[below=1pt] {$d_1$} -- +(0,\locH)
       node[above=-1pt] {$d_2$};
  \filldraw[line width=\widthMor,fill=white,draw=\colorMor]
       (-\locWM+5*\locW,0.5*\locH-0.5*\locHM) rectangle +(\locW+2*\locWM,\locHM) ;
  \node at (0.5*\locW+5*\locW,0.5*\locH) {$g$} ;
  \node[\colorCat] at (2.5*\locW+5*\locW,0.5*\locH) {$\call$} ;
  \end{tikzpicture}
  }
  \ee
i.e.\ as a $2$-morphism in the bicategory $\mathbf{Mod}^{{\rm ex}\!}(\calc)$. 
Furthermore, in a pivotal setting the bicategory $\Mor$ inherits a pivotal structure
according to \Cite{Thm.\,5.9}{FGJS} and the pseudo-functor factors through a pivotal pseudo-functor 
  \be
  \Mor^{\textbf{1(op)}}\longrightarrow\mathbf{Mod}^{{\rm piv}\!}(\calc)\,.
  \ee

%%%%%%%%%%%%%%%%%%%%%%%%%%%%%%%%%%%%%%%%%%%%%%%%%%%%%%%%%%%%%%%%%%%%%%%%
\newpage
%%%%%%%%%%%%%%%%%%%%%%%%%%%%%%%%%%%%%%%%%%%%%%%%%%%%%%%%%%%%%%%%%%%%%%%%

\newcommand\wb{\,\linebreak[0]} \def\wB {$\,$\wb}
\newcommand\Arep[2]  {{\em #2}, available at {\tt #1}}
\newcommand\Bi[2]    {\bibitem[#2]{#1}}
\newcommand\inBO[9]  {{\em #9}, in:\ {\em #1}, {#2}\ ({#3}, {#4} {#5}), p.\ {#6--#7} {\tt [#8]}}
\newcommand\J[7]     {{\em #7}, {#1} {#2} ({#3}) {#4--#5} {{\tt [#6]}}}
\newcommand\JO[6]    {{\em #6}, {#1} {#2} ({#3}) {#4--#5} }
\newcommand\JP[7]    {{\em #7}, {#1} ({#3}) {{\tt [#6]}}}
\newcommand\Jpress[7]{{\em #7}, {#1} {} (in press) {} {{\tt [#6]}}}
\newcommand\Jtoa[7]  {{\em #7}, {#1} {} (to appear) {} {{\tt [#6]}}}
\newcommand\K[6]       {\ {\sl #6}, {#1} {#2} ({#3}) {#4}}
\renewcommand\K[6]     {\ {\sl #6}, {#1} {#2} ({#3}) {#4} {\small\tt [#5]}}
\newcommand\BOOK[4]  {{\em #1\/} ({#2}, {#3} {#4})}
\newcommand\PhD[2]   {{\em #2}, Ph.D.\ thesis #1}
\newcommand\Prep[2]  {{\em #2}, preprint {\tt #1}}
\newcommand\uPrep[2] {{\em #2}, unpublished preprint {\tt #1}}
\def\aagt  {Alg.\wB\&\wB Geom.\wb Topol.}   
\def\alrt  {Algebr.\wb Represent.\wB Theory}   
\def\anma  {Ann.\wb Math.}
\def\inma  {Invent.\wb Math.}
\def\imrn  {Int.\wb Math.\wb Res.\wb Notices}
\def\jhrs  {J.\wb Ho\-mo\-to\-py\wb Relat.\wb Struct.}
\def\joal  {J.\wB Al\-ge\-bra}
\def\jpaa  {J.\wB Pure\wB Appl.\wb Alg.}
\def\kyjm  {Ky\-o\-to J.\ Math.}
\def\mams  {Memoirs\wB Amer.\wb Math.\wb Soc.}
\def\nupb  {Nucl.\wb Phys.\ B}
\def\pajm  {Pa\-cific\wB J.\wb Math.}
\def\quto  {Quantum Topology}
\def\tams  {Trans.\wb Amer.\wb Math.\wb Soc.}
\def\trgr  {Trans\-form.\wB Groups}
\def\lamp  {Journal of Logical and Algebraic Methods in Programming}

\small
% \addcontentsline{toc}{chapter}{Bibliography}

%%%%%%%%%%%%%%%%%%%%%%%%%%%%%%%%%%%%%%%%%%%%%%%%%%%%%%%%%%%%%%%%%%%%%%%%
 \end{document}